\documentclass[reqno,11pt,a4paper]{amsart}

\usepackage[left=3cm,right=3cm,top=3cm,bottom=3cm]{geometry}
\usepackage[utf8]{inputenc}

\usepackage{dsfont}
\usepackage{nicefrac}
\usepackage{graphicx}
\usepackage{caption}
\usepackage[subrefformat=parens, labelfont=up]{subcaption}
\usepackage{enumerate}
\usepackage{array}
\usepackage{hyperref}
\usepackage{amsfonts,amsmath,amssymb,amsthm,mathtools,mathrsfs}
\usepackage{xcolor}
\usepackage{microtype}
\usepackage{multirow}
\usepackage{booktabs}
\usepackage[normalem]{ulem}

\def\bb#1\eb{\textcolor{blue}{#1}}
\def\br#1\er{\textcolor{red}{#1}}
\def\bv#1\ev{\textcolor{green}{#1}}
\def\bl#1\el{\textcolor{violet}{#1}}

\newtheorem{thm}{Theorem}[section]
\newtheorem{prop}[thm]{Proposition}
\newtheorem{lemma}[thm]{Lemma}
\newtheorem{cor}[thm]{Corollary}
\theoremstyle{definition}
\newtheorem{defi}[thm]{Definition}
\newtheorem{notation}[thm]{Notation}
\newtheorem{exe}[thm]{Example}

\newtheorem{rem}[thm]{Remark}

\newcommand{\C}{\mathcal{C}}
\newcommand{\T}{\mathcal{T}}
\newcommand{\leftindex}{{\textcolor{white} ,}}

\title[Fermat's principle and Snell's law for cone structures]{Generalized Fermat's principle and Snell's law \\ for cone structures and applications}

\author[M.~\'{A}.~Javaloyes]{Miguel \'Angel Javaloyes}\address{Departamento de Matem\'{a}ticas, \hfill\break\indent Universidad de Murcia, \hfill\break\indent Campus de Espinardo,\hfill\break\indent 30100 Espinardo, Murcia, Spain.} \email{majava@um.es}

\author[S.~Markvorsen]{Steen Markvorsen}\address{DTU Compute, \hfill\break\indent Section for Mathematics, \hfill\break\indent Technical University of Denmark, \hfill\break\indent DK-2800 Kgs. Lyngby, Denmark.}\email{stema@dtu.dk}

\author[E.~Pend\'{a}s-Recondo]{Enrique Pend\'{a}s-Recondo}\address{Departamento de Matem\'{a}ticas, \hfill\break\indent Facultad de Ciencias, \hfill\break\indent Universidad de Oviedo, \hfill\break\indent 33007 Oviedo, Spain.}\email{pendasenrique@uniovi.es}

\author[M.~S\'{a}nchez]{Miguel S\'{a}nchez}\address{Departamento de Geometr\'{\i}a y Topolog\'{\i}a, \hfill\break\indent Facultad de Ciencias \& IMAG (Centro de Excelencia Mar\'{i}a de Maeztu), \hfill\break\indent Universidad de Granada, \hfill\break\indent 18071 Granada, Spain.}\email{sanchezm@ugr.es}

\begin{document}

\begin{abstract}
Fermat's principle is fully generalized to the case where a smooth interface separates two cone structures---Lorentz-Finsler lightcones---representing wave propagation in a potentially inhomogeneous, anisotropic, time-dependent and discontinuous medium. The interface, wave source and receiver are assumed to be a hypersurface, a submanifold and a curve in the spacetime, respectively, of any causal character. For a trajectory to fulfill Fermat's principle---i.e., to be a critical point of the arrival time functional---its direction must change at the interface, obeying a precise condition that generalizes Snell's law of refraction when the wave crosses the interface, or the law of reflection when it remains in a single medium. Both laws are analyzed in detail to establish the conditions ensuring the existence and uniqueness of refracted and reflected trajectories, and to determine whether they actually minimize the arrival time. Applications to Zermelo's navigation problem and the determination of geodesics in discretized spacetimes are also emphasized.
\vspace{10mm}

\end{abstract}

\maketitle

\tableofcontents

\section{Introduction}
\label{sec:intro}
This work is devoted to the study, under very general assumptions, of two fundamental physical principles that are naturally intertwined: Fermat's principle and Snell's law of refraction. The generality of our framework is expected to support a discretization approach to the problem of wave propagation, based on the cones determined by the maximum propagation speed at each point and direction. This applies both to the propagation of non-relativistic waves in arbitrary (inhomogeneous, anisotropic and time-dependent) media and to the discretization of relativistic spacetimes, where the generalized Snell’s formula makes it possible to track lightlike geodesics.

\subsection{Classical Fermat's principle and its generalizations}
The classical Fermat's principle states that the path taken by a light ray between two fixed points in space is the one that makes the traveltime critical. This follows from the natural ability of waves---according to Huygens' principle---to ``test'' alternative paths: a wavefront spreading from a given point sweeps all possible wave paths originating from that point. Thus, this variational principle is inherently general---rooted in the very nature of waves---but requires a precise mathematical formulation in each specific context to determine these ``critical paths'', which can then be taken as the actual light trajectories in that setting.

When the medium of propagation is inhomogeneous and time-dependent (but infinitesimally isotropic), the natural setting to work in is a Lorentzian manifold $ (Q,g) $, even in the classical (Newtonian) case. Here, the variational problem takes the following form, first proposed in \cite{K90}: among all possible light paths (lightlike curves) from an event (point) $ p $ to an observer (timelike curve) $ \alpha $, we seek those which make the arrival time critical, where this ``arrival time'' is measured by an arbitrary parametrization of $ \alpha $. These critical points turn out to be the lightlike pregeodesics of $ g $ (see \cite{P90,P90b}).

Two main extensions of this relativistic formulation of Fermat’s principle have been explored. The first natural generalization involves relaxing some of the constraints on the initial point (source) and arrival curve (receiver):
\begin{itemize}
\item The case where the source is a spacelike submanifold $ P $ and the receiver is a timelike submanifold $ B $ (both of arbitrary codimension) is addressed in \cite{PP98}. In a natural way, $ P $ imposes the condition that lightlike geodesics must be orthogonal to it in order to remain critical. However, the introduction of $ B $ is more subtle, as it requires redefining the arrival time functional. In \cite{PP98}, this is accomplished by introducing a temporal function $ t_B $ on $ B $. Consequently, upon arrival, critical lightlike geodesics must project onto $ B $ in the direction of the gradient of $ t_B $.

\item The case where the arrival curve $ \alpha $ can have any causal character is studied in \cite[\S~7]{CJS14}. Interestingly, Fermat's principle remains the same here as in the case with a timelike $ \alpha $, provided that we consider only lightlike curves non-orthogonal to $ \alpha $.
\end{itemize}

The other natural extension is to consider an anisotropic medium of propagation. In this case, the Lorentzian framework must be generalized to include anisotropic lightcones, allowing the speed of light to depend on the spatial direction. This leads to the notion of Lorentz-Finsler metrics, Finsler spacetimes and, more fundamentally, cone structures, which focus solely on the geometry of lightcones without requiring an explicit metric (see \cite{JS20}). The extension of Fermat's principle to Finsler spacetimes was first studied in \cite{P06} (for light rays traveling from $ p $ to a timelike $ \alpha $) and, formally, it turns out to be the same as in the Lorentzian setting: the critical points of the arrival time functional are precisely the lightlike pregeodesics of the corresponding Lorentz-Finsler metric.

\subsection{Classical Snell's law and its generalizations}
Snell's law follows naturally from Fermat's principle when light---or any other wave---travels between two different media.\footnote{In fact, Fermat's principle was originally proposed as an explanation of the refraction phenomenon (see \cite[\S~2]{mahoney1994}).} In its classical formulation, one has two homogeneous, isotropic and time-independent media in the Euclidean space, separated by an interface. Due to the discontinuity in the speed of light, a light ray---otherwise a straight line within each medium under these conditions---must change direction at the interface in order to minimize the traveltime between the endpoints. This ``break'' is described by the classical Snell formula, which, in dimension 2, relates the angles of incidence and refraction to the refractive indices of the media (the ratio of the speed of light in vacuum to its speed in the medium).

Unlike Fermat's principle, Snell's law has always been studied from a purely spatial perspective (see e.g. \cite{GG03}). An interesting generalization arises when one considers (in the 2-dimensional Euclidean space) an anisotropic medium that is homogeneous in one direction. This symmetry provides---according to Noether's theorem---a constant of motion along the ray paths, which can be used as an equivalent to Snell's law. This case has been particularly studied in the context of seismic waves (see \cite{BS02,SW99}).

More recently, in \cite{MP23}, Snell's law and the analogous law of reflection have been fully generalized to arbitrary inhomogeneous and anisotropic (static) media---allowing for an arbitrary background manifold of any dimension and any shape of the (smooth) interface. The key idea is that the wave velocity in space induces a Finsler metric, with respect to which the length of a curve coincides with the time required to travel it at the given wave velocity. This leads to yet another generalization of Fermat's principle: a spatial curve traveling between two Finsler manifolds is a critical point of the arrival time functional (where this ``arrival time'' now is measured as a Finsler length) if and only if it is a Finsler pregeodesic in each media and satisfies Snell's law at the interface.

However, to the best of our knowledge, Snell’s law has never been explored from a spacetime viewpoint, where a possible time dependence---either in the wave velocity or in the interface itself---and a geometric evolutionary framework can be obtained and applied to both the Newtonian and relativistic settings.

\subsection{Our approach and main results}
We aim to unify all previous generalizations within the most general framework (described in \S~\ref{sec:setting}): we consider an arbitrary background manifold $ Q $ of any dimension and a smooth interface $ \eta $ of any causal character, which divides the manifold into two regions $ Q_1,Q_2 \subset Q $. These regions are each endowed solely with a cone structure $ \C^1,\C^2 $, which do not necessarily match at $ \eta $. We will show that cone structures provide the necessary notions of causality, geodesics and orthogonality to derive all the results, although for explicit computations, it will also be useful to work with compatible Lorentz-Finsler metrics $ L_1,L_2 $. We consider a submanifold $ P \subset Q_1 $ of arbitrary codimension as the source, and an embedded curve $ \alpha $ in $ Q_2 $ as the receiver, whose parametrization measures the ``arrival time''. Both $ P $ and $ \alpha $ can have, a priori, any causal character, although we will find certain restrictions on them a posteriori. Moreover, the extension of $ \alpha $ to a submanifold $ B $ (generalizing \cite{PP98}) can also be incorporated within our general framework (see Appendix~\ref{app:receiver}).

The main result (Theorem~\ref{thm:snell}) fully generalizes Fermat's principle and Snell's law to this setting. It states that, under some restrictions on the test curves (see Remark~\ref{rem:restrictions}), a lightlike curve $ \gamma $ is a critical point of the arrival time functional if and only if it is orthogonal to $ P $, it is a cone geodesic (i.e., a locally horismotic curve) in each cone structure, and it satisfies the generalized Snell's law at the interface (Equation~\eqref{eq:snell}). In this generalized version, the law admits the following geometric interpretation:
\begin{quote}
Snell's law states that, in order to be a critical point of the arrival time functional, the incident direction $ \dot{\gamma}(\tau^-) $ and the refracted direction $ \dot{\gamma}(\tau^+) $ of a trajectory $ \gamma $ crossing the interface $\eta$ at $ \gamma(\tau) $ must satisfy
\begin{equation*}
\dot{\gamma}(\tau^-)^{\perp_{\C^1}} \cap T_{\gamma(\tau)}\eta \subset \dot{\gamma}(\tau^+)^{\perp_{\C^2}} \cap T_{\gamma(\tau)}\eta,
\end{equation*}
where $ \dot{\gamma}(\tau^-)^{\perp_{\C^1}} $ and $ \dot{\gamma}(\tau^+)^{\perp_{\C^2}} $ are the hyperplanes tangent to the cone structures $ \C^1 $ and $ \C^2 $ along the incident and refracted directions, respectively, at $ \gamma(\tau) \in \eta $.
\end{quote}
Snell's formula applies whenever it is well defined, that is, when both an incident and a refracted trajectory actually exist. It can also be explicitly expressed in coordinates (Equation~\eqref{eq:snell_coord}) and, as expected, it reduces to the Finslerian version found in \cite{MP23} in the time-independent case (Equation~\eqref{eq:snell_natural}) and to the classical formula in the isotropic Euclidean space (see Example~\ref{ex:classical_laws}). The generalized law of reflection (Corollary~\ref{thm:reflection}) is also obtained as a direct consequence of Fermat's principle when the wave is forced to return to the first medium after hitting the interface---in fact, it will be seen as a particular case of Snell's law. Naturally, the relativistic and Finslerian versions of Fermat's principle are also recovered from this general formulation when both cone structures match at the interface. In addition, Appendix~\ref{app:receiver} extends Theorem~\ref{thm:snell} to the case where the receiver is not a curve but rather a submanifold of $ Q_2 $ (of any causal character) endowed with a submersion $ t_B: B \rightarrow \mathds{R} $ that serves as a measure of the ``arrival time'' (see Theorem~\ref{thm:app}).

Once Snell's law is established, the next goal is to study in detail the existence and uniqueness of its solutions. Since this law only operates at the level of the tangent space $ T_pQ $, where $ p \in \eta $ is the point where $ \gamma $ meets the interface, we can disregard the global trajectory, fix an incident direction and focus on the refracted directions we can obtain. Theorem~\ref{thm:existence_refraction} lists all possible scenarios, depending on the incident direction and the causal character of $ T_p\eta $ with respect to $ \C^2 $. The analogous result for the law of reflection is provided in Corollary~\ref{thm:existence_reflection}. In most natural situations, there is a unique refraction and reflection. However, this general framework allows for every other combination: reflection without refraction, refraction without reflection, cases where neither exists, or even situations where the refracted direction is not uniquely determined. Some of these cases lead to particularly interesting phenomena, which are studied separately: the total reflection (see \S~\ref{subsec:total_reflection}) and the double refraction (see \S~\ref{subsec:double_refraction}).

Our final study focuses on identifying which solutions to Fermat’s principle are actual minimizers of the arrival time. This is captured in the notion of Snell cone geodesics, which extends the local horismoticity property of cone geodesics beyond the interface---namely, they can be regarded as locally time-minimizing curves among causal curves arriving at any (future-directed) timelike $\alpha$. In order to identify these curves, a detailed analysis is required, relying on a specific orientation that can be induced on the hyperplanes orthogonal to the incident, refracted and reflected directions. Based on whether these orientations match, Theorem~\ref{thm:orientation} characterizes when a refracted or reflected trajectory yields a Snell cone geodesic. This is a subtle question, because Fermat's principle and Snell's law (or the law of reflection) combine to determine the possible horismoticity of the curve. Consequently, the spacetime viewpoint is essential here, even when applied to the minimization of a Newtonian problem. In line with the classical intuition, reflected trajectories are not Snell cone geodesics unless they remain unbroken at the interface (Corollary~\ref{cor:cone_reflected}). In contrast, for refracted trajectories, both outcomes are possible, depending on the arrangement of the cones (as Figure~\ref{fig:broken_hyperplane} explicitly shows). This becomes particularly relevant in the double refraction phenomenon, where only one of the two refractions qualifies as a Snell cone geodesic (Proposition~\ref{prop:double_refraction}).

\subsection{Prospective applications}
The applications of the generalized Fermat’s principle and Snell’s law extend far beyond the determination of light rays in a (potentially inhomogeneous, anisotropic, time-dependent and discontinuous) optical medium or gravitational field. In fact, cone structures can describe the propagation of any wave or physical phenomenon governed by Huygens' principle---e.g., sound waves, seismic waves or the spread of wildfires (see \cite{JPS21})---as well as the travel of a moving object, such as in Zermelo’s navigation problem (see \cite[\S~6.3]{JS20}). In \S~\ref{sec:natural}, we present the natural framework for modeling these situations, along with the explicit form that Snell's law (Corollary~\ref{cor:snell_physical}) and the law of reflection (Corollary~\ref{cor:reflection_physical}) take in this setting. More concretely, these laws can be expressed in terms of the (Finslerian) angles of the incident, refracted and reflected directions with respect to the interface (Equations~\eqref{eq:snell_angles} and \eqref{eq:reflection_angles}). In the isotropic case, these directions and the Euclidean normal to the interface lie in the same plane (see Example~\ref{ex:classical_laws}), allowing us to recover the classical formulas in terms of the (Euclidean) angles to the normal (Equation~\eqref{eq:classical}).

Within this context, it is also worth pointing out the relevance of considering strongly convex cones more general than the (quadratic) relativistic ones. Indeed, the aforementioned modeling of anisotropic wave propagation in a classical setting naturally leads to the most general class of convex cones. In fact, if the pointwise indicatrix (the topological sphere of maximum velocities in each direction) were not convex, its convex hull would yield the same wavefronts (see \cite[Appendix~A.2]{JPS23}). Our approach, based on cones associated with a Lorentz-Finsler metric, imposes only the mild assumption of {\em strong convexity}, which ensures that the Chern anisotropic connection is well defined and that cone geodesics are unique in every direction (and smoothly dependent on the initial conditions). As a consequence, these (strongly convex) anisotropic cones, beyond being the natural objects in Finslerian extensions of relativity, would also be emerging objects in fundamental theories of the spacetime---particularly in approaches derived from quantum gravity, such as causal fermion systems (see~\cite{FK23}) or causal sets (see~\cite{S19}).

Additionally, the comprehensive study we will conduct on Snell's law, considering every possible scenario, can be directly applied to the problem of computing lightlike geodesics in discretized spacetimes (a natural issue in numerical relativity), as briefly described in \S~\ref{sec:discretization}. For our purposes, this discretization can be implemented by selecting a grid that partitions the background manifold into cells. A cone is assigned to each cell---so that the cone structure remains constant within cells---while the cell boundaries serve as interfaces that may have, in general, any causal character. Then, the problem of finding lightlike geodesics in this discretized spacetime leads directly to the general Snell problem of finding refracted trajectories, fully explored in this work.

\subsection{Outline of the paper}
This work is structured in the following sections:
\begin{itemize}
\item \S~\ref{sec:preliminaries}: We begin by introducing the notion of cone structures, along with the causality and cone geodesics they induce (\S~\ref{subsec:cones}), as well as their links to Lorentz-Finsler metrics and Finsler spacetimes (\S~\ref{subsec:finsler_metrics}). We also emphasize the concepts of orthogonality and transversality (\S~\ref{subsec:orthogonality}), and the definition of Lorentz-Finsler geodesics through the Chern anisotropic connection, highlighting their equivalence with cone geodesics (\S~\ref{subsec:chern}).
\item \S~\ref{sec:setting}: We establish the general framework of our study, along with the definition of the arrival time functional and the admissible variations we will consider.
\item \S~\ref{sec:fermat}: We develop the generalized version of Fermat's principle (Theorem~\ref{thm:snell}), which includes the condition that extends Snell's law of refraction (\S~\ref{sec:snell}) and the analogous law of reflection (\S~\ref{sec:reflection}).
\item \S~\ref{sec:existence}: After some preliminaries (\S~\ref{subsec:def_res}), we analyze Snell's law and the law of reflection in detail to determine under which conditions we can ensure the existence and uniqueness of refraction (\S~\ref{sec:refracted_dir}) and reflection (\S~\ref{sec:reflected_dir}). The particular case of the total reflection phenomenon is also discussed (\S~\ref{subsec:total_reflection}).
\item \S~\ref{sec:cone_geodesics}: We first generalize some causality notions (\S~\ref{subsec:gen_causality}) and then introduce a method, based on a specific orientation induced on the hyperplanes appearing in Snell’s law and the law of reflection, that allows us to determine whether a refracted or reflected trajectory yields a Snell cone geodesic, i.e. a (local) minimizer of the arrival time functional (\S~\ref{subsec:orientation}). This, in turn, makes it possible to compute causal futures in the general setting by taking Snell cone geodesics as their boundaries. Furthermore, it provides a clear interpretation of the double refraction phenomenon (\S~\ref{subsec:double_refraction}).
\item \S~\ref{sec:applications}: Finally, we apply the developed theory to a natural physical setting, deriving simplified versions of Snell's law and the law of reflection that reduce to the classical ones in the isotropic case (\S~\ref{sec:natural}). We then briefly outline an application to the study of discretized spacetimes (\S~\ref{sec:discretization}).
\item Appendix~\ref{app:receiver}: To complete our discussion, we generalize the results in Section~\ref{sec:fermat} to obtain a version of Fermat's principle compatible with the setting where the receiver is a submanifold endowed with a submersion (Theorem~\ref{thm:app}).
\end{itemize}

\section{Preliminaries on cone structures}
\label{sec:preliminaries}
In order to make this work as self-contained as possible, we summarize in this section the main definitions and notions regarding cone structures and their relation to Lorentz-Finsler metrics, following mainly \cite{JS20}. Throughout this work, $ Q $ will denote a smooth (i.e., $ C^{\infty} $) manifold of dimension $ n+1 \geq 3 $, and $ TQ $ will be its tangent bundle. Also, every hypersurface or submanifold will be assumed smooth and embedded.

\subsection{Cone structures and causality}
\label{subsec:cones}
We start with the definition of cone structures, followed by the causality notions they induce.
\begin{defi}
\label{def:cone_structure}
A hypersurface $ \C $ of the slit tangent bundle $ TQ \setminus \textbf{0} $ is a {\em cone structure} on $ Q $ if for each $ p \in Q $:
\begin{itemize}
\item[(i)] $ \C $ is transverse to the fibers of the tangent bundle, i.e. if $ v \in \C_p \coloneqq T_pQ \cap \C $, then $ T_v(T_pQ) + T_{(p,v)}\C = T_{(p,v)}(TQ) $.
\item[(ii)] $ \C_p $ is a {\em cone}, meaning that it satisfies the following properties:
\begin{itemize}
\item[$ \circ $] {\em Conic}: if $ v \in \C_p $, then $ \lbrace \lambda v: \lambda > 0 \rbrace \subset \C_p $.
\item[$ \circ $] {\em Salient}: if $ v \in \C_p $, then $ -v \notin \C_p $.
\item[$ \circ $] {\em Convex interior}: $ \C_p $ is the boundary in $ T_pQ \setminus \lbrace 0 \rbrace $ of an open subset $ A_p \subset T_pQ \setminus \lbrace 0 \rbrace $ (the $ \C_p $-interior) which is convex, i.e. for any $ v, w \in A_p $, the segment $ \lbrace \lambda v + (1-\lambda)w: 0 \leq \lambda \leq 1 \rbrace $ is included entirely in $ A_p $.
\item[$ \circ $] {\em (Non-radial) strong convexity}: the second fundamental form of $ \C_p $ as an affine hypersurface of $ T_pQ $ is positive semi-definite (with respect to an inner direction pointing to $ A_p $) and its radical at each $ v \in \C_p $ is spanned by the radial direction $ \lbrace \lambda v: \lambda > 0 \rbrace $.
\end{itemize}
\end{itemize}
Given a cone structure $ \C $ on $ Q $, we denote by $ A \coloneqq \cup_{p \in Q} A_p $ its {\em cone domain}.
\end{defi}

The transversality condition (i) is necessary to ensure that the cone $ \C_p $ is a hypersurface of $ T_pQ \setminus \{0\} $ varying smoothly with $ p \in Q $ (see \cite[Remark~2.8]{JS20}). On the other hand, the hypotheses in (ii) are enough to ensure the existence of a hyperplane $ H $ in each $ T_pQ $ intersecting transversely all the directions in $ \C_p $, so that $ H \cap \C_p \subset T_pQ $ is a compact strongly convex submanifold of dimension $ (n-1) $ (see \cite[Lemma~2.5]{JS20}).

\begin{defi}
\label{def:causality}
Let $ \C $ be a cone structure on $ Q $ with cone domain $ A $. We will denote by $ \overline{A} $ the closure of $ A $ in $ TQ \setminus \textbf{0} $, and $ \overline{A}_p \coloneqq \overline{A} \cap T_pQ $.
\begin{itemize}
\item[(i)] A vector $ v \in T_pQ $ is
\begin{itemize}
\item[$ \circ $] {\em timelike} if $ v \in A_p $,
\item[$ \circ $] {\em lightlike} if $ v \in \C_p $,
\item[$ \circ $] {\em causal} if $ v \in \overline{A}_p $, and
\item[$ \circ $] {\em spacelike} if neither $ v $ nor $ -v $ is causal.
\end{itemize}
Note that $ \lambda v $ keeps the same causal character as $ v $, for any $ \lambda > 0 $.
\item[(ii)] A vector field $ V $ on $ Q $ is {\em timelike}, {\em lightlike}, {\em causal} or {\em spacelike} if $ V_p $ is so for all $ p \in Q $.
\item[(iii)] A piecewise smooth curve $ \gamma: I \rightarrow Q $ is {\em timelike}, {\em lightlike}, {\em causal} or {\em spacelike} if its tangent vector field $ \dot{\gamma} $ (or both $ \dot{\gamma}(t_0^+) $ and $ \dot{\gamma}(t_0^-) $ at any break $ t_0 \in I $) is so everywhere.
\item[(iv)] A linear subspace $ \Pi \subset T_pQ $ is
\begin{itemize}
\item[$ \circ $] {\em timelike} if there is at least one timelike vector $ v \in \Pi $,
\item[$ \circ $] {\em lightlike} if $ \Pi $ is tangent to the cone $ \C_p $ along one radial direction $ v \in \C_p $ (and only one, due to the non-radial strong convexity of the cone),
\item[$ \circ $] {\em causal} if it is timelike or lightlike, and
\item[$ \circ $] {\em spacelike} if every vector $ v \in \Pi $ is spacelike.
\end{itemize}
\item[(v)] A submanifold $ P \subset Q $ is {\em timelike}, {\em lightlike}, {\em causal} or {\em spacelike} if every tangent space $ T_pP \subset T_pQ $ is so.
\end{itemize}
When working with more than one cone structure, we will say that something is $ \C $-timelike, $ \C $-lightlike, $ \C $-causal or $ \C $-spacelike, making explicit reference to the cone structure we are using in each case. In addition, we define the following usual notions:
\begin{itemize}
\item {\em Chronological future}: $ I^+(p) \coloneqq \{q \in Q: \exists \text{ timelike curve from } p \text{ to } q \} $.
\item {\em Chronological past}: $ I^-(p) \coloneqq \{q \in Q: \exists \text{ timelike curve from } q \text{ to } p \} $.
\item {\em Causal future}: $ J^+(p) \coloneqq \{q \in Q: p = q \text{ or } \exists \text{ causal curve from } p \text{ to } q \} $.
\item {\em Causal past}: $ J^-(p) \coloneqq \{q \in Q: p = q \text{ or } \exists \text{ causal curve from } q \text{ to } p \} $.
\item {\em Horismotic relation}: Two points $ p,q \in Q $ are {\em horismotically related}, denoted $ p \rightarrow q $, if $ q \in J^+(p) \setminus I^+(p) $. When $ \Omega $ is an open subset of $ Q $, we denote by $ \rightarrow_{\Omega} $ the horismotic relation of the restricted cone structure $ \C|_\Omega $ on $ \Omega $. This means, in particular, that the chronological and causal futures and pasts must be computed using curves entirely contained in $ \Omega $; these sets will be denoted $ I^{\pm}_{\Omega}(p), J^{\pm}_{\Omega}(p) $.
\end{itemize}
Finally, a smooth function $ t: Q \rightarrow \mathds{R} $ is called {\em temporal} if $ dt(v) > 0 $ for all causal vectors. If such a function exists, we say that $ \C $ is {\em stably causal}.
\end{defi}

Cone structures also admit the following notion of geodesic.

\begin{defi}
\label{def:cone_geodesic}
Let $ \C $ be a cone structure on $ Q $. A continuous curve $ \gamma: I \rightarrow Q $ is a {\em cone geodesic} of $ \C $ if it is locally horismotic, i.e. for each $ t_0 \in I $ there exists an open neighborhood $ \Omega $ of $ \gamma(t_0) $ and $ \varepsilon > 0 $ such that $ I_{\varepsilon} \coloneqq [t_0-\varepsilon,t_0+\varepsilon] \cap I $ satisfies $ \gamma(I_{\varepsilon}) \subset \Omega $ and
\begin{equation*}
t_1 < t_2 \Leftrightarrow \gamma(t_1) \rightarrow_\Omega \gamma(t_2), \quad \forall t_1,t_2 \in I_{\varepsilon}.
\end{equation*}
\end{defi}

\subsection{Lorentz-Finsler metrics}
\label{subsec:finsler_metrics}
Even though cone structures will set the underlying framework upon which everything else will be built, to make explicit calculations it is usually easier to work with an auxiliary metric that shares the same lightlike curves. This subsection introduces both Finsler and Lorentz-Finsler metrics, relating the latter with cone structures.

\begin{defi}
A {\em Finsler metric} on $ Q $ is a continuous function $ F: TQ \rightarrow [0,\infty) $ satisfying the following properties:
\begin{itemize}
\item $ F $ is smooth and positive on $ TQ \setminus \textbf{0} $.
\item $ F $ is positive homogeneous of degree 1: $ F(\lambda v) = \lambda F(v) $ for all $ \lambda > 0 $.
\item For every $ v \in T_pQ \setminus \{0\} $, $ p \in Q $, the {\em fundamental tensor} $ g^F_v $, defined as
\begin{equation}
\label{eq:fund_tensor_F}
g^F_v(u,w) \coloneqq \frac{1}{2} \left. \frac{\partial^2}{\partial r \partial s} F(v+ru+sw)^2 \right\rvert_{r=s=0}, \quad \forall u,w \in T_pQ,
\end{equation}
is positive definite.
\end{itemize}
In this case, each $ F|_{T_pQ} $ is called a {\em Minkowski norm} and we say that $ (Q,F) $ is a {\em Finsler manifold}.
\end{defi}

\begin{defi}
A {\em Lorentz-Finsler metric} on $ Q $ is a smooth function $ L: A \subset TQ \setminus \textbf{0} \rightarrow (0,\infty) $ satisfying the following properties:
\begin{itemize}
\item $ \overline{A} $ (the closure of $ A $ in $ TQ \setminus \textbf{0} $) is an embedded smooth manifold with boundary $ \C $, such that each $ A_p \coloneqq A \cap T_pQ $ is non-empty, open, connected and conic.
\item $ L $ is positive homogeneous of degree 2: $ L(\lambda v) = \lambda^2 L(v) $ for all $ \lambda > 0 $.
\item For every $ v \in A_p $, $ p \in Q $, the fundamental tensor $ g^L_v $, defined as
\begin{equation}
\label{eq:fund_tensor_L}
g^L_v(u,w) \coloneqq \frac{1}{2} \left. \frac{\partial^2}{\partial r \partial s} L(v+ru+sw) \right\rvert_{r=s=0}, \quad \forall u,w \in T_pQ,
\end{equation}
has index $ n $, i.e. signature $ (+,-,\ldots,-) $.
\item $ L $ can be smoothly extended as zero to $ \C $ with non-degenerate fundamental tensor.
\end{itemize}
In this case, we say that $ (Q,L) $ is a {\em Finsler spacetime} with cone domain $ A $ and {\em lightcone} $ \C $.
\end{defi}

As the notation suggests, the lightcone $ \C $ of a Lorentz-Finsler metric is actually a cone structure (see \cite[Corollary 3.7]{JS20}). The full relationship between these two notions is given by the following result (see \cite[Theorem 3.11 and Corollary 5.8]{JS20}).

\begin{prop}
\label{prop:cone_structures}
Let $ \C $ be a cone structure on $ Q $. Then:
\begin{itemize}
\item $ \C $ is the lightcone of a (highly non-unique) Lorentz-Finsler metric on $ Q $.
\item Two Lorentz-Finsler metrics $ L, \tilde{L}: \overline{A} \rightarrow [0,\infty) $ with the same lightcone $ \C = \partial A \setminus \textbf{0} $ are anisotropically equivalent, i.e. their quotient $ \tilde{L}/L $ extends to a smooth positive function $ \lambda: \overline{A} \rightarrow (0,\infty) $.
\end{itemize}
Therefore, a cone structure $ \C $ uniquely determines a (non-empty) class of anisotropically equivalent Lorentz-Finsler metrics with lightcone $ \C $, any of which will be called {\em compatible with} $ \C $.
\end{prop}

\begin{rem}
\label{rem:comments}
Some comments regarding the previous definitions will be useful throughout this work:
\begin{enumerate}
\item Finsler spacetimes inherit the causality induced by their lightcones. In particular, since we only consider one cone $ \C_p $ at each tangent space (not a double cone as in Lorentzian geometry), this already determines a time orientation, i.e. every causal vector $ v $ is considered {\em future-directed}. One can also define {\em past-directed} vectors (when $ -v $ is causal), but this distinction will not be necessary here.
\item The fact that $ L $ is smooth on $ \C $ with non-degenerate fundamental tensor implies that $ L $ can be extended to an open conic subset $ A^* \subset TQ \setminus \textbf{0} $ containing $ \overline{A} $ such that $ g_v^L $ has index $ n $ on $ A^* $ and $ L < 0 $ on $ A^* \setminus \overline{A} $ (see \cite[Remark~3.6(2)]{JS20}). As usual, we will use the notation $ A^*_p \coloneqq A^* \cap T_pQ $.
\item From the definition of the fundamental tensor \eqref{eq:fund_tensor_F} of a Finsler metric $ F $ and the positive $ 1 $-homogeneity, it can be easily deduced that, for every $ v \in T_pQ \setminus \{0\} $, $ p \in Q $:
\begin{equation}
\label{eq:fund_tensor_2}
g^F_v(v,v) = F(v)^2 \qquad \text{and} \qquad g^F_v(v,w) = \frac{1}{2} \left. \frac{\partial}{\partial s} F(v+sw)^2 \right\rvert_{s=0}, \quad \forall w \in T_pQ.
\end{equation}
Analogously, from the definition of the fundamental tensor \eqref{eq:fund_tensor_L} of a Lorentz-Finsler metric $ L $ and the positive $ 2 $-homogeneity, for every $ v \in A^*_p $, $ p \in Q $:
\begin{equation}
\label{eq:fund_tensor_1}
g^L_v(v,v) = L(v) \qquad \text{and} \qquad g^L_v(v,w) = \frac{1}{2} \left. \frac{\partial}{\partial s} L(v+sw) \right\rvert_{s=0}, \quad \forall w \in T_pQ.
\end{equation}
\item Given $ \C $, one can always find a timelike one-form $ \omega $ on $ Q $ (i.e., $ \omega(v) > 0 $ for any causal vector $ v $) and an $ \omega $-unit timelike vector field $ T $ on $ Q $ (i.e., $ T $ is timelike and $ \omega(T) = 1 $). This yields the natural splitting $ TQ = \textup{Span}(T) \oplus \textup{Ker}(\omega) $, with the projection $ \pi: TQ \rightarrow \textup{Ker}(\omega) $ determined at each $ p \in Q $ by
\begin{equation}
\label{eq:v_decomposition}
v = \omega(v)T_p + \pi(v), \quad \forall v \in T_pQ.
\end{equation}
Then, there exists a unique Finsler metric $ F $ on the vector bundle $ \textup{Ker}(\omega) $ such that, for any $ v \in T_pQ $, $ p \in Q $:
\begin{equation}
\label{eq:cone_triple}
v \in \C_p \Leftrightarrow v = F(\pi(v))T_p + \pi(v).
\end{equation}
We say that $ (\omega,T,F) $ is a {\em cone triple associated with} $ \C $. Conversely, for any cone triple $ (\omega,T,F) $ composed of a non-vanishing one-form $ \omega $, an $ \omega $-unit vector field $ T $ and a Finsler metric $ F $ on $ \text{Ker}(\omega) $, there exists a unique cone structure $ \C $ satisfying \eqref{eq:cone_triple}, which is said to be {\em associated with} the cone triple  (see \cite[\S~2.4]{JS20}).

Moreover, in both cases a (non-smooth) Lorentz-Finsler metric $ L: TQ \rightarrow \mathds{R} $ compatible with $ \C $ is given by
\begin{equation}
\label{eq:non_smooth_L}
L(v) \coloneqq \omega(v)^2-F(\pi(v))^2, \quad \forall v \in T_pQ, \quad \forall p \in Q.
\end{equation}
Observe that $ L $ is defined on the whole $ TQ $, but it fails to be smooth on $ \textup{Span}(T) $ (unless $ F $ is Riemannian). Nevertheless, this type of Lorentz-Finsler metrics can be smoothened without altering the metric in a neighborhood of $ \C $ (in fact, it is only necessary to modify it within an arbitrarily small neighborhood of $ \textup{Span}(T) $). As a consequence, any cone structure $ \C $ is the lightcone of a (smooth) Lorentz-Finsler metric $ L $ with cone domain $ TQ $ (see \cite[\S~5]{JS20}).

\item When the underlying manifold is a global splitting $ Q = \mathds{R} \times N $, being $ t: Q \rightarrow \mathds{R} $ the standard projection, the most natural cone triple is given by $ \omega = dt $, $ T = \partial_t $ and any Finsler metric $ F^t $ on $ \text{Ker}(dt) $. Note that $ \text{Ker}(dt_{(t,x)}) \equiv \{t\} \times T_xN $ for any $ (t,x) \in Q $, so $ F^t $ can be seen as a {\em time-dependent Finsler metric} on $ N $---namely, $ F^t $ varies smoothly with $ t $ and, for each $ t_0 \in \mathds{R} $, $ F^{t_0} $ is a Finsler metric on $ N $. In this setting, the (non-smooth) Lorentz-Finsler metric \eqref{eq:non_smooth_L} takes the form $ L \coloneqq dt^2-(F^t)^2 $, i.e.
\begin{equation*}
L(v^0,\tilde{v}) \coloneqq (v^0)^2-F^t(\tilde{v})^2, \quad \forall (v^0,\tilde{v}) \in T_pQ \equiv \mathds{R} \times T_xN, \quad \forall p=(t,x) \in Q,
\end{equation*}
where $ v^0 = \omega(v) $ and $ \tilde{v}=\pi(v) $, following the decomposition of $ v $ in \eqref{eq:v_decomposition}. The cone structure associated with $ (dt,\partial_t,F^t) $ can be expressed as $ \C = L^{-1}(0) \cap dt^{-1}((0,\infty)) $, and observe that $ t $ is a temporal function for this cone structure. Also, the fact that $ L $ is not smooth on $ \textup{Span}(\partial_t) $ will not be relevant for our purposes here, as our study will be focused on lightlike curves and $ \partial_t $ is always timelike in this case. We will use this construction explicitly in \S~\ref{sec:natural}.
\end{enumerate}
\end{rem}

\subsection{Orthogonality and transversality}
\label{subsec:orthogonality}
One of the key concepts we will be constantly using throughout the text is the notion of orthogonality. Notice that, in the classical Finslerian setting, this is a non-symmetric relation because, given a Finsler metric $ F $ and two vectors $ v, w \in T_pQ \setminus\{0\} $, one has two scalar products, $ g^F_v $ and $ g^F_w $. We will first provide an abstract definition of orthogonality for vectors in cone structures and, then, introduce the natural notion for Lorentz-Finsler metrics, which is compatible with the former.

\begin{defi}
\label{def:orth}
Let $ \C $ be a cone structure on $ Q $:
\begin{itemize}
\item We say that $ v \in \C_p $ is {\em $ \C $-orthogonal} to $ w \in T_pQ $, denoted $ v \perp_{\C} w $, if $ w \in T_v\C_p $. We will use the notation $ v^{\perp_\C} \coloneqq \{w \in T_pQ: v \perp_\C w\} = T_v\C_p $ (naturally identified as a subspace of $ T_pQ $).
\item If $ P $ is a submanifold of $ Q $ and $ v \in \C_p $, then $ v $ is {\em $ \C $-orthogonal} to $ P $, denoted $ v \perp_\C P $, if $ v \perp_\C w $ for all $ w \in T_pP $, i.e. $ T_pP \subset v^{\perp_\C} $. Similarly, if $ \Pi $ is a linear subspace of $ T_pQ $, then $ v \perp_\C \Pi $ if $ v \perp_\C w $ for all $ w \in \Pi $, i.e. $ \Pi \subset v^{\perp_\C} $.
\end{itemize}
\end{defi}

\begin{defi}
Let $ (Q,L) $ be a Finsler spacetime with cone domain $ A $ and $ (Q,F) $ a Finsler manifold.
\begin{itemize}
\item We say that $ v \in \overline{A}_p $ is {\em $ L $-orthogonal} to $ w \in T_pQ $, denoted $ v \perp_L w $, if
\begin{equation*}
g^L_v(v,w) = 0,
\end{equation*}
and we will use the notation $ v^{\perp_L} \coloneqq \{w \in T_pQ: v \perp_L w\} $.
\item Analogously, $ v \in T_pQ \setminus \{0\} $ is {\em $ F $-orthogonal} to $ w \in T_pQ $, denoted $ v \perp_F w $, if
\begin{equation*}
g^F_v(v,w) = 0,
\end{equation*}
and we will use the notation $ v^{\perp_F} \coloneqq \{w \in T_pQ: v \perp_F w\} $.
\item If $ P $ is a submanifold of $ Q $ and $ v \in \overline{A}_p $ (resp. $ v \in T_pQ \setminus \{0\} $), then $ v \perp_L P $ (resp. $ v \perp_F P $) if $ v \perp_L w $ (resp. $ v \perp_F w $) for all $ w \in T_pP $. Similarly, if $ \Pi $ is a linear subspace of $ T_pQ $, then $ v \perp_L \Pi $ (resp. $ v \perp_F \Pi $) if $ v \perp_L w $ (resp. $ v \perp_F w $) for all $ w \in \Pi $.
\end{itemize}
\end{defi}

\begin{rem}
\label{rem:orth}
It is important to keep the following observations in mind:
\begin{enumerate}
\item The non-symmetry is a feature of every definition. In particular, note that $ v^{\perp_{\C}} $ is only defined for lightlike vectors, and $ v^{\perp_L} $ for causal ones.
\item If $ L $ is a Lorentz-Finsler metric with lightcone $ \C $ and $ v \in \C_p $, then $ v^{\perp_L} = v^{\perp_\C} = T_v\C_p $ (see \cite[Proposition~3.4(iii)]{JS20}). As a result, $ v \perp_\C w $ if and only if $ v \perp_L w $ for one (and then, for all) Lorentz-Finsler metric $ L $ compatible with $ \C $.
\item For any $ v \in \C_p $, the hyperplane $ v^{\perp_{\C}} = T_v\C_p $ is lightlike and any vector in $ v^{\perp_{\C}} $ non-proportional to $ v $ is spacelike, due to the non-radial strong convexity of the cone $ \C_p $ (recall Definition~\ref{def:cone_structure}(ii)). In particular, for any cone structure $ \C $, two lightlike vectors are $ \C $-orthogonal if and only if they are linearly dependent.
\end{enumerate}
\end{rem}

Together with these definitions of orthogonality, the following notions of transversality will also appear throughout this work.

\begin{defi}
\label{def:top_transverse}
Let $ \gamma: I \rightarrow Q $ be a piecewise smooth curve, and $ P $ a submanifold of $ Q $. We say that $ \gamma $ is {\em topologically transverse} to $ P $ if it intersects $ P $ exactly once at some $ \tau \in I $. If, in addition, $ \dot{\gamma}(\tau^{\pm}) \notin T_{\gamma(\tau)}P $, such a curve is said to be {\em transverse} to $ P $. Analogously, given $ p \in P $ and $ v \in T_pQ $, we say that $ v $ is {\em transverse} to $ P $ if $ v \notin T_pP $.
\end{defi}

\begin{defi}
\label{def:c_transverse}
Let $ \C $ be a cone structure on $ Q $, and $ \eta $ a hypersurface of $ Q $. Given $ p \in \eta $ and $ v \in \C_p $, we say that $ v $ is $ \C $-{\em transverse} to $ \eta $ if $ v^{\perp_\C} \not= T_p\eta $.
\end{defi}

\begin{rem}
\label{rem:c_transverse}
The restriction that $ v $ be $ \C $-transverse to $ \eta $ is equivalent to saying that $ \C_p $ cannot be tangent to $ T_p\eta $ in the direction $ v $, i.e. $ T_v\C_p \not= T_p\eta $. Another equivalent condition is that $ v \not\perp_{\C} T_p\eta $, i.e. there exists $ w \in T_p\eta $ such that $ v \not\perp_{\C} w $. In particular, this automatically holds when $ v $ is transverse to $ \eta $ (transversality implies $ \C $-transversality). Conversely, $ v^{\perp_\C} = T_v\C_p = T_p\eta $ (or, equivalently, $ v \perp_{\C} T_p\eta $) if and only if $ T_p\eta $ is lightlike and $ v \in T_p\eta $.
\end{rem}

\subsection{Chern anisotropic connection and geodesics}
\label{subsec:chern}
Geodesics will play an important role in this work. For the sake of completeness, we provide now an overview on how to define and compute these curves in Finsler spacetimes. Even though we give here very general definitions, keep in mind that our study will be focused specifically on lightlike geodesics. Throughout this subsection, the term {\em $ L $-admissible} (applied to curves or vector fields) means that each tangent vector is causal---or, more generally, that it belongs to an extension $A^*$ of the cone domain of $ L $, following Remark~\ref{rem:comments}(2).

\begin{notation}
\label{not:notation1}
Let $ \varphi=(x^0,\ldots,x^n) $ be a coordinate system on $ \Omega \subset Q $ and consider also the associated coordinate system $ (\varphi,\dot{\varphi}) $ on $ T\Omega \subset TQ $, being $ \dot{\varphi} = (y^0,\ldots,y^n) $ the natural coordinates on the fiber induced by $ \varphi $. Any (piecewise smooth) curve $ \gamma $ will be written in these coordinates as
\begin{equation*}
\begin{split}
\gamma(t) & \equiv (\gamma^0(t),\ldots,\gamma^n(t)), \\
\dot{\gamma}(t) & \equiv (\dot{\gamma}^0(t),\ldots,\dot{\gamma}^n(t)), \quad \text{with} \quad \dot{\gamma}^i(t) = \frac{d\gamma^i}{dt}(t),
\end{split}
\end{equation*}
and, similarly, we will write $ V \equiv (V^0,\ldots,V^n) $ for vector fields and $ v \equiv (v^0,\ldots,v^n) $ for vectors. Also, Einstein's summation convention will be used throughout the text, omitting the sum from $ 0 $ to $ n $ when an index appears up and down.
\end{notation}

Let $ (Q,L) $ be a Finsler spacetime with cone domain $ A $. A ($ 0 $-homogeneous) {\em anisotropic connection} is a global intrinsic notion (see \cite[Definition~4]{JSV22}) which, from a practical viewpoint, determines an affine connection $ \nabla^V $ for each choice of an $ L $-admissible vector field $ V $ on $ \Omega \subset Q $.\footnote{This is an intermediate notion between the nonlinear connection and the classical Finsler connections (see \cite[\S~4 and 6]{JSV22}).} This connection $ \nabla^V $ only depends pointwise on $ v = V_p $ (not on the extension of $ v $) and, locally, it is characterized by its {\em Christoffel symbols}, defined as the smooth functions $ \Gamma^k_{\ ij}: T\Omega \cap \overline{A} \rightarrow \mathds{R} $ satisfying:
\begin{equation*}
\nabla^v_{\partial_{x^i}}(\partial_{x^j}) = \Gamma^k_{\ ij}(v) \partial_{x^k}, \quad i,j,k = 0,\ldots,n.
\end{equation*}
This allows us, given a smooth curve $ \gamma(t) $ and an $ L $-admissible vector field $ W $ along $ \gamma $, to define the {\em covariant derivative $ D^W_{\gamma} $ along $ \gamma $ with reference $ W $} (see \cite[Proposition~5]{JSV22}):
\begin{equation}
\label{eq:covariant}
D^W_{\gamma}X \coloneqq \frac{dX^i}{dt} \partial_{x^i} + X^i \dot{\gamma}^j \Gamma^k_{\ ij}(W) \partial_{x^k},
\end{equation}
for every smooth vector field $ X $ along $ \gamma $. Observe that the (positive) $ 0 $-homogeneity of $ \nabla^V $ is transferred to the Christoffel symbols and, thus, also to the covariant derivative: $ D^{\lambda W}_{\gamma} = D^W_{\gamma} $ for all $ \lambda > 0 $.

This way, we say that an $ L $-admissible smooth curve $ \gamma $ is a {\em geodesic} of $ L $ if it is an auto-parallel curve of the chosen anisotropic connection, i.e. $ D_{\gamma}^{\dot{\gamma}}\dot\gamma = 0 $. In coordinates, using \eqref{eq:covariant}, this is equivalent to the {\em geodesic equations}:
\begin{equation}
\label{eq:geod_eq0}
\ddot{\gamma}^k + \Gamma^k_{\ ij}(\dot{\gamma}) \dot{\gamma}^i \dot{\gamma}^j = 0, \quad k = 0,\ldots,n.
\end{equation}
In addition, we say that $ \gamma $ is a {\em pregeodesic} if it can be reparametrized as a geodesic (keeping the orientation). As in the semi-Riemannian case, one has the following characterization of pregeodesics.

\begin{prop}
\label{prop:pregeodesic}
Let $ (Q,L) $ be a Finsler spacetime with an anisotropic connection, and $ \gamma: I \rightarrow Q $ an $ L $-admissible smooth curve. Then, $ \gamma $ is a pregeodesic of $ L $ if and only if
\begin{equation}
\label{eq:cond_pregeod}
D_{\gamma}^{\dot{\gamma}}\dot\gamma(t) = f(t) \dot{\gamma}(t), \quad \forall t \in I,
\end{equation}
for some function $ f: I \rightarrow \mathds{R} $.
\end{prop}
\begin{proof}
Let $ \beta(s) \coloneqq \gamma(t(s)) $ be an arbitrary (positive) reparametrization of $ \gamma $, so that $ \dot{\beta}(s) = \dot{t}(s)\dot{\gamma}(t(s)) $. Using \eqref{eq:covariant}, we get
\begin{equation}
\label{eq:pregeodesic}
D_{\beta}^{\dot{\beta}}\dot{\beta}(s) = \ddot{t}(s)\dot{\gamma}(t(s)) + \dot{t}(s)^2 D_{\gamma}^{\dot{\gamma}}\dot\gamma(t(s)).
\end{equation}
If $ \gamma $ is a pregeodesic, then $ \beta $ becomes a geodesic for some $ t(s) $. This means that $ D_{\beta}^{\dot{\beta}}\dot{\beta}(s) = 0 $ and thus, from \eqref{eq:pregeodesic}, $ \gamma $ satisfies \eqref{eq:cond_pregeod} with $ f(t) = \left. -\frac{\ddot{t}(s)}{\dot{t}(s)^2} \right\rvert_{s=s(t)} $. Conversely, if \eqref{eq:cond_pregeod} holds, we obtain from \eqref{eq:pregeodesic}:
\begin{equation*}
D_{\beta}^{\dot{\beta}}\dot{\beta}(s) = (\ddot{t}(s) + \dot{t}(s)^2 f(t(s))) \dot\gamma(t(s)),
\end{equation*}
and then, it is straightforward to check that, choosing $ t(s) $ so that its inverse $ s(t) $ satisfies $ \dot{s}(t) = \dot{s}(t_0) e^{\int_{t_0}^t f(r)dr} $ for some $ t_0 \in I $ (with $ \dot{s}(t_0) > 0 $), $ \beta $ becomes a geodesic.
\end{proof}

Throughout this work, we will use specifically the {\em Chern anisotropic connection} $ \nabla $ (see \cite[Theorem~4]{JSV22}), which is determined by the following family of affine connections: given an $ L $-admissible vector field $ V $ on $ \Omega \subset Q $, we define $ \nabla^V $ as the unique affine connection which is torsion-free and almost $ g $-compatible (see \cite[Proposition~2.3]{J14}). The almost $ g $-compatibility is also transferred to the covariant derivative (see \cite[Equation~(4)]{J14}):
\begin{equation}
\label{eq:g_compat}
\frac{d}{dt}g^L_W(X,Y) = g^L_W(D^W_{\gamma}X,Y) + g^L_W(X,D^W_{\gamma}Y) + 2C^L_W(D^W_{\gamma}W,X,Y),
\end{equation}
for any smooth vector fields $ X,Y $ along $ \gamma $, where $ C^L $ is the {\em Cartan tensor} of $ L $, defined for every $ v \in \overline{A}_p $, $ p \in Q $, as
\begin{equation}
\label{eq:cartan}
C^L_v(w_1,w_2,w_3) \coloneqq \frac{1}{4} \left.\frac{\partial^3}{\partial s^1 \partial s^2 \partial s^3} L(v+s^i w_i) \right\rvert_{s^1=s^2=s^3=0}, \quad \forall w_1, w_2, w_3 \in T_pQ,
\end{equation}
which is symmetric in its three arguments and satisfies (from the $ 2 $-homogeneity of $ L $ and Euler's theorem): $ C^L_v(v,\cdot,\cdot) = C^L_v(\cdot,v,\cdot) = C^L_v(\cdot,\cdot,v) = 0 $.

In coordinates, the Christoffel symbols of $ \nabla $ are given by
\begin{equation*}
\Gamma^{k}_{\ ij}(v) = \frac{1}{2}(g^{L}_{v})^{kr}\left(\frac{\delta (g^{L}_{v})_{rj}}{\delta x^i}+\frac{\delta (g^{L}_{v})_{ri}}{\delta x^j}-\frac{\delta (g^{L}_{v})_{ij}}{\delta x^r}\right), \quad \forall v \in \overline{A}_p, \quad \forall p \in Q,
\end{equation*}
where $ (g^{L}_{v})_{ij} \coloneqq g^L_v(\partial_{x^i},\partial_{x^j}) $, $ (g^{L}_{v})^{ij} $ denotes the coordinates of the inverse matrix of $ (g^{L}_{v})_{ij} $ and $ \delta/\delta x^i $ is defined using the coefficients $ N^k_{\ i} $ of the {\em nonlinear connection}, namely
\begin{equation*}
\frac{\delta}{\delta x^i} \coloneqq \frac{\partial}{\partial x^i} - N^k_{\ i}\frac{\partial}{\partial y^k}, \quad \text{with} \quad N^k_{\ i}(v) = \gamma^k_{\ ij}(v)v^j-(g^L_v)^{kj}(C_v^L)_{ijl}\gamma^l_{\ rs}(v)v^rv^s,
\end{equation*}
being $ (C_v^L)_{ijk} \coloneqq C_v^L(\partial_{x^i},\partial_{x^j},\partial_{x^k}) $ and $ \gamma^k_{\ ij} $ the {\em formal Christoffel symbols} of $ L $, defined as
\begin{equation}
\label{eq:christoffel}
\gamma_{\ ij}^k(v) \coloneqq \frac{1}{2}(g^{L}_{v})^{kr}\left(\frac{\partial (g^{L}_{v})_{rj}}{\partial x^i}+\frac{\partial (g^{L}_{v})_{ri}}{\partial x^j}-\frac{\partial (g^{L}_{v})_{ij}}{\partial x^r}\right), \quad \forall v \in \overline{A}_p, \quad \forall p \in Q.
\end{equation}
In the end, when working with the Chern anisotropic connection, it turns out that only these formal Christoffel symbols contribute to the geodesic equations \eqref{eq:geod_eq0}, so an $ L $-admissible smooth curve $ \gamma $ is a geodesic of $ L $ if and only if
\begin{equation}
\label{eq:geod_eq}
\ddot{\gamma}^k + \gamma^k_{\ ij}(\dot{\gamma}) \dot{\gamma}^i \dot{\gamma}^j = 0, \quad k = 0,\ldots,n.
\end{equation}

\begin{rem}
We will be interested in properties of geodesics, which can be determined from the nonlinear connection. Hence, the results will be independent of the choice of the anisotropic connection, so long as it shares the same geodesics as the Chern anisotropic connection (e.g., this is the case of the Berwald anisotropic connection). Anyway, the use of the Chern anisotropic connection will be useful for computations related to geodesics, in a similar way the choice of a compatible Lorentz-Finsler metric is useful for computations related to cone structures.

Also, it is worth pointing out that all of the above remains the same for the definition of geodesics in a Finsler manifold $ (Q,F) $, simply by replacing the Lorentz-Finsler metric $ L $ with the Finsler metric $ F $ everywhere (except in \eqref{eq:cartan}, where $ L $ should be replaced by $ F^2 $) and, consistently, $ \overline{A} $ with $ TQ \setminus \textbf{0} $.
\end{rem}

Finally, the association between cone structures and Lorentz-Finsler metrics pointed out in Proposition~\ref{prop:cone_structures} is also transferred to their cone and lightlike geodesics (see \cite[Proposition 3.4]{JavSoa20} and \cite[Theorem 6.6]{JS20}).

\begin{thm}
\label{thm:cone_geodesics}
Given a cone structure $ \C $ on $ Q $:
\begin{itemize}
\item All Lorentz-Finsler metrics compatible with $ \C $ share the same lightlike pregeodesics.
\item A curve is a cone geodesic if and only if it is a lightlike pregeodesic of one (and then, of all) Lorentz-Finsler metric compatible with $ \C $.
\end{itemize}
\end{thm}

\section{General setting}
\label{sec:setting}
Let $ Q $ be a smooth manifold of dimension $ n+1 \geq 3 $ and consider a hypersurface $ \eta $ that divides it into two open subsets $ Q_1, Q_2 $, so that their closures $ \overline{Q}_1, \overline{Q}_2 $ (satisfying $ \overline{Q}_1 \cup \overline{Q}_2 = Q $) are smooth manifolds with the same boundary $ \eta = \partial Q_1 = \partial Q_2 $. Let $ \C^1, \C^2 $ be two different cone structures defined respectively on $ \overline{Q}_1, \overline{Q}_2 $,\footnote{This means, in particular, that $ \C^1, \C^2 $ can be smoothly extended to an open region a little beyond $ \eta $.} which in general do not match even continuously on $ \eta $. When needed, $ L_{\mu} $ ($ \mu = 1,2 $) will denote any Lorentz-Finsler metric compatible with $ \C^{\mu} $, so that $ (Q_{\mu},L_{\mu}) $ is a Finsler spacetime with lightcone $ \C^{\mu} $. Some comments regarding the causality of the different elements we will work with are in order:
\begin{itemize}
\item We do not impose any condition on the causality of the interface $ \eta $ a priori. However, this will become relevant---and will be studied in detail---in \S~\ref{sec:existence}.

\item Similarly, no causality condition is imposed on any cone structure, although global hyperbolicity or, at least, stable causality (so that there is a temporal function) are typical assumptions in this setting, particularly when there is a physical interpretation involved. This particular case is addressed in \S~\ref{sec:natural}.
\end{itemize}

Throughout this work, we will focus mainly on the study of lightlike curves, with the following generalizations to the extended setting $ (Q,\C^1 \cup \C^2) $ presented above. Let $ \gamma: [a,b] \rightarrow Q $ be a (piecewise smooth) curve topologically transverse to $ \eta $ (recall Definition~\ref{def:top_transverse}), with $ \gamma|_{[a,\tau)} \subset Q_1 $, $ \gamma|_{(\tau,b]} \subset Q_2 $ and $ \gamma(\tau) \in \eta $, for some $ \tau \in (a,b) $. We say that $ \gamma $ is {\em lightlike} if it is so for both cone structures, i.e.
\begin{equation*}
\begin{split}
& \dot{\gamma}(t) \in \C^1_{\gamma(t)}, \quad \forall t \in [a,\tau], \\
& \dot{\gamma}(t) \in \C^2_{\gamma(t)}, \quad \forall t \in [\tau,b],
\end{split}
\end{equation*}
which means that, at the interface, $ \dot{\gamma}(\tau^-) \in \C^1_{\gamma(\tau)} $ and $ \dot{\gamma}(\tau^+) \in \C^2_{\gamma(\tau)} $. Also, for any Lorentz-Finsler metrics $ L_1, L_2 $ compatible with $ \C^1, \C^2 $:
\begin{equation*}
\begin{split}
L_1(\dot{\gamma}(t)) & = 0, \quad \forall t \in [a,\tau], \\
L_2(\dot{\gamma}(t)) & = 0, \quad \forall t \in [\tau,b],
\end{split}
\end{equation*}
where $ L_1(\dot{\gamma}(\tau)) \coloneqq L_1(\dot{\gamma}(\tau^-)) $ and $ L_2(\dot{\gamma}(\tau)) \coloneqq L_2(\dot{\gamma}(\tau^+)) $.

\begin{defi}
\label{def:arrival_time}
Let $ P \subset Q_1 $ be a submanifold of codimension $ r \geq 1 $, and $ \alpha: I \rightarrow Q_2 $ an embedded smooth curve, where $ I \subset \mathds{R} $ is an open interval.\footnote{This setting will be used directly for refraction, but the case of reflection---when $ \alpha $ is contained in $ Q_1 $---will be seen as a particular case of this framework (see \S~\ref{sec:reflection} below).} Fixing an interval $ [a,b] $ and some $ \tau \in (a,b) $, let $ \mathcal{N}_{P,\alpha} $ be the set of all (regular) piecewise smooth curves $ \gamma: [a,b] \rightarrow Q $ such that (see Figure~\ref{fig:setting}):
\begin{itemize}
\item $ \gamma(a) \in P $, $ \gamma(\tau) \in \eta $ and $ \gamma(b) \in \textup{Im}(\alpha) $.
\item $ \gamma $ is topologically transverse to $ \eta $ and lightlike.
\end{itemize}
We define the {\em arrival time functional} $ \T $ as
\begin{equation}
\label{eq:arrival_time}
\begin{array}{crcl}
\T \colon & \mathcal{N}_{P,\alpha} & \longrightarrow & I \\
& \gamma & \longmapsto & \T[\gamma] \coloneqq \alpha^{-1}(\gamma(b)).
\end{array}
\end{equation}
In addition, we say that $ \gamma \in \mathcal{N}_{P,\alpha} $ is {\em cone-transverse} to $ \eta $ if $ \dot{\gamma}(\tau^-) $ is $ \C^1 $-transverse to $ \eta $ and $ \dot{\gamma}(\tau^+) $ is $ \C^2 $-transverse to $ \eta $ (recall Definition~\ref{def:c_transverse}).
\end{defi}

\begin{figure}
\centering
\includegraphics[width=0.45\textwidth]{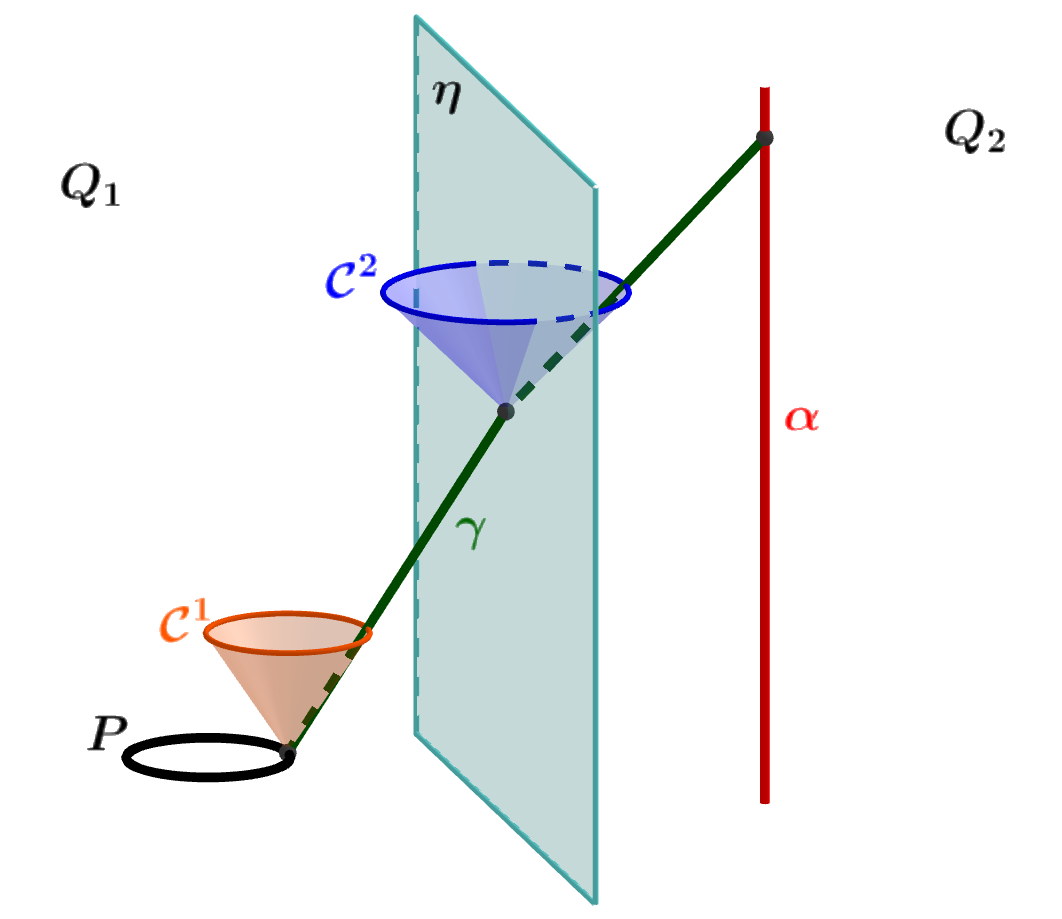}
\caption{A simple representation of the setting we are working on. We highlight that $ P $, $ \eta $ and $ \alpha $ can have, in principle, any causal character (in the figure, $ P $ is $ \C^1 $-spacelike, $ \eta $ is timelike for both cone structures and $ \alpha $ is $ \C^2 $-timelike).}
\label{fig:setting}
\end{figure}

\begin{rem}
\label{rem:parametrization}
Some justifications for the hypotheses in Definition~\ref{def:arrival_time} are worth mentioning:
\begin{enumerate}
\item Our study will be focused on the local behavior of $ \gamma $ around $ \tau $, so the topological transversality becomes a natural simplification---if the global curve ends up crossing the interface more than once at isolated points, each intersection can be treated separately. If we remove this hypothesis, we would have not only that $ \gamma $ can remain in $ \eta $, but also the possibility that $ \gamma $ crosses the interface infinitely many times close to $ \tau $. This can happen even in the most favorable case when the metrics on $ Q_1 $ and $ Q_2 $ are the same, $ \gamma $ is a lightlike geodesic with no breaks and $ \eta $ is timelike and smooth (this can be constructed explicitly by reasoning as in \cite[Appendix~A]{AFS21}).
\item For some results, we will need to impose the stronger restriction that $ \gamma $ be transverse or cone-transverse to $ \eta $. Specifically, the former will be required in Proposition~\ref{prop:variational_field}(II), and the latter in Theorem~\ref{thm:snell}(II), Corollary~\ref{thm:reflection}(II) and \S~\ref{subsec:orientation} below.
\item Note that $ \T[\gamma] $ is independent of the parametrization of $ \gamma $, so any piecewise smooth curve from $ P $ to $ \alpha $ (crossing $ \eta $ once) can be reparametrized so that it belongs to $ \mathcal{N}_{P,\alpha} $, without affecting its arrival time. Therefore, the choice of $ [a,b] $ and $ \tau $ is arbitrary and non-restrictive. Also, we stress that $ P $ can be, in particular, a single point.
\end{enumerate}
\end{rem}

\begin{defi}
\label{def:variation}
Let $ \gamma \in \mathcal{N}_{P,\alpha} $, with (possible) breaks at $ \{t_i\}_{i=1}^k \subset (a,b) $ (including $ t = \tau $):
\begin{itemize}
\item[(i)] An {\em admissible variation} $ \Lambda $ of $ \gamma $ is a continuous map
\begin{equation*}
\begin{array}{crcl}
\Lambda \colon & (-\varepsilon,\varepsilon) \times [a,b] & \longrightarrow & Q \\
& (\omega,t) & \longmapsto & \Lambda(\omega,t) \coloneqq \gamma_{\omega}(t)
\end{array}
\end{equation*}
such that
\begin{itemize}
\item[$ \circ $] $ \Lambda $ is smooth on $ (-\varepsilon,\varepsilon) \times [t_{i-1},t_i] $ for $ i = 1,\ldots,k+1 $, where $ t_0 = a $ and $ t_{k+1} = b $.

\item[$ \circ $] $ \Lambda(\omega,\cdot) = \gamma_{\omega} \in \mathcal{N}_{P,\alpha} $ for all $ \omega \in (-\varepsilon,\varepsilon) $, with $ \gamma_0 = \gamma $. In particular, $ \gamma_{\omega}(a) \in P $, $ \gamma_{\omega}(\tau) \in \eta $ and $ \gamma_{\omega}(b) \in \textup{Im}(\alpha) $.
\end{itemize}
\item[(ii)] Let $ \mathfrak{X}(\gamma) $ be the set of all piecewise smooth vector fields along $ \gamma $ (with one possible break at $ t=\tau $). The {\em variational vector field} $ Z \in \mathfrak{X}(\gamma) $ associated with $ \Lambda $ is given by the tangent vector of the curve $ \omega \mapsto \Lambda(\omega,t) $ at $ \omega = 0 $, i.e.
\begin{equation*}
Z(t) \coloneqq \frac{\partial \Lambda}{\partial \omega}(0,t) = \left. \frac{\partial \gamma_{\omega}(t)}{\partial \omega} \right\rvert_{\omega = 0}.
\end{equation*}

\item[(iii)] Consistently, any $ Z \in \mathfrak{X}(\gamma) $ will be called an {\em admissible variational vector field} if it is the variational vector field associated with an admissible variation of $ \gamma $.

\item[(iv)] We say that $ \gamma $ is a {\em critical point} of $ \T $ if
\begin{equation*}
\left. \frac{d}{d\omega} \T[\gamma_{\omega}] \right\rvert_{\omega=0} = 0
\end{equation*}
for every admissible variation $ \Lambda(\omega,\cdot) = \gamma_{\omega} $ (notice that the map $ \omega\mapsto \T[\gamma_\omega] = \alpha^{-1}(\gamma_\omega(b)) $ is always smooth).
\end{itemize}
\end{defi}

\begin{notation}
Following Notation~\ref{not:notation1}, any coordinate system $ (\varphi,\dot{\varphi}) $ on $ T\Omega \subset TQ $ will be written as $ \varphi = (x^0,\ldots,x^n) $, $ \dot{\varphi} = (y^0,\ldots,y^n) $ and, if $ \Lambda(\omega,t) = \gamma_{\omega}(t) $ is an admissible variation of $ \gamma $ with associated variational vector field $ Z $, then in coordinates:
\begin{equation*}
\begin{split}
\gamma_{\omega}(t) & \equiv (\gamma_{\omega}^0(t),\ldots,\gamma_{\omega}^n(t)), \\
\dot{\gamma}_{\omega}(t) & \equiv (\dot{\gamma}_{\omega}^0(t),\ldots,\dot{\gamma}_{\omega}^n(t)), \quad \text{with} \quad \dot{\gamma}_{\omega}^i(t) = \frac{\partial \gamma_{\omega}^i(t)}{\partial t}, \\
Z(t) & \equiv (Z^0(t),\ldots,Z^n(t)), \quad \text{with} \quad Z^i(t) = \left. \frac{\partial \gamma_{\omega}^i(t)}{\partial \omega} \right\rvert_{\omega=0}.
\end{split}
\end{equation*}
\end{notation}

\section{Fermat's principle}
\label{sec:fermat}
Snell's law is a well-known formula used in physics (and particularly, in optics) to describe the change in direction that light---or any other wave---undergoes when passing through the interface $ \eta $ of two different media $ Q_1, Q_2 $. When the speed of the wave is anisotropic (i.e., direction-dependent), this effect can be modeled using two Finsler metrics $ F_1,F_2 $ whose unit vectors represent the wave velocity (see \cite{MP23}). In the isotropic case, $ F_1,F_2 $ reduce to Riemannian norms and one recovers the classical version of Snell's law. Here we will go one step further by endowing $ Q_1,Q_2 $ with two cone structures $ \C^1, \C^2 $ (or any choice of two compatible Lorentz-Finsler metrics $ L_1,L_2 $), as described above, with the aim of obtaining a time-dependent anisotropic Snell's law.

In general, Snell's law arises naturally from Fermat's principle: the change in the direction of propagation at $ \eta $ occurs to ensure that the wave trajectories are always critical points of the arrival time functional. Therefore, our goal in this section will be to find these critical points, establishing a generalized Fermat's principle when there is a discontinuity in the cone structure.

\subsection{Snell's law of refraction}
\label{sec:snell}
We will follow here the approaches presented in \cite{PP98,P06} and \cite[\S~7]{CJS14}, which provide different generalizations of the classical Fermat's principle: the first one establishes this principle in Lorentzian spacetimes when the source $ P $ is a spacelike submanifold, the last one when the arrival curve $ \alpha $ is an arbitrary embedded curve, and the second one generalizes the setting to Finsler spacetimes, but with the restriction that $ \alpha $ must be timelike. Our result will encompass all these cases and will be even more general, allowing for the cone structure to be discontinuous and for $ P $ to have any causal character (a priori). We begin by characterizing the admissible variational vector fields in the following proposition.

\begin{prop}
\label{prop:variational_field}
Let $ \gamma \in \mathcal{N}_{P,\alpha} $ and $ Z \in \mathfrak{X}(\gamma) $:
\begin{itemize}
\item[(I)] If $ Z $ is an admissible variational vector field, then $ Z(a) \in T_{\gamma(a)}P $, $ Z(\tau) \in T_{\gamma(\tau)}\eta $, $ Z(b) $ is proportional to $ \dot{\alpha}(\T[\gamma]) $ and
\begin{equation}
\label{eq:lemma}
\begin{split}
g^{L_1}_{\dot{\gamma}}(\dot{\gamma},\leftindex^1{D}_{\gamma}^{\dot{\gamma}} Z) & = 0, \quad \forall t \in [a,\tau], \\
g^{L_2}_{\dot{\gamma}}(\dot{\gamma},\leftindex^2{D}_{\gamma}^{\dot{\gamma}}Z) & = 0, \quad \forall t \in [\tau,b],
\end{split}
\end{equation}
for any Lorentz-Finsler metrics $ L_1, L_2 $ compatible with $ \C^1, \C^2 $, respectively, where $ \leftindex^{\mu}{D} $ refers to the covariant derivative of the Chern anisotropic connection of $ (Q_{\mu},L_{\mu}) $ (recall \S~\ref{subsec:chern}).
\item[(II)] The converse of \textup{(I)} holds when $ \gamma $ is transverse to $ \eta $ and $ \dot{\gamma}(b) \not\perp_{\C^2} \dot{\alpha}(\T[\gamma]) $.\footnote{This means that $ \dot{\alpha}(\T[\gamma]) \not\in \dot{\gamma}(b)^{\perp_{\C^2}} $, i.e. $ \dot{\alpha}(\T[\gamma]) \not\in T_{\dot{\gamma}(b)}\C_{\gamma(b)}^2 $. Equivalently, $ \dot{\gamma}(b) \not\perp_{L_2} \dot{\alpha}(\T[\gamma]) $ for any Lorentz-Finsler metric $ L_2 $ compatible with $ \C^2 $ (recall the orthogonality notions in \S~\ref{subsec:orthogonality}). Note that this condition automatically holds when $ \dot{\alpha}(\T[\gamma]) $ is $ \C^2 $-timelike.}
\end{itemize}
\end{prop}
\begin{proof}
If $ Z \in \mathfrak{X}(\gamma) $ is the variational vector field associated with an admissible variation $ \Lambda(\omega,t) = \gamma_{\omega}(t) $, then $ \gamma_{\omega}(a) \in P $, $ \gamma_{\omega}(\tau) \in \eta $ and $ \gamma_{\omega}(b) \in \textup{Im}(\alpha) $ for all $ \omega \in (-\varepsilon,\varepsilon) $, by Definition~\ref{def:variation}, which directly implies that $ Z(a) \in T_{\gamma(a)}P $, $ Z(\tau) \in T_{\gamma(\tau)}\eta $ and $ Z(b) $ is proportional to $ \dot{\alpha}(\T[\gamma]) $, respectively. Moreover, since $ \gamma_{\omega} $ is lightlike for all $ \omega $, we have on each $ Q_{\mu} $ (using \eqref{eq:fund_tensor_1}):
\begin{equation*}
0 = L_{\mu}(\dot{\gamma}_{\omega}) = g^{L_{\mu}}_{\dot{\gamma}_{\omega}}(\dot{\gamma}_{\omega},\dot{\gamma}_{\omega}), \quad \mu=1,2,
\end{equation*}
for any Lorentz-Finsler metric $ L_{\mu} $ compatible with $ \C^{\mu} $, and taking derivatives on both sides with respect to $ \omega $:
\begin{equation*}
\begin{split}
0 & = \left. \frac{\partial}{\partial\omega} g^{L_{\mu}}_{\dot{\gamma}_{\omega}}(\dot{\gamma}_{\omega},\dot{\gamma}_{\omega}) \right\rvert_{\omega=0} = \left. 2g^{L_{\mu}}_{\dot{\gamma}_{\omega}}(\dot{\gamma}_{\omega},\leftindex^{\mu}{D}_{\zeta_t}^{\dot{\gamma}_{\omega}}\dot{\gamma}_{\omega}) \right\rvert_{\omega=0} \\
& = 2g^{L_{\mu}}_{\dot{\gamma}}(\dot{\gamma},\leftindex^{\mu}{D}_{\gamma}^{\dot{\gamma}}Z), \quad \mu=1,2,
\end{split}
\end{equation*}
where $ \zeta_t $ is the curve $ \zeta_t: \omega \mapsto \gamma_\omega(t) $ and we have used that $ \leftindex^{\mu}{D}_{\zeta_t}^{\dot{\gamma}_{\omega}} \dot{\gamma}_{\omega} = \leftindex^{\mu}{D}_{\gamma_{\omega}}^{\dot{\gamma}_{\omega}} \dot{\zeta}_t $ (see \cite[Proposition~3.2]{J14}) and $ \dot{\zeta}_t(0) = Z(t) $. This way, we arrive at \eqref{eq:lemma}.

To prove (II), let $ \gamma $ be transverse to $ \eta $ and assume that $ \dot{\gamma}(b) \not\perp_{\C^2} \dot{\alpha}(\T[\gamma]) $. Then, given a vector field $Z$ satisfying the conditions in the second part of (I), we will construct an admissible variation $\Lambda$ of $\gamma$ whose associated variational vector field is $Z$. For this purpose, let us set a collection of  coordinate charts $(\Omega _1, \varphi_1),\ldots, (\Omega_l,\varphi_l)$ with the following properties:
\begin{enumerate}
\item[(i)] $\gamma(a)\in\Omega_1$ and in this chart it holds that $ P \equiv \{x^0_1=\ldots=x^{r-1}_1=0\} $ around $ \gamma(a) $.
\item[(ii)] The domains $\Omega_1,\ldots,\Omega_l$ cover $\gamma$, and there is no proper subcollection that covers $\gamma$. Moreover, each  $\Omega_i$ is intersected only by the previous and the next domains in the collection. 
\item[(iii)] The intersection of $\gamma$ with $\Omega_i\cap\Omega_{i+1}$ is connected  for all $i=1,\ldots,l-1$.
\item[(iv)] There is only one domain $\Omega_j$ that intersects $\eta$, and in this chart $\eta \equiv \{ x^1_j=0 \} $ around $ \gamma(\tau) $.
\item[(v)] $\Omega_l$ is the only domain that intersects $ \alpha $ in such a way that $\alpha$ is an integral curve of $ \partial_{x^0_l} $ around $ \gamma(b) $.
\item[(vi)] All the domains satisfy that $ \dot{\gamma} \not\perp_{\C^\mu} \partial_{x^0_i} $  everywhere in their intersection with $Q_\mu$ (this condition  on $ \partial_{x^0_i} $ is compatible at $ \gamma(b) $ with the fact that $\alpha$ is an integral curve of $\partial_{x^0_l}$  thanks to the non-orthogonality hypothesis on $ \alpha $).
\end{enumerate}
Writing the coordinates of $ \gamma$ as $ (\gamma^0_i,\gamma^1_i,\ldots,\gamma^n_i) = (\gamma^0_i,\tilde{\gamma}_i) $ and the coordinates of $ Z$ as $ (Z^0_i,Z^1_i,\ldots,Z^n_i) = (Z^0_i,\tilde{Z}_i) $ in the chart $(\Omega_i,\varphi_i)$, note that $ \gamma^0_1(a) = \ldots = \gamma^{r-1}_1(a) = 0 $, $ Z^0_1(a) = \ldots = Z^{r-1}_1(a) = 0 $.  Choose $a_1\in [a,b]$ such that $\gamma(a_1)\in \Omega_1\cap\Omega_2$ and $\gamma$ is smooth at $a_1$. We can construct the following variation $ \tilde{\Lambda}_1 $ of the projection $ \tilde{\gamma}_1|_{[a,a_1]} $ on the hyperplane $ \{x^0_1=0\} \subset \{0\} \times \mathds{R}^n $:
\begin{equation}
\label{eq:base_variation}
\begin{array}{crcl}
\tilde{\Lambda}_1 \colon & (-\varepsilon,\varepsilon) \times [a,a_1] & \longrightarrow & \{x^0_1=0\} \\
& (\omega,t) & \longmapsto & \tilde{\Lambda}_1(\omega,t) \coloneqq \tilde{\gamma}_1(t)+\omega \tilde{Z}_1(t),
\end{array}
\end{equation}
whose associated variational vector field is $ \tilde{Z}_1|_{[a,a_1]} $. Up to taking a smaller $ \varepsilon $, this variation $ \tilde{\Lambda}_1 $ can be lifted to a unique variation $ \Lambda_1 = (\Lambda_1^0,\tilde{\Lambda}_1): (-\varepsilon,\varepsilon) \times [a,a_1] \rightarrow  \varphi_1(\Omega_1) \subset \mathds{R}^{n+1} $ of $ \gamma|_{[a,a_1]} $ by lightlike curves. Indeed, let us consider a Lorentz-Finsler metric $L_1$ associated with the cone structure $\C^1$.  Writing $\tilde \gamma_{1\omega}(t):= \tilde \Lambda_1(\omega,t) $ and $\gamma_{1\omega}(t):= \Lambda_1(\omega,t) $, if every curve of the variation must be lightlike, then $ \gamma_{1\omega}^0(t) :=\Lambda_1^0(\omega,t) $ is determined by the differential equation
\begin{equation}
\label{eq:diff1}
L_1(\dot{\gamma}_{1\omega}(t)) = 0, \quad \forall t \in [a,a_1],
\end{equation}
with the initial conditions  $ \gamma_{1\omega}^0(a) = 0 $ and $\dot \gamma_{1\omega}^0(a)|_{\omega=0}=\dot\gamma^0(a)$. Moreover, we can also determine the initial condition $\dot \gamma_{1\omega}^0(a)$ for $\omega\not=0$ by lifting the vectors $\dot{\tilde \gamma}_{1\omega}(a)$ to lightlike vectors in a continuous way. Indeed, using \eqref{eq:fund_tensor_1} (in coordinates),
\begin{equation*}
\left. \frac{\partial L_1}{\partial y^0_1}(\dot{\gamma}_{1\omega}(t)) \right\rvert_{\omega=0} = \left. 2  g^{L_1}_{\dot{\gamma}_{1\omega}}(\dot{\gamma}_{1\omega}(t),\partial_{x^0_1}) \right\rvert_{\omega=0} \not= 0, \quad \forall t \in [a,a_1],
\end{equation*}
by item (vi). Therefore, by the implicit function theorem, we can solve the equation in coordinates $L_1((x^0,x^1,\ldots,x^n),(y^0,y^1,\ldots,y^n))=0$, where
$(x^0,x^1,\ldots,x^n),(y^0,y^1,\ldots,y^n)\in \mathds{R}\times \mathds{R}^n\times \mathds{R}\times \mathds{R}^n$, to get $ y^0 $ as a function of $(x^0,x^1,\ldots,x^n,y^1,\ldots,y^n)$ around $(\gamma^0_1(t),\tilde\gamma_1(t),\dot{\tilde\gamma}_1(t))$. In particular, taking $t=a$ we get the initial conditions  $\dot \gamma_{1\omega}^0(a)$ for $\omega$ small enough. Moreover, this shows that \eqref{eq:diff1} is a Picard type equation with a smooth dependence on $\omega$, and as a consequence, there is a unique maximal solution $\gamma^0_{1\omega}$ for every $\omega$ small enough. If there is any break $t_i$ on $[a,a_1]$, one can lift the curves until that point and then take as initial values for $ \gamma_{1\omega}^0(t_i)$, the obtained values with the variation of the previous smooth piece. 

It remains to prove that the variational vector field $ W $ associated with $ \Lambda_1 $ coincides with $ Z|_{[a,a_1]}$. Let $(W^0_1,\tilde W_1)$ be the coordinates of $W$ in $(\Omega_1,\varphi_1)$. Obviously, $ \tilde{W}_1 $ coincides with $\tilde{Z}_1 $ by construction and note also that  $ W^0_1(a) = Z^0_1(a) = 0 $ (due to the initial condition $ \gamma_{1\omega}^0(a) = 0 $). Now, since $ W $ is an admissible variational vector field, it satisfies \eqref{eq:lemma}, applying the first part of this proposition. Developing the first equation, using \eqref{eq:covariant}, we obtain
\begin{equation}
\label{eq:variational_field}
\begin{split}
0 & = g^{L_1}_{\dot{\gamma}}(\dot{\gamma},\leftindex^1{D}_{\gamma}^{\dot{\gamma}}W) = g^{L_1}_{\dot{\gamma}}(\dot{\gamma},\dot W^k_1\partial_{x^k_1}) + W^i_1 \dot{\gamma}^j_1 \Gamma^k_{\ ij}(\dot{\gamma}) g^{L_1}_{\dot{\gamma}}(\dot{\gamma},\partial_{x^k_1}) \\
& = \frac{1}{2}\frac{\partial L_1}{\partial y^i_1}(\dot{\gamma})\dot{W}^i_1+\frac{1}{2} W^i_1 \dot{\gamma}^j_1 \Gamma^k_{\ ij}(\dot{\gamma}) \frac{\partial L_1}{\partial y^k_1}(\dot{\gamma}), \quad \forall t \in [a,a_1],
\end{split}
\end{equation}
and taking into account the non-orthogonality condition
\begin{equation*}
\frac{\partial L_1}{\partial y^0_1}(\dot{\gamma}) =  2  g^{L_1}_{\dot{\gamma}}(\dot{\gamma},\partial_{x^0_1}) \not= 0, \quad \forall t \in [a,a_{1}],
\end{equation*}
we can explicitly solve \eqref{eq:variational_field} for $ \dot{W}^0_1 $ on $ [a,a_{1}] $. By hypothesis, $ Z $ also satisfies \eqref{eq:lemma} and thus, $ W^0_1|_{[a,a_1]} $ and $ Z^0_1|_{[a,a_{1}]} $  solve  the same differential equation with the same initial condition, so $ W^0_1|_{[a,a_{1}]} = Z^0_1|_{[a,a_{1}]} $  and hence $ W = Z $ on $[a,a_{1}]$.

Following the same steps, we can now construct another variation $ \tilde{\Lambda}_2 $ of the projection  $ \tilde{\gamma}_2|_{[a_1,a_2]} $ on $ \{x^0_2=0\} $ for the chart $(\Omega_2,\varphi_2)$, with $\gamma(a_2)\in \Omega_2\cap\Omega_3$ whenever $\gamma(\tau)\notin \Omega_2$ or $a_2=\tau$ in the opposite case, having $ \tilde{Z}_2|_{[a_1,a_2]} $ as the associated variational vector field. Then we can lift this variation $ \tilde{\Lambda}_2 $ to a variation by lightlike curves, using $\Lambda_1(\cdot,a_1)$ as initial values for each $\omega $ (in any case, $ a_1 $ can be regarded as a break).
 
Finally, we can proceed inductively with all the charts, taking $ \tau $ as the endpoint of the interval when $ \gamma(\tau) \in \Omega_{j}$. Note that the fact that $ \gamma $ is transverse to $ \eta $ ensures that all the variational curves $ \gamma_{\omega} $ constructed from \eqref{eq:base_variation} belong to $ \mathcal{N}_{P,\alpha} $ for sufficiently small $ \varepsilon $ (otherwise, if $ \gamma $ were tangent to $ \eta $, the curves $ \gamma_{\omega} $ might cross the interface more than once, failing to remain topologically transverse). Hence, by items (i), (iv), (v) and (vi), we obtain an admissible variation $\Lambda$ of $\gamma$ whose associated variational vector field, by construction, is $ Z $.
\end{proof}

Now we are ready to characterize the critical points of the arrival time functional $ \T $, stating the generalized Fermat's principle.

\begin{lemma}
\label{lem:Z(t_0)}
Let $ \gamma \in \mathcal{N}_{P,\alpha} $. Then, $ \gamma $ is a critical point of $ \T $ if and only if $ Z(b) = 0 $ for every admissible variational vector field $ Z \in \mathfrak{X}(\gamma) $.
\end{lemma}
\begin{proof}
Let $ Z \in \mathfrak{X}(\gamma) $ be any admissible variational vector field. Then
\begin{equation}
\label{eq:Z(t_0)}
Z(b) = \left. \frac{d}{d\omega} \gamma_{\omega}(b) \right\rvert_{\omega=0} = \left. \frac{d}{d\omega} \alpha(\T[\gamma_{\omega}]) \right\rvert_{\omega=0} = \left. \frac{d}{d\omega} \T[\gamma_{\omega}] \right\rvert_{\omega=0} \dot{\alpha}(\T[\gamma]),
\end{equation}
and since $ \alpha $ is embedded (so, in particular, it must be regular), we conclude that
\begin{equation*}
Z(b) = 0 \Leftrightarrow \left. \frac{d}{d\omega} \T[\gamma_{\omega}] \right\rvert_{\omega=0} = 0.
\end{equation*}
\end{proof}

\begin{thm}[Fermat's principle in cone structures with a jump discontinuity]
\label{thm:snell}
Let $ \gamma \in \mathcal{N}_{P,\alpha} $ such that $ \dot{\gamma}(b) \not\perp_{\C^2} \dot{\alpha}(\T[\gamma]) $:
\begin{itemize}
\item[(I)] If $ \dot{\gamma}(a) \perp_{\C^1} P $, $ \gamma|_{[a,\tau]} $ is a cone geodesic of $ \C^1 $, $ \gamma|_{[\tau,b]} $ is a cone geodesic of $ \C^2 $ and
\begin{equation}
\label{eq:snell}
\dot{\gamma}(\tau^-)^{\perp_{\C^1}} \cap T_{\gamma(\tau)}\eta \subset \dot{\gamma}(\tau^+)^{\perp_{\C^2}} \cap T_{\gamma(\tau)}\eta,
\end{equation}
then $ \gamma $ is a critical point of $ \T $.
\item[(II)] When $ \gamma $ is cone-transverse to $ \eta $,\footnote{Recall Definition~\ref{def:arrival_time} for the notion of cone-transversality.} the converse of \textup{(I)} holds with equality in \eqref{eq:snell}.
\end{itemize}
Moreover, the condition $ \dot{\gamma}(a) \perp_{\C^1} P $ imposes the following restrictions on $ P $ (and its codimension $ r \geq 1 $):\footnote{When $ P $ is a single point, $ \dot{\gamma}(a) \perp_{\C^1} P $ holds trivially and no restriction on $ P $ appears.}
\begin{itemize}
\item $ P $ cannot be $ \C^1 $-timelike at $ \gamma(a) $.
\item If $ P $ is $ \C^1 $-lightlike at $ \gamma(a) $, then $ \dot{\gamma}(a) \in T_{\gamma(a)}P $.
\item If $ P $ is $ \C^1 $-spacelike at $ \gamma(a) $, then $ r > 1 $.
\end{itemize}
\end{thm}
\begin{proof}
Let $ L_{\mu} $ be any Lorentz-Finsler metric compatible with $ \C^{\mu} $ ($ \mu = 1, 2 $). By Remark~\ref{rem:orth}(2) and Theorem~\ref{thm:cone_geodesics}, note that the hypotheses in (I) are equivalent to $ \dot{\gamma}(a) \perp_{L_1} P $, $ \gamma|_{[a,\tau]} $ is a lightlike pregeodesic of $ L_1 $, $ \gamma|_{[\tau,b]} $ is a lightlike pregeodesic of $ L_2 $ and
\begin{equation}
\label{eq:snell_L}
\dot{\gamma}(\tau^-)^{\perp_{L_1}} \cap T_{\gamma(\tau)}\eta \subset \dot{\gamma}(\tau^+)^{\perp_{L_2}} \cap T_{\gamma(\tau)}\eta,
\end{equation}
meaning that, for any $ z \in T_{\gamma(\tau)}\eta $, $ g^{L_1}_{\dot{\gamma}(\tau^-)}(\dot{\gamma}(\tau^-),z) = 0 $ implies $ g^{L_2}_{\dot{\gamma}(\tau^+)}(\dot{\gamma}(\tau^+),z) = 0 $.

Therefore, to prove (I) it is enough to see that these properties (in terms of $ L_{\mu} $) imply that $ \gamma $ is a critical point of $ \T $. Since this is independent of the parametrization of $ \gamma $ (recall Remark~\ref{rem:parametrization}(3)), we can assume, without loss of generality, that $ \gamma|_{[a,\tau]} $ and $ \gamma|_{[\tau,b]} $ are in fact geodesics of $ L_1 $ and $ L_2 $, respectively. Then, for any admissible variational vector field $ Z $:
\begin{equation*}
\begin{split}
0 & = \left. \frac{\partial}{\partial\omega} g^{L_{\mu}}_{\dot{\gamma}_{\omega}}(\dot{\gamma}_{\omega},\dot{\gamma}_{\omega}) \right\rvert_{\omega=0} = 2 g^{L_{\mu}}_{\dot{\gamma}}(\dot{\gamma},\leftindex^{\mu}{D}_{\gamma}^{\dot{\gamma}} Z) = \\
& = 2 \left( \frac{d}{dt} g^{L_{\mu}}_{\dot{\gamma}}(\dot{\gamma},Z) - g^{L_{\mu}}_{\dot{\gamma}}(\leftindex^{\mu}{D}_{\gamma}^{\dot{\gamma}} \dot{\gamma},Z) \right) = 2 \frac{d}{dt} g^{L_{\mu}}_{\dot{\gamma}}(\dot{\gamma},Z), \quad \mu=1,2,
\end{split}
\end{equation*}
where we have used \eqref{eq:g_compat} and, in the last step, $ \leftindex^{\mu}{D}_{\gamma}^{\dot{\gamma}} \dot{\gamma} = 0 $ because $ \gamma $ is a geodesic of $ L_{\mu} $. Therefore,
\begin{equation}
\label{eq:constant}
\begin{split}
g^{L_1}_{\dot{\gamma}}(\dot{\gamma},Z) & = c_1 \in \mathds{R}, \quad \forall t \in [a,\tau], \\
g^{L_2}_{\dot{\gamma}}(\dot{\gamma},Z) & = c_2 \in \mathds{R}, \quad \forall t \in [\tau,b].
\end{split}
\end{equation}
Now, observe that $ Z(a) \in T_{\gamma(a)}P $ (by Proposition~\ref{prop:variational_field}(I)) and $ \dot{\gamma}(a) \perp_{L_1} P $, so $ \dot{\gamma}(a) \perp_{L_1} Z(a) $ and thus, $ c_1 = 0 $. Also, since $ Z(\tau) \in T_{\gamma(\tau)}\eta $ (again by Proposition~\ref{prop:variational_field}(I)), we can apply \eqref{eq:snell_L} to get
\begin{equation*}
0 = g^{L_1}_{\dot{\gamma}(\tau^-)}(\dot{\gamma}(\tau^-),Z(\tau)) = g^{L_2}_{\dot{\gamma}(\tau^+)}(\dot{\gamma}(\tau^+),Z(\tau))
\end{equation*}
and hence, $ c_2 = 0 $. In particular, $ g^{L_2}_{\dot{\gamma}(b)}(\dot{\gamma}(b),Z(b)) = 0 $. Using \eqref{eq:Z(t_0)}, we obtain
\begin{equation*}
\left. \frac{d}{d\omega} \T[\gamma_{\omega}] \right\rvert_{\omega=0} g^{L_2}_{\dot{\gamma}(b)}(\dot{\gamma}(b),\dot{\alpha}(\T[\gamma])) = 0,
\end{equation*}
and since $ \dot{\gamma}(b) \not\perp_{L_2} \dot{\alpha}(\T[\gamma]) $ (because $ \dot{\gamma}(b) \not\perp_{\C^2} \dot{\alpha}(\T[\gamma]) $ by hypothesis), we conclude that $ \gamma $ is a critical point of $ \T $.

To prove (II), assume now that $ \gamma $ is cone-transverse to $ \eta $ and a critical point of $ \T $, with breaks at $ \{t_i\}_{i=1}^k \subset (a,b) $. By Lemma~\ref{lem:Z(t_0)}, this means that $ Z(b) = 0 $ for every admissible variational vector field $ Z $. Now, fix an auxiliary vector field $ U \in \mathfrak{X}(\gamma) $ such that $ U(\tau) \in T_{\gamma(\tau)}\eta $, $ U(b) $ is proportional to $ \dot{\alpha}(\T[\gamma]) $ and $ \dot{\gamma} \not\perp_{L_{\mu}} U $ everywhere (such a vector field exists thanks to the hypotheses $ \dot{\gamma}(b) \not\perp_{L_2} \dot{\alpha}(\T[\gamma]) $ and the cone-transversality of $ \gamma $, which implies $ \dot{\gamma}(\tau^-), \dot{\gamma}(\tau^+) \not\perp_{L_{\mu}} T_{\gamma(\tau)}\eta $; recall Remark~\ref{rem:c_transverse}),\footnote{To obtain explicitly the vector field $U$, first observe that we can make a continuous choice of one of the two half-spaces determined by $\dot\gamma(t)^{\perp_{\C^{\mu}}}$. By choosing local representatives for $U$ in small intervals within the selected half-spaces we can then use a partition of unity of $[a,b]$ to obtain a global smooth choice for $U$.} and consider the following set of vector fields along $ \gamma $:
\begin{equation*}
\mathcal{V} \coloneqq \{ W \in \mathfrak{X}(\gamma): W(a) \in T_{\gamma(a)}P, W(\tau) \in T_{\gamma(\tau)}\eta \text{ and } W(b) = 0 \}.
\end{equation*}
Given any $ W \in \mathcal{V} $, we define
\begin{equation}
\label{eq:Z_W}
Z_W(t) \coloneqq \left\lbrace
\begin{array}{l}
W(t) + f_1(t) U(t), \quad t \in [a,\tau], \\
W(t) + f_2(t) U(t), \quad t \in (\tau,b],
\end{array} \right.
\end{equation}
where $ f_1 $ and $ f_2 $ are chosen so that they satisfy the differential equation given by \eqref{eq:lemma}:
\begin{equation*}
0 = g^{L_{\mu}}_{\dot{\gamma}}(\dot{\gamma},\leftindex^{\mu}{D}_{\gamma}^{\dot{\gamma}} Z_W) = g^{L_{\mu}}_{\dot{\gamma}}(\dot{\gamma},\leftindex^{\mu}{D}_{\gamma}^{\dot{\gamma}} W) + g^{L_{\mu}}_{\dot{\gamma}}(\dot{\gamma},U) \dot{f}_{\mu} + g^{L_{\mu}}_{\dot{\gamma}}(\dot{\gamma},\leftindex^{\mu}{D}_{\gamma}^{\dot{\gamma}} U) f_{\mu},
\end{equation*}
with the additional conditions $ f_1(a) = 0 $ and $ f_1(\tau) = f_2(\tau) $. Namely:
\begin{equation}
\label{eq:f}
\begin{split}
f_1(t) & \coloneqq - e^{-\rho_1(t)} \int_a^t \frac{g^{L_1}_{\dot{\gamma}}(\dot{\gamma},\leftindex^1{D}_{\gamma}^{\dot{\gamma}}W)}{g^{L_1}_{\dot{\gamma}}(\dot{\gamma},U)} e^{\rho_1} ds, \\
f_2(t) & \coloneqq - e^{-\rho_2(t)} \left( \int_a^{\tau} \frac{g^{L_1}_{\dot{\gamma}}(\dot{\gamma},\leftindex^1{D}_{\gamma}^{\dot{\gamma}}W)}{g^{L_1}_{\dot{\gamma}}(\dot{\gamma},U)} e^{\rho_1} ds + \int_{\tau}^t \frac{g^{L_2}_{\dot{\gamma}}(\dot{\gamma},\leftindex^2{D}_{\gamma}^{\dot{\gamma}}W)}{g^{L_2}_{\dot{\gamma}}(\dot{\gamma},U)} e^{\rho_2} ds \right),
\end{split}
\end{equation}
with
\begin{equation*}
\rho_1(t) \coloneqq \int_a^t \frac{g^{L_1}_{\dot{\gamma}}(\dot{\gamma},\leftindex^1{D}_{\gamma}^{\dot{\gamma}}U)}{g^{L_1}_{\dot{\gamma}}(\dot{\gamma},U)} ds, \quad  \rho_2(t) \coloneqq \rho_1(\tau) + \int_{\tau}^t \frac{g^{L_2}_{\dot{\gamma}}(\dot{\gamma},\leftindex^2{D}_{\gamma}^{\dot{\gamma}}U)}{g^{L_2}_{\dot{\gamma}}(\dot{\gamma},U)} ds.
\end{equation*}
This way, $ Z_W $ satisfies all the hypotheses in Proposition~\ref{prop:variational_field}(II). When $\gamma$ is not tangent to $\eta$, this guarantees that $Z_W$ is admissible for all $W\in\mathcal{V}$. Otherwise, we can generate enough admissible variational vector fields $Z_W$ for our purposes. First, we can assume that in a certain chart of $\gamma(\tau)$, the vector field $U$ along $\gamma$ is parallel to $\eta$ (with respect to the chosen coordinates). Then, for every vector field $W\in\mathcal{V}$ such that, close to $\eta$, it is either parallel to $ \eta $ or points out to the opposite side of $\eta$, the vector field $Z_W$ is admissible by a direct application of \eqref{eq:base_variation}. In particular, if $W$ is constantly equal to a vector $u\in T_{\gamma(\tau)}\eta$ close to $\tau$ (in particular if it is zero), then $Z_W$ is admissible. Let us denote
\begin{equation}\label{tildeV}
\tilde{\mathcal{V}} \coloneqq \{W\in\mathcal{V}: \text{$Z_W$ is admissible} \}.
\end{equation}
As said above, $\tilde{\mathcal{V}}$ includes all $W\in\mathcal{V}$ which in the chosen chart are parallel to $\eta$ or identically zero. As $ \gamma $ is a critical point of $ \T $, $ Z_W(b) = 0 $ by Lemma~\ref{lem:Z(t_0)} and thus, $ f_2(b) = 0 $, i.e.
\begin{equation}
\label{eq:integrals}
\int_a^{\tau} \frac{g^{L_1}_{\dot{\gamma}}(\dot{\gamma},\leftindex^1{D}_{\gamma}^{\dot{\gamma}}W)}{g^{L_1}_{\dot{\gamma}}(\dot{\gamma},U)} e^{\rho_1} ds + \int_{\tau}^{b} \frac{g^{L_2}_{\dot{\gamma}}(\dot{\gamma},\leftindex^2{D}_{\gamma}^{\dot{\gamma}}W)}{g^{L_2}_{\dot{\gamma}}(\dot{\gamma},U)} e^{\rho_2} ds = 0.
\end{equation}
Choosing $ W\in\tilde{\mathcal{V}} $ so that $ W(t_i) = 0 $ at any break $ t_i \not= \tau $ and applying integration by parts (using also that $ W(b) = 0 $), \eqref{eq:integrals} becomes
\begin{equation*}
\begin{split}
0 & = -\int_a^{\tau} g^{L_1}_{\dot{\gamma}}(\leftindex^1{D}_{\gamma}^{\dot{\gamma}}(\dot{\gamma} \phi_1),W) ds - \int_{\tau}^{b} g^{L_2}_{\dot{\gamma}}(\leftindex^2{D}_{\gamma}^{\dot{\gamma}}(\dot{\gamma} \phi_2),W) ds \\
& + g^{L_1}_{\dot{\gamma}(\tau^-)}(\dot{\gamma}(\tau^-),W(\tau)) \phi_1(\tau^-) - g^{L_2}_{\dot{\gamma}(\tau^+)}(\dot{\gamma}(\tau^+),W(\tau)) \phi_2(\tau^+) \\
& - g^{L_1}_{\dot{\gamma}(a)}(\dot{\gamma}(a),W(a)) \phi_1(a),
\end{split}
\end{equation*}
where $ \phi_{\mu} \coloneqq e^{\rho_{\mu}}/g^{L_{\mu}}_{\dot{\gamma}}(\dot{\gamma},U) $ for $ \mu=1,2 $. Since $ W $ is arbitrary outside the breaks and endpoints (indeed, as said above, it is enough to choose it as zero close to $\gamma(\tau)$), we obtain the following conditions:
\begin{equation}
\label{eq:conditions}
\begin{split}
0 & = \leftindex^1{D}_{\gamma}^{\dot{\gamma}}(\dot{\gamma} \phi_1), \quad \forall t \in (a,\tau) \setminus \{t_i\}_{i=1}^k, \\
0 & = \leftindex^2{D}_{\gamma}^{\dot{\gamma}}(\dot{\gamma} \phi_2), \quad \forall t \in (\tau,b) \setminus \{t_i\}_{i=1}^k, \\
0 & = g^{L_1}_{\dot{\gamma}(a)}(\dot{\gamma}(a),W(a)) \phi_1(a), \quad \forall W \in \tilde{\mathcal{V}},\\
0 & = g^{L_1}_{\dot{\gamma}(\tau^-)}(\dot{\gamma}(\tau^-),W(\tau)) \phi_1(\tau^-) \\
& - g^{L_2}_{\dot{\gamma}(\tau^+)}(\dot{\gamma}(\tau^+),W(\tau)) \phi_2(\tau^+), \quad \forall W \in \tilde{\mathcal{V}}.
\end{split}
\end{equation}
The first two equations imply that $ \gamma|_{[a,\tau]} $ and $ \gamma|_{[\tau,b]} $ are piecewise pregeodesics of $ L_1 $ and $ L_2 $, respectively. Indeed, using \eqref{eq:covariant} we get
\begin{equation*}
0 = \leftindex^{\mu}{D}_{\gamma}^{\dot{\gamma}}(\dot{\gamma} \phi_{\mu}) \Leftrightarrow \leftindex^{\mu}{D}_{\gamma}^{\dot{\gamma}}\dot{\gamma} = -\frac{\dot{\phi}_{\mu}}{\phi_{\mu}} \dot{\gamma}, \quad \mu = 1,2,
\end{equation*}
and since $ \phi_{\mu} \not= 0 $, Proposition~\ref{prop:pregeodesic} can be applied on each smooth piece of $ \gamma $. Moreover, for any break $ t_i \not= \tau $, we can choose $ W \in \tilde{\mathcal{V}} $ so that $ W(t_i) \not= 0 $ and $ W $ vanishes at the other breaks. Then, we can apply again integration by parts in \eqref{eq:integrals} and use \eqref{eq:conditions} to obtain the condition
\begin{equation}
\label{eq:breaks}
0 = g^{L_{\mu}}_{\dot{\gamma}(t_i^-)}(\dot{\gamma}(t_i^-),W(t_i)) \phi_{\mu}(t_i^-) - g^{L_{\mu}}_{\dot{\gamma}(t_i^+)}(\dot{\gamma}(t_i^+),W(t_i)) \phi_{\mu}(t_i^+),
\end{equation}
which must hold for any arbitrary $ W(t_i) \in T_{\gamma(t_i)}Q_{\mu} $. In particular, for $ W(t_i) = \dot{\gamma}(t_i^-) $, \eqref{eq:breaks} yields $ \dot{\gamma}(t_i^+) \perp_{L_{\mu}} \dot{\gamma}(t_i^-) $, which means that $ \dot{\gamma}(t_i^+) $ and $ \dot{\gamma}(t_i^-) $ must be proportional to each other (recall Remark~\ref{rem:orth}(3)). Therefore, $ \gamma|_{[a,\tau]} $ and $ \gamma|_{[\tau,b]} $ can be reparametrized as smooth pregeodesics of $ L_1 $ and $ L_2 $, respectively. The third equation in \eqref{eq:conditions} directly implies that $ \dot{\gamma}(a) \perp_{L_1} P $, since $ \phi_1(a) \not= 0 $ and we can freely choose any $ W(a) \in T_{\gamma(a)}P $. Finally, it only remains to prove that \eqref{eq:snell} holds, which is independent of the (positive) parametrization of $ \gamma $. So, we can assume, without loss of generality, that $ \gamma|_{[a,\tau]} $ and $ \gamma|_{[\tau,b]} $ are geodesics of $ L_1 $ and $ L_2 $, respectively. Using \eqref{eq:Z_W}, the last equation in \eqref{eq:conditions} written in terms of $ Z_W $ becomes
\begin{equation}
\label{eq:snell_Z_W}
\begin{split}
0 & = g^{L_1}_{\dot{\gamma}(\tau^-)}(\dot{\gamma}(\tau^-),Z_W(\tau)) \phi_1(\tau^-) - e^{\rho_1(\tau)} f_1(\tau) \\
& - g^{L_2}_{\dot{\gamma}(\tau^+)}(\dot{\gamma}(\tau^+),Z_W(\tau)) \phi_2(\tau^+) + e^{\rho_2(\tau)} f_2(\tau), \quad \forall W \in \tilde{\mathcal{V}}.
\end{split}
\end{equation}
From \eqref{eq:f}, observe that $ e^{\rho_1(\tau)} f_1(\tau) = e^{\rho_2(\tau)} f_2(\tau) $. Moreover, since $ \gamma $ is a geodesic and $ Z_W $ is admissible, \eqref{eq:constant} holds (with $ c_1 = 0 $ because $ Z_W(a) \in T_{\gamma(a)}P $) and therefore
\begin{equation}
\label{eq:g1=0}
g^{L_1}_{\dot{\gamma}(\tau^-)}(\dot{\gamma}(\tau^-),Z_W(\tau)) = 0, \quad \forall W \in \tilde{\mathcal{V}},
\end{equation}
which reduces \eqref{eq:snell_Z_W} to the condition
\begin{equation}
\label{eq:g2=0}
g^{L_2}_{\dot{\gamma}(\tau^+)}(\dot{\gamma}(\tau^+),Z_W(\tau)) = 0, \quad \forall W \in \tilde{\mathcal{V}}.
\end{equation}
Now, for any $ W \in \tilde{\mathcal{V}} $, we have from \eqref{eq:Z_W}:
\begin{equation*}
W(\tau) = Z_W(\tau) - f_1(\tau) U(\tau),
\end{equation*}
and since $ \dot{\gamma} \not\perp_{L_{\mu}} U $ everywhere, by \eqref{eq:g1=0} we deduce that $ \dot{\gamma}(\tau^-) \perp_{L_1} W(\tau) $ if and only if $ f_1(\tau) = 0 $ and thus, because $ f_1(\tau) = f_2(\tau) $, if and only if $ \dot{\gamma}(\tau^+) \perp_{L_2} W(\tau) $ by \eqref{eq:g2=0}. But notice that
\begin{equation*}
T_{\gamma(\tau)}\eta = \lbrace W(\tau): W \in \tilde{\mathcal{V}} \rbrace,
\end{equation*}
so in fact, for any $ z \in T_{\gamma(\tau)}\eta $, $ \dot{\gamma}(\tau^-) \perp_{L_1} z $ if and only if $ \dot{\gamma}(\tau^+) \perp_{L_2} z $, which is equivalent to \eqref{eq:snell_L} with equality:
\begin{equation}
\label{eq:snell_equality}
\dot{\gamma}(\tau^-)^{\perp_{L_1}} \cap T_{\gamma(\tau)}\eta = \dot{\gamma}(\tau^+)^{\perp_{L_2}} \cap T_{\gamma(\tau)}\eta.
\end{equation}

For the last assertion, observe that the condition $ \dot{\gamma}(a) \perp_{\C^1} P $ imposes some restrictions on the causal character of $ P $:
\begin{itemize}
\item If $ P $ is $ \C^1 $-timelike at $ \gamma(a) $ (i.e., there is at least one tangent vector to $ P $ which is $\C^1 $-timelike), then $ \gamma $ cannot be critical because there are no $ \C^1 $-lightlike vectors $ \C^1 $-orthogonal to $ P $ at $ \gamma(a) $.
\item If $ P $ is $ \C^1 $-lightlike at $ \gamma(a) $ (i.e., $ T_{\gamma(a)}P $ is tangent to $ \C^1_{\gamma(a)} $), then $ \dot{\gamma}(a) \perp_{\C^1} P $ is equivalent to $ \dot{\gamma}(a) \in T_{\gamma(a)}P $.
\item If $ P $ is $ \C^1 $-spacelike at $ \gamma(a) $ (i.e., any tangent vector to $ P $ is $ \C^1 $-spacelike), then necessarily $ r > 1 $ for $ \gamma $ to be critical. Otherwise, there are no $ \C^1 $-lightlike vectors $ \C^1 $-orthogonal to $ P $ at $ \gamma(a) $.
\end{itemize}
\end{proof}

\begin{rem}[Snell's law]
Equation~\eqref{eq:snell} describes the change in direction of an incident trajectory as it crosses the interface $ \eta $. Thus, it can be interpreted as a generalized version of {\em Snell's law of refraction} for cone structures. An equivalent formulation of \eqref{eq:snell} is that $ \dot{\gamma}(\tau^+) $ must be $ \C^2 $-orthogonal to $ \dot{\gamma}(\tau^-)^{\perp_{L_1}} \cap T_{\gamma(\tau)}\eta $. In coordinates, taking into account that $ g^{L_{\mu}}_v(v,z) = \frac{1}{2} \frac{\partial L_{\mu}}{\partial y^i}(v)z^i $ (from \eqref{eq:fund_tensor_1}), this law is equivalent to the following condition for $ \dot{\gamma}(\tau^{\pm}) $:
\begin{equation}
\label{eq:snell_coord}
\frac{\partial L_1}{\partial y^i}(\dot{\gamma}(\tau^-))z^i = 0 \Rightarrow \frac{\partial L_2}{\partial y^i}(\dot{\gamma}(\tau^+))z^i = 0, \quad \forall z \equiv (z^0,\ldots,z^n) \in T_{\gamma(\tau)}\eta,
\end{equation}
for any Lorentz-Finsler metric $ L_{\mu} $ compatible with $ \C^{\mu} $.
\end{rem}

\begin{rem}
\label{rem:restrictions}
In order to establish Theorem~\ref{thm:snell}, we have explicitly imposed two restrictions on the base curve $ \gamma $:
\begin{enumerate}
\item The restriction $ \dot{\gamma}(b) \not\perp_{\C^2} \dot{\alpha}(\T[\gamma]) $ (which automatically holds if $ \alpha $ is $ \C^2 $-timelike) comes from Proposition~\ref{prop:variational_field}(II) because the standard procedure used there to construct admissible variations of $ \gamma $ does not work without this non-orthogonality condition. To illustrate that, in general, Theorem~\ref{thm:snell} is not valid if this restriction is removed, consider the following counterexamples:
\begin{itemize}
\item[$ \circ $] Lightlike $ \alpha $: assume that $ \gamma $, being a cone geodesic in each cone structure and satisfying Snell's law \eqref{eq:snell}, coincides with $ \alpha $ along the interval $ [b-\varepsilon,b] $, for some $ \varepsilon > 0 $. In this case, we can easily arrange $ \alpha $ so that there are variations that trivially reduce or increase the arrival time functional, making $ \gamma $ unable to be critical.
\item[$ \circ $] Spacelike $ \alpha $: in the standard Minkowski spacetime (with $ x^0 $ as the time coordinate), let $ P = \{(0,0,0)\} $ and $ \alpha(s) = (1,\cos{s},\sin{s}) $, with $ s \in (0,\pi) $---i.e., $ \alpha $ corresponds to half of the intersection between the hyperplane $ \{x^0=1\} $ and the lightcone from the origin. In this case, all the lightlike geodesics from $ P $ to $ \alpha $ are orthogonal to $ \alpha $, but none of them is critical, since these trajectories themselves define an admissible variation with nonzero derivative (when composed with the arrival time functional).
\end{itemize}
Obviously, this constraint is natural even in the Lorentzian case (see \cite[Theorem~7.4]{CJS14}), although there are some particular situations where it can be dropped (see \cite[Theorem~7.8]{CJS14}).

\item The restriction that $\gamma$ be cone-transverse to $ \eta $ is explicitly used in Theorem~\ref{thm:snell}(II) to construct the auxiliary vector field $ U $, which is an essential part of the proof. In any case, this restriction only applies when $ T_{\gamma(\tau)}\eta $ is $ \C^1 $- or $ \C^2 $-lightlike, as it automatically holds for any other combination of causal characters of $ T_{\gamma(\tau)}\eta $ for $ \C^1 $ and $ \C^2 $.
\end{enumerate}

The framework also involves some implicit restrictions, because the assumption $ \gamma \in \mathcal{N}_{P,\alpha} $ implies the existence of a lightlike curve connecting $ \alpha $ to $ P $ and crossing the interface at $ \gamma(\tau) $. Consequently, $ \alpha $ and $ P $ must be causally related and neither $ \C^1_{\gamma(\tau)} $ nor $ \C^2_{\gamma(\tau)} $ can point entirely toward $ Q_1 $.
\end{rem}

\subsection{The law of reflection}
\label{sec:reflection}
So far, we have focused on trajectories that pass from one medium to another. Now we turn our attention to those that bounce off the interface, remaining in one medium. Hence, keeping the same general setting as in the beginning of \S~\ref{sec:setting}, we first define the set of curves that correspond to these new trajectories (compare with Definition~\ref{def:arrival_time}).

\begin{defi}
\label{def:arrival_time_reflection}
Let $ P \subset Q_1 $ be a $ \C^1 $-spacelike submanifold and $ \alpha: I \rightarrow Q_1 $ an embedded smooth curve, where $ I \subset \mathds{R} $ is an open interval. Fixing an interval $ [a,b] $ and some $ \tau \in (a,b) $, let $ \mathcal{N}^*_{P,\alpha} $ be the set of all (regular) piecewise smooth curves $ \gamma: [a,b] \rightarrow \overline{Q}_1 $ such that:
\begin{itemize}
\item $ \gamma(a) \in P $, $ \gamma(\tau) \in \eta $ and $ \gamma(b) \in \textup{Im}(\alpha) $.
\item $ \gamma $ is topologically transverse to $ \eta $ and $ \C^1 $-lightlike.
\end{itemize}
Then $ \T^*: \mathcal{N}^*_{P,\alpha} \rightarrow I $, defined as in \eqref{eq:arrival_time}, is the {\em arrival time functional via reflection}. In addition, we say that $ \gamma \in \mathcal{N}^*_{P,\alpha} $ is {\em cone-transverse} to $ \eta $ if both $ \dot{\gamma}(\tau^-) $ and $ \dot{\gamma}(\tau^+) $ are $ \C^1 $-transverse to $ \eta $.
\end{defi}

\begin{rem}
\label{rem:double}
Notice that $ Q_2 $ does not play any role here. So, instead of considering $ Q $ as $ \overline{Q}_1 \cup \overline{Q}_2 $, we can reduce the study of the critical points of $ \T^* $ to the previous study of $ \T $ by considering the {\em double manifold} of $ \overline{Q}_1 $, given by the adjunction space $ \overline{Q}_1 \cup_{\textup{Id}} \overline{Q}_1 $, where $ \textup{Id}: \eta \rightarrow \eta $ is the identity map of $ \eta = \partial Q_1 $ (see \cite[Theorem~9.29 and Example~9.32]{L13}). Namely, we attach two copies of $ \overline{Q}_1 $ along their boundaries, identifying each boundary point in one copy with the same boundary point in the other. Each copy of $ Q_1 $ is endowed with the same cone structure $ \C^1 $, although the cones do not need to match continuously on $ \eta $: the two cones at each interface point are symmetric with respect to the tangent plane to $ \eta $.\footnote{Intuitively, the two copies of $ \overline{Q}_1 $ are mirrored with respect to each other.}

With this construction, observe that the set $ \mathcal{N}^*_{P,\alpha} $ in $ \overline{Q}_1 \cup \overline{Q}_2 $ can be identified with $ \mathcal{N}_{P,\alpha^*} $ in $ \overline{Q}_1 \cup_{\textup{Id}} \overline{Q}_1 $, simply considering the copy $ \alpha^* $ of the arrival curve $ \alpha $ in the new $ Q_1 $ that substitutes $ Q_2 $ (so that the reflected trajectories in $ \overline{Q}_1 \cup \overline{Q}_2 $ are identified with refracted ones in $ \overline{Q}_1 \cup_{\textup{Id}} \overline{Q}_1 $). Under this identification, $ \T^*: \mathcal{N}^*_{P,\alpha} \rightarrow I $ in the original setting $ \overline{Q}_1 \cup \overline{Q}_2 $ coincides with $ \T: \mathcal{N}_{P,\alpha^*} \rightarrow I $ in this new setting $ \overline{Q}_1 \cup_{\textup{Id}} \overline{Q}_1 $.
\end{rem}

Taking this observation into account, Theorem~\ref{thm:snell} directly becomes the following analogous result.

\begin{cor}[Fermat's principle for reflection]
\label{thm:reflection}
Let $ \gamma \in \mathcal{N}^*_{P,\alpha} $ such that $ \dot{\gamma}(b) \not\perp_{\C^1} \dot{\alpha}(\T^*[\gamma]) $:
\begin{itemize}
\item[(I)] If $ \dot{\gamma}(a) \perp_{\C^1} P $, $ \gamma|_{[a,\tau]} $ and $ \gamma|_{[\tau,b]} $ are cone geodesics of $ \C^1 $, and
\begin{equation}
\label{eq:reflection}
\dot{\gamma}(\tau^-)^{\perp_{\C^1}} \cap T_{\gamma(\tau)}\eta \subset \dot{\gamma}(\tau^+)^{\perp_{\C^1}} \cap T_{\gamma(\tau)}\eta,
\end{equation}
then $ \gamma $ is a critical point of $ \T^* $.
\item[(II)] When $ \gamma $ is cone-transverse to $ \eta $, the converse of (I) holds with equality in \eqref{eq:reflection}.
\end{itemize}
\end{cor}

\begin{rem}[Law of reflection]
Equation~\eqref{eq:reflection} describes the change in direction of a geodesic reflected at the interface $ \eta $. Thus, it can be interpreted as a generalized version of the {\em law of reflection} for cone structures. An equivalent formulation of \eqref{eq:reflection} is that $ \dot{\gamma}(\tau^+) $ must be $ \C^1 $-orthogonal to $ \dot{\gamma}(\tau^-)^{\perp_{L_1}} \cap T_{\gamma(\tau)}\eta $. In coordinates, this law is equivalent to the following condition for $ \dot{\gamma}(\tau^{\pm}) $:
\begin{equation*}
\frac{\partial L_1}{\partial y^i}(\dot{\gamma}(\tau^-))z^i = 0 \Rightarrow \frac{\partial L_1}{\partial y^i}(\dot{\gamma}(\tau^+))z^i = 0, \quad \forall z \equiv (z^0,\ldots,z^n) \in T_{\gamma(\tau)}\eta,
\end{equation*}
for any Lorentz-Finsler metric $ L_1 $ compatible with $ \C^1 $.
\end{rem}

\begin{rem}
A few relevant comments regarding Corollary~\ref{thm:reflection} are in order:
\begin{itemize}
\item We stress that $ \mathcal{N}^*_{P,\alpha} $ is a quite restrictive set, where the curves are forced to touch the interface. Critical points of the arrival time in this set will not be so, in general, in broader spaces such as the one containing all the $ \C^1 $-causal curves going from $ P $ to $ \alpha $ (see the discussion in \cite[\S~5.2]{MP23}).
\item The condition $ \dot{\gamma}(a) \perp_{\C^1} P $ imposes the same restrictions on $ P $ as in Theorem~\ref{thm:snell}.
\item Even though the wording of the result is completely analogous to that of Theorem~\ref{thm:snell}, we can already anticipate that both phenomena---the refraction and the reflection---are not actually symmetric. For instance, reflected trajectories cannot exist when $ \eta $ is $ \C^1 $-spacelike, whereas refracted trajectories can appear for any causal character of $ \eta $ with respect to $ \C^2 $. This different behavior, depending on the causal character of $ \eta $, will be addressed in detail next in \S~\ref{sec:existence} (compare Theorem~\ref{thm:existence_refraction} and Corollary~\ref{thm:existence_reflection} below).
\end{itemize}
\end{rem}

\section{Existence and uniqueness of solutions}
\label{sec:existence}
When the cone structures $ \C^1, \C^2 $ represent the infinitesimal propagation of a certain wave in two different media, one usually seeks to determine the wave trajectory departing from a given initial point and direction (see e.g. \cite{JPS21}). Such a trajectory $ \gamma: [a,b] \rightarrow Q $ must satisfy Fermat's principle, i.e. it must be a critical point of the traveltime from $ \gamma(a) $ to any (non-orthogonal) arrival curve $ \alpha $. In principle, this means that $ \gamma $ would be determined by following these steps:
\begin{enumerate}
\item Fix two Lorentz-Finsler metrics $ L_1, L_2 $ compatible with $ \C^1, \C^2 $, respectively. Usually, this choice is made so that the $ L_\mu $-unit vectors represent physical elements related to the measurement of the wave velocities (as in \S~\ref{sec:natural} below).
\item Given $ \gamma(a) $ and $ \dot{\gamma}(a) $, solve \eqref{eq:geod_eq} to obtain the unique geodesic of $ L_1 $ with these initial conditions. Stop the computation once it reaches $ \eta $ at some $ t = \tau $. This determines the incident trajectory $ \gamma|_{[a,\tau]} $, which fixes $ \gamma(\tau) $ and $ \dot{\gamma}(\tau^-) $.
\item Knowing $ \gamma(\tau) $ and $ \dot{\gamma}(\tau^-) $, solve Snell's law \eqref{eq:snell} and/or the law of reflection \eqref{eq:reflection} for $ \dot{\gamma}(\tau^+) $, with the restriction that $ \dot{\gamma}(\tau^+) $ has to point to the $ Q_2 $-side of $ T_{\gamma(\tau)}\eta $ when looking for the refracted trajectory, and to the $ Q_1 $-side when looking for the reflected one (in particular, $ \dot{\gamma}(\tau^+) $ can be tangent to $ \eta $ in both cases). However, we will see in this section that the existence and uniqueness of $ \dot{\gamma}(\tau^+) $ are not guaranteed.
\item Finally, given $ \dot{\gamma}(\tau^+) $, solve \eqref{eq:geod_eq} again to obtain the unique geodesic of $ L_{\mu} $ ($ \mu = 1,2 $ depending on whether we are looking for a reflected or refracted trajectory, respectively) departing from $ \gamma(\tau) $ with initial velocity $ \dot{\gamma}(\tau^+) $. If $ \dot{\gamma}(\tau^+) $ is tangent to $ \eta $, then $ \gamma|_{[\tau,b]} $ might not exist, depending on the shape of $ \eta $. If the resulting trajectory intersects $ \eta $ more than once at isolated points, apply again the previous steps (recall Remark~\ref{rem:parametrization}(1)).
\end{enumerate}

Observe that steps (2) and (4) are straightforward (at least computationally), as they involve solving an ordinary differential equation system with fixed initial conditions. Step (3), on the other hand, deserves a closer look. The aim of this section is, therefore, to determine the cases where we can ensure the existence of refraction and reflection, and whether they are unique. To this end, we will work in the tangent space $ T_p\eta $, for some arbitrary $ p = \gamma(\tau) \in \eta $, and we will talk about directions instead of velocities (vectors), since the cone structure (see Definition~\ref{def:cone_structure}), the causality (see Definition~\ref{def:causality}), the orthogonality and the transversality (see \S~\ref{subsec:orthogonality}) are preserved under positive vector rescalings---hence, so are the solutions to Snell's law and the law of reflection.

\begin{notation}
\label{not:subspaces}
Given any $ p \in \eta $ and $ v \in T_pQ $, by {\em direction} $ v $ we mean the half-ray $ \{\lambda v: \lambda > 0 \} \subset T_pQ $. Also, we will denote by $ Q^1_p, Q^2_p \subset T_pQ $ the two open subspaces of $ T_pQ $ divided by $ T_p\eta $, each one lying on the corresponding side of $ T_p\eta $ (i.e., $ Q^{\mu}_p $ corresponds to the $ Q_{\mu} $-side of $ T_p\eta $).
\end{notation}

\subsection{Initial definitions and results}
\label{subsec:def_res}
First of all, let us start with two lemmas that will prove very useful later. Keep in mind that, for any direction $ v \in \C_p^{\mu} $, $ v^{\perp_{\C^{\mu}}} = T_v\C_p^{\mu} $ (recall Definition~\ref{def:orth}).

\begin{lemma}
\label{lem:Pi}
Let $ \Pi \subset T_p\eta $, with $ p \in \eta $, be a linear subspace of $ \dim(\Pi) = n-1 $. Then:
\begin{itemize}
\item[(i)] If $ \Pi $ is $ \C^{\mu} $-spacelike, there are exactly two directions $ v, w \in \C^{\mu}_p $ such that $ v, w \perp_{\C^{\mu}} \Pi $. In this case, $ T_p\eta $ can have any causal character. Moreover:
\begin{itemize}
\item[$ \circ $] If $ T_p\eta $ is $ \C^{\mu} $-timelike, then either $ v \in Q^1_p $ and $ w \in Q^2_p $, or vice versa.
\item[$ \circ $] If $ T_p\eta $ is $ \C^{\mu} $-lightlike, then either $ v \in T_p\eta $ and $ w \notin T_p\eta $, or vice versa.
\item[$ \circ $] If $ T_p\eta $ is $ \C^{\mu} $-spacelike, then either $ v, w \in Q^1_p $ or $ v, w \in Q^2_p $.
\end{itemize}
\item[(ii)] If $ \Pi $ is $ \C^{\mu} $-lightlike, there exists a unique direction $ v \in \C^{\mu}_p $, necessarily in $ \Pi $ (and thus, tangent to $ \eta $), such that $ v \perp_{\C^{\mu}} \Pi $. In this case, $ T_p\eta $ cannot be $ \C^{\mu} $-spacelike.
\item[(iii)] If $ \Pi $ is $ \C^{\mu} $-timelike, there is no direction $ v \in \C^{\mu}_p $ such that $ v \perp_{\C^{\mu}} \Pi $. In this case, $ T_p\eta $ must be $ \C^{\mu} $-timelike.
\end{itemize}
\end{lemma}
\begin{proof}
If $ \Pi $ is $ \C^{\mu} $-spacelike, there are exactly two directions $ v, w \in \C^{\mu}_p $ $ \C^{\mu} $-orthogonal to $ \Pi $ (see \cite[Proposition~5.2]{AJ16}). Indeed, choosing any $ \C^{\mu} $-spacelike hyperplane $ H \subset T_pQ $ containing $ \Pi $ and a $ \C^{\mu} $-timelike vector $ u_t $, all the directions in $ \C^{\mu}_p $ are given by the vectors $ u_t+\Sigma_p $, where $ \Sigma_p $ is a compact strongly convex hypersurface of $ H $ (as in the comment following Definition~\ref{def:cone_structure}). Then, there are exactly two hyperplanes of $ H $ parallel to $ \Pi $ and tangent to $ \Sigma_p $. If $ v_H, w_H $ are the tangency points, the required directions are given by $ v=u_t+v_H, w=u_t+w_H $. Notice that only $ \Pi $ and $ \C^{\mu}_p $ are relevant for this construction, so $ T_p\eta \subset T_pQ $ may be any hyperplane (of any causal character) containing $ \Pi $. Furthermore:
\begin{itemize}
\item If $ T_p\eta $ is $ \C^{\mu} $-timelike, we can choose $ u_t \in T_p\eta $, in which case $ T_p\eta = \textup{Span}\{ u_t,\Pi \} $. As $ T_p\eta $ divides $ \C^{\mu}_p $ into two pieces, then $ \Pi $ also divides $ \Sigma_p $ into two pieces, with $ v_H $ and $ w_H $ (and thus, $ v $ and $ w $) lying on opposite sides.
\item If $ T_p\eta $ is $ \C^{\mu} $-lightlike, then it is tangent to $ \C_p^{\mu} $ along a unique $ \C^{\mu} $-lightlike direction (recall Definition~\ref{def:causality}(iv)), which is $ \C^{\mu} $-orthogonal to $ T_p\eta $ and, thus, to $ \Pi $.
\item If $ T_p\eta $ is $ \C^{\mu} $-spacelike, the assertion immediately follows from the fact that either $ \C_p^{\mu} \subset Q^1_p $ or $ \C_p^{\mu} \subset Q^2_p $.
\end{itemize}

If $ \Pi $ is $ \C^{\mu} $-lightlike, then $ \Pi $ is tangent to $ \C^{\mu}_p $ along a unique radial direction $ v = \C^{\mu}_p \cap \Pi $ (recall Definition~\ref{def:causality}(iv)), i.e. $ \Pi \subset T_v\C^{\mu}_p = v^{\perp_{\C^{\mu}}} $. Therefore, $ v $ is $ \C^{\mu} $-orthogonal to $ \Pi $ and note that $ v \in \Pi \subset T_p\eta $ (so $ T_p\eta $ cannot be $ \C^{\mu} $-spacelike). Moreover, there is no other direction in $ \C^{\mu}_p $ $ \C^{\mu} $-orthogonal to $ v $ (recall Remark~\ref{rem:orth}(3)) and thus, $ \C^{\mu} $-orthogonal to $ \Pi $.

If $ \Pi $ is $ \C^{\mu} $-timelike, obviously so is $ T_p\eta $, and then $ \Pi \not\subset v^{\perp_{\C^{\mu}}} $ for all $ v \in \C^{\mu}_p $, since $ v^{\perp_{\C^{\mu}}} $ does not contain any $ \C^{\mu} $-timelike directions (again by Remark~\ref{rem:orth}(3)).
\end{proof}

\begin{lemma}
\label{lem:ort_equiv}
Let $ \Pi \subset T_p\eta $, with $ p \in \eta $, be a linear subspace of $ \dim(\Pi) = n-1 $. Then, for any $ v \in \C^{\mu}_p $ such that $ v $ is $ \C^{\mu} $-transverse to $ \eta $:\footnote{Recall Definition~\ref{def:c_transverse} and Remark~\ref{rem:c_transverse} for more insight regarding this transversality condition.}
\begin{equation*}
v \perp_{\C^{\mu}} \Pi \Leftrightarrow \Pi = v^{\perp_{\C^{\mu}}} \cap T_p\eta.
\end{equation*}
\end{lemma}
\begin{proof}
Observe that
\begin{equation*}
v \perp_{\C^{\mu}} \Pi \Leftrightarrow \Pi \subset v^{\perp_{\C^{\mu}}} \cap T_p\eta \Leftrightarrow \Pi = v^{\perp_{\C^{\mu}}} \cap T_p\eta,
\end{equation*}
where we have used, in the first equivalence, that $ \Pi \subset T_p\eta $, and in the second one, that $ \dim(\Pi) = \dim(v^{\perp_{\C^{\mu}}} \cap T_p\eta) = n-1 $ because $ v^{\perp_{\C^{\mu}}} \not= T_p\eta $.
\end{proof}

Next, we define what we understand as incident, refracted and reflected directions and trajectories, in agreement with Fermat's principle (Theorem~\ref{thm:snell}(I) and Corollary~\ref{thm:reflection}(I)).

\begin{defi}
\label{def:directions}
For any $ r \in \C_p^{\mu} $, we define (see Figure~\ref{fig:pi_r})
\begin{equation*}
\Pi^{\mu}_r \coloneqq r^{\perp_{\C^{\mu}}} \cap T_p\eta.
\end{equation*}
Then, we say that $ u \in \C_p^1 $ is an {\em incident direction} if $ u \in Q^2_p \cup T_p\eta $, and given such a direction:
\begin{itemize}
\item $ v \in \C_p^2 $ is a {\em refracted direction} (associated with $ u $) if $ v \in Q^2_p \cup T_p\eta $ and Snell's law $ \Pi^1_u \subset \Pi^2_v $ (i.e. $ v \perp_{\C^2} \Pi^1_u $) holds.
\item $ w \in \C_p^1 $ is a {\em reflected direction} (associated with $ u $) if $ w \in Q^1_p \cup T_p\eta $ and the law of reflection $ \Pi^1_u \subset \Pi^1_w $ (i.e. $ w \perp_{\C^1} \Pi^1_u $) holds.
\end{itemize}
Consistently, $ \gamma_u: [a,\tau] \rightarrow \overline{Q}_1 $ is an {\em incident trajectory} if it is a cone geodesic of $ \C^1 $ topologically transverse to $ \eta $, with $ \gamma_u(\tau) \in \eta $, such that $ u = \dot{\gamma}_u(\tau^-) $ is an incident direction, and given such a curve:
\begin{itemize}
\item $ \gamma_v: [\tau,b] \rightarrow \overline{Q}_2 $ is a {\em refracted trajectory} (associated with $ \gamma_u $) if it is a cone geodesic of $ \C^2 $ topologically transverse to $ \eta $, with $ \gamma_v(\tau) \in \eta $, such that $ v = \dot{\gamma}_v(\tau^+) $ is a refracted direction (associated with $ u $).
\item $ \gamma_w: [\tau,b] \rightarrow \overline{Q}_1 $ is a {\em reflected trajectory} (associated with $ \gamma_u $) if it is a cone geodesic of $ \C^1 $ topologically transverse to $ \eta $, with $ \gamma_w(\tau) \in \eta $, such that $ w = \dot{\gamma}_w(\tau^+) $ is a reflected direction (associated with $ u $).
\end{itemize}
In addition, we say that $ \gamma_r $ is $ \C^{\mu} $-{\em transverse} to $ \eta $ if $ r \in \C^{\mu}_p $ is so, with $ p = \gamma_r(\tau) $, i.e. if $ r^{\perp_{\C^{\mu}}} \not= T_p\eta $ ($ \mu = 1 $ when $ r = u,w $; $ \mu = 2 $ when $ r = v $).
\end{defi}

\begin{figure}
\centering
\begin{subfigure}{0.4\textwidth}
\includegraphics[width=\textwidth]{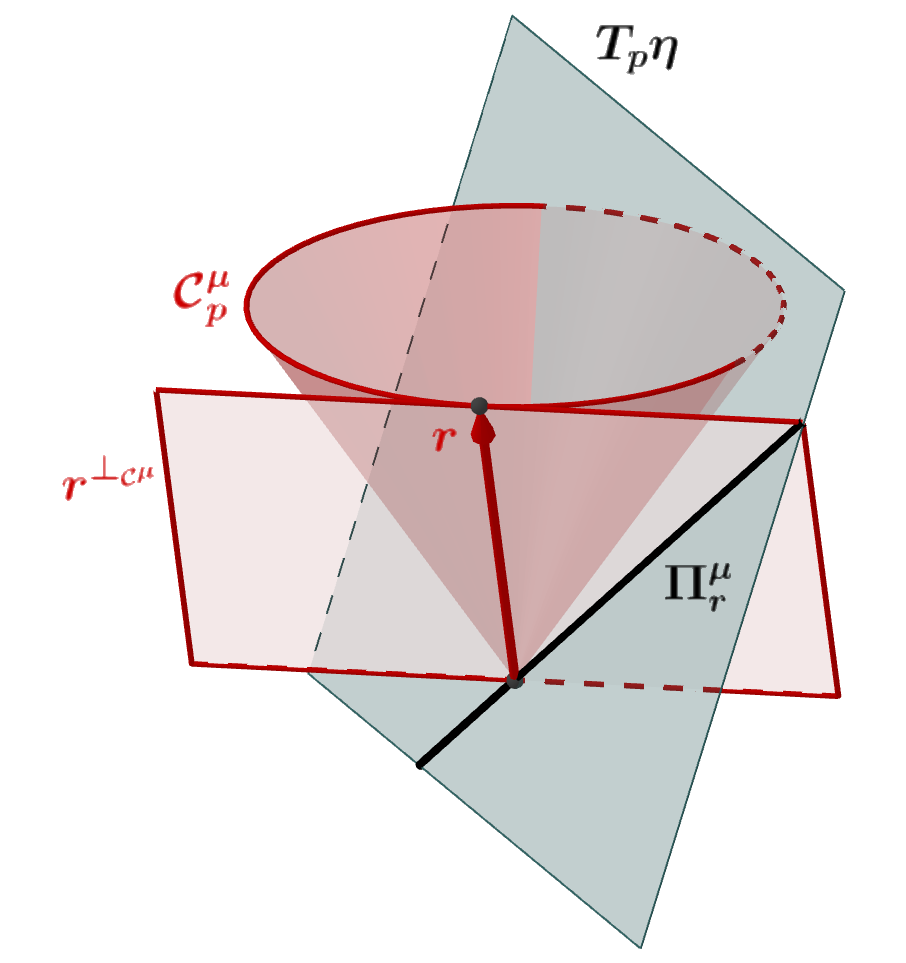}
\caption{$ r^{\perp_{\C^{\mu}}} \not= T_p\eta $.}
\label{fig:pi_r_a}
\end{subfigure}
\begin{subfigure}{0.4\textwidth}
\includegraphics[width=\textwidth]{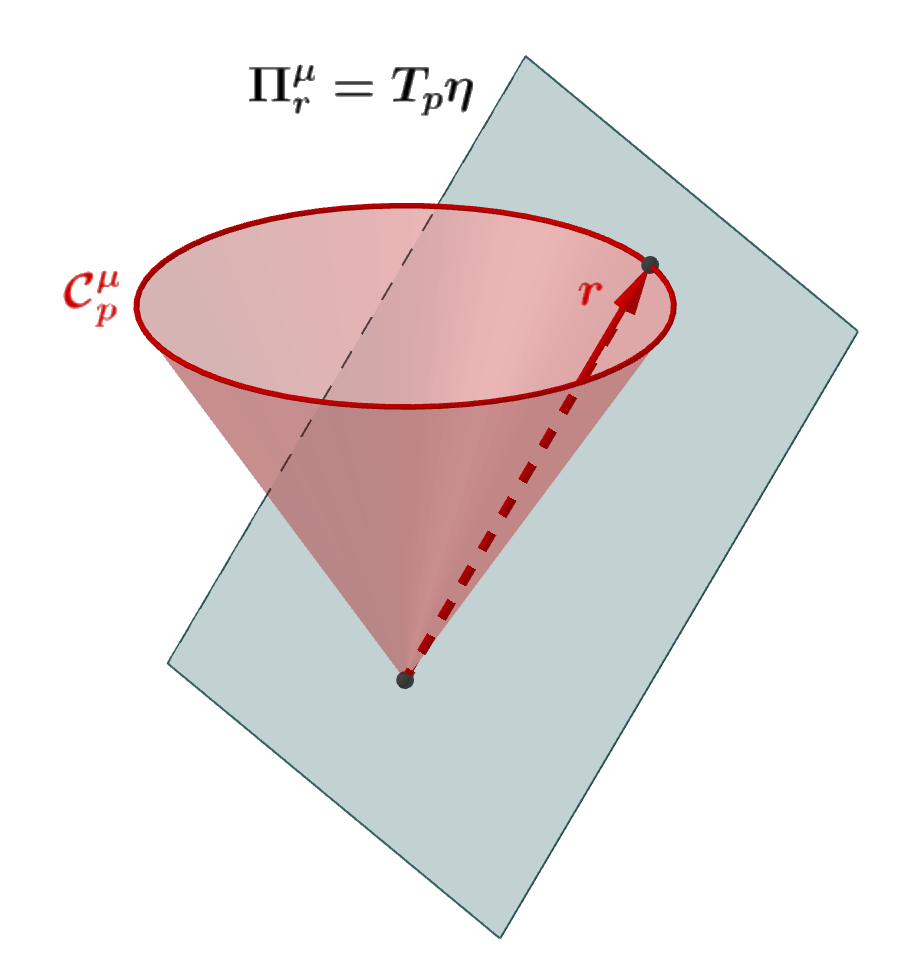}
\caption{$ r^{\perp_{\C^{\mu}}} = T_p\eta $.}
\label{fig:pi_r_b}
\end{subfigure}
\caption{On the left, $ r $ is transverse (and thus, $\C^{\mu} $-transverse) to $ \eta $ and $ \Pi^{\mu}_r \coloneqq r^{\perp_{\C^{\mu}}} \cap T_p\eta $ is a codimension 1 subspace of $ T_p\eta $. On the right, $ \C_p^{\mu} $ is tangent to $ T_p\eta $ along the direction $ r $ and $ \Pi^{\mu}_r = T_p\eta $.}
\label{fig:pi_r}
\end{figure}

\begin{rem}
\label{rem:directions}
Observe that:
\begin{enumerate}
\item The definition of a refracted (resp. reflected) direction ensures that if the corresponding refracted (resp. reflected) trajectory exists, then it is a critical point of $ \T $ (resp. $ \T^* $), by Theorem~\ref{thm:snell}(I) (resp. Corollary~\ref{thm:reflection}(I)).
\item A refracted or reflected trajectory may not exist, even if its corresponding refracted or reflected direction does. Specifically, when this direction is tangent to $ \eta $, the bending of $ \eta $ may cause the corresponding cone geodesic to lie in the ``wrong'' side of the interface.
\item Given any $ r \in \C_p^{\mu} $, the following are equivalent characterizations of the $ \C^{\mu} $-transversality, due to Remark~\ref{rem:c_transverse} (see Figure~\ref{fig:pi_r_a}):
\begin{itemize}
\item[$ \circ $] $ r $ is $ \C^{\mu} $-transverse to $ T_p\eta $, i.e. $ r^{\perp_{\C^{\mu}}} \not= T_p\eta $.
\item[$ \circ $] $ \Pi^{\mu}_r $ is a codimension 1 subspace of $ T_p\eta $.
\item[$ \circ $] Either $ T_p\eta $ is not $ \C^{\mu} $-lightlike or $ r \notin T_p\eta $ (or both).
\end{itemize}
\item Conversely, $ r \in \C_p^{\mu} $ is not $ \C^{\mu} $-transverse to $ \eta $ when one of the following equivalent conditions holds (see Figure~\ref{fig:pi_r_b}):
\begin{itemize}
\item[$ \circ $] $ \C_p^{\mu} $ is tangent to $ T_p\eta $ along the direction $ r $, i.e. $ r^{\perp_{\C^{\mu}}} = T_p\eta $.
\item[$ \circ $] $ \Pi^{\mu}_r = T_p\eta $.
\item[$ \circ $] $ T_p\eta $ is $ \C^{\mu} $-lightlike and $ r \in T_p\eta $.
\end{itemize}
\item Due to Lemma~\ref{lem:ort_equiv}, if $ r \in \C^\mu_p $ is $ \C^{\mu} $-transverse to $ \eta $, then $ \Pi_u^1 \subset \Pi_r^\mu $ implies $ \Pi_u^1 = \Pi_r^\mu $ (which can only hold if $ u $ is also $ \C^1 $-transverse to $ \eta $). In particular, when $ T_p\eta $ is not $ \C^\mu $-lightlike, then every $ r \in \C^\mu_p $ is $ \C^{\mu} $-transverse to $ \eta $, so Snell's law (when $ \mu=2 $) or the law of reflection (when $ \mu=1 $) can only be satisfied with equality: $ \Pi_u^1 = \Pi_r^\mu $. Therefore, the cases in which these laws hold as a strict inclusion rather than an equality correspond to non-generic configurations with less physical relevance, since an arbitrarily small perturbation of $ T_p\eta $ restores the equality.
\end{enumerate}
\end{rem}

\subsection{Refracted directions}
\label{sec:refracted_dir}
Now we are ready to study the existence and uniqueness of refracted directions, depending on the different possibilities for the incident direction $ u $ and for the causal characters of $ T_p\eta $ and $ \Pi^1_u $.

\begin{thm}[Existence and uniqueness of refracted directions]
\label{thm:existence_refraction}
Let $ u \in \C^1_p $ be an incident direction at $ p \in \eta $. When $ u $ is $ \C^1 $-transverse to $ \eta $ and $ \C^2_p \cap Q^2_p \not= \emptyset $:
\begin{itemize}
\item[(A)] If $ T_p\eta $ is $ \C^2 $-timelike, then:
\begin{itemize}
\item[(i)] If $ \Pi^1_u $ is $ \C^2 $-spacelike, there exists a unique refracted direction, necessarily transverse to $ \eta $ (see Figure~\ref{fig:case_ai}).
\item[(ii)] If $ \Pi^1_u $ is $ \C^2 $-lightlike, there exists a unique refracted direction, necessarily tangent to $ \eta $ (see Figure~\ref{fig:case_aii}).
\item[(iii)] If $ \Pi^1_u $ is $ \C^2 $-timelike, there are no refracted directions.
\end{itemize}
\item[(B)] If $ T_p\eta $ is $ \C^2 $-lightlike, then $ \Pi^1_u $ cannot be $ \C^2 $-timelike and:
\begin{itemize}
\item[(i)] If $ \Pi^1_u $ is $ \C^2 $-spacelike, there are exactly two refracted directions, necessarily one tangent to $ \eta $ and the other transverse (see Figure~\ref{fig:case_bi}).
\item[(ii)] If $ \Pi^1_u $ is $ \C^2 $-lightlike, there exists a unique refracted direction, necessarily tangent to $ \eta $ (see Figure~\ref{fig:case_bii}).
\end{itemize}
\item[(C)] If $ T_p\eta $ is $ \C^2 $-spacelike, there are exactly two refracted directions, necessarily transverse to $ \eta $ (see Figure~\ref{fig:case_c}).
\end{itemize}
Otherwise, when $ u^{\perp_{\C^1}} = T_p\eta $ (see Figure~\ref{fig:ex_refraction_a}) or $ \C^2_p \cap Q^2_p = \emptyset $ (see Figure~\ref{fig:ex_refraction_b}), there exists a refracted direction (necessarily unique and tangent to $ \eta $) if and only if $ T_p\eta $ is $ \C^2 $-lightlike.
\end{thm}
\begin{proof}
When $ u $ is $ \C^1 $-transverse to $ \eta $ and $ \C^2_p \cap Q^2_p \not= \emptyset $, we have the following cases, depending on the causal character of $ T_p\eta $ with respect to $ \C^2 $:
\begin{itemize}
\item If $ T_p\eta $ is $ \C^2 $-timelike, this guarantees that $ v $ is $ \C^2 $-transverse to $ \eta $, for any $ v \in \C_p^2 $. Then:
\begin{itemize}
\item[$ \circ $] If $ \Pi^1_u $ is $ \C^2 $-spacelike, Lemmas~\ref{lem:Pi}(i) and \ref{lem:ort_equiv} tell us that there are exactly two directions $ v, \hat{v} \in \C_p^2 $ satisfying Snell's law $ \Pi^1_u = \Pi^2_v = \Pi^2_{\hat{v}} $, but either $ v \in Q^2_p $ and $ \hat{v} \in Q^1_p $, or vice versa, i.e. there is exactly one refracted direction.
\item[$ \circ $] If $ \Pi_u^1 $ is $ \C^2 $-lightlike, then by Lemmas~\ref{lem:Pi}(ii) and \ref{lem:ort_equiv} there is a unique direction $ v \in \C_p^2 $, necessarily tangent to $ \eta $, satisfying Snell's law $ \Pi_u^1 = \Pi_v^2 $.
\item[$ \circ $] If $ \Pi_u^1 $ is $ \C^2 $-timelike, then Lemma~\ref{lem:Pi}(iii) implies that there are no refracted directions.
\end{itemize}
\item If $ T_p\eta $ is $ \C^2 $-lightlike, it means that $ T_p\eta $ is tangent to $ \C_p^2 $ along a unique radial direction $ \hat{v} \in T_p\eta $, with $ \hat{v}^{\perp_{\C^2}} = T_p\eta $, so $ \hat{v} $ satisfies Snell's law $ \Pi_u^1 \subset \Pi_{\hat{v}}^2 = T_p\eta $. Then:
\begin{itemize}
\item[$ \circ $] If $ \Pi_u^1 $ is $ \C^2 $-spacelike, by Lemma~\ref{lem:Pi}(i) there are two different $ \C^2 $-lightlike directions $ \C^2 $-orthogonal to $ \Pi_u^1 $, and clearly $ \hat{v} $ is one of them. Since $ \C^2_p \cap Q^2_p \not= \emptyset $, this implies that $ \C^2_p \cap Q^1_p = \emptyset $ and thus, the other orthogonal direction $ v $ is necessarily in $ Q^2_p $.
\item[$ \circ $] If $ \Pi_u^1 $ is $ \C^2 $-lightlike, then by Lemma~\ref{lem:Pi}(ii) there is a unique $ \C^2 $-lightlike direction $ \C^2 $-orthogonal to $ \Pi_u^1 $, which has to be $ \hat{v} $.
\item[$ \circ $] $ \Pi_u^1 $ cannot be $ \C^2 $-timelike, since $ \Pi_u^1 \subset T_p\eta $ but $ T_p\eta $ only contains $ \C^2 $-spacelike or $ \C^2 $-lightlike directions.
\end{itemize}
\item If $ T_p\eta $ is $ \C^2 $-spacelike, the hypothesis $ \C^2_p \cap Q^2_p \not= \emptyset $ guarantees that $ \C^2_p \subset Q^2_p $ and $ v $ is $ \C^2 $-transverse to $ \eta $, for any $ v \in \C_p^2 $. Since $ \Pi_u^1 \subset T_p\eta $ must also be $ \C^2 $-spacelike, by Lemmas~\ref{lem:Pi}(i) and \ref{lem:ort_equiv} there are exactly two directions $ v, \hat{v} \in \C_p^2 \subset Q^2_p $ satisfying Snell's law $ \Pi_u^1 = \Pi_v^2 = \Pi_{\hat{v}}^2 $, so both are refracted directions non-tangent to $ \eta $.
\end{itemize}
When $ u^{\perp_{\C^1}} = T_p\eta $ or $ \C^2_p \cap Q^2_p = \emptyset $, $ v \in \C_p^2 $ is a refracted direction if and only if $ v^{\perp_{\C^2}} = T_p\eta $ (and thus, $ \Pi^2_v = T_p\eta $), i.e. $ \C^2_p $ is tangent to $ T_p\eta $ along the radial direction $ v \in \C^2_p \cap T_p\eta $ (recall Remark~\ref{rem:directions}(4)). Moreover, if this direction exists, then it is necessarily unique, due to the non-radial strong convexity of $ \C^2_p $.
\end{proof}

\begin{figure}
\centering
\begin{subfigure}{0.4\textwidth}
\includegraphics[width=\textwidth]{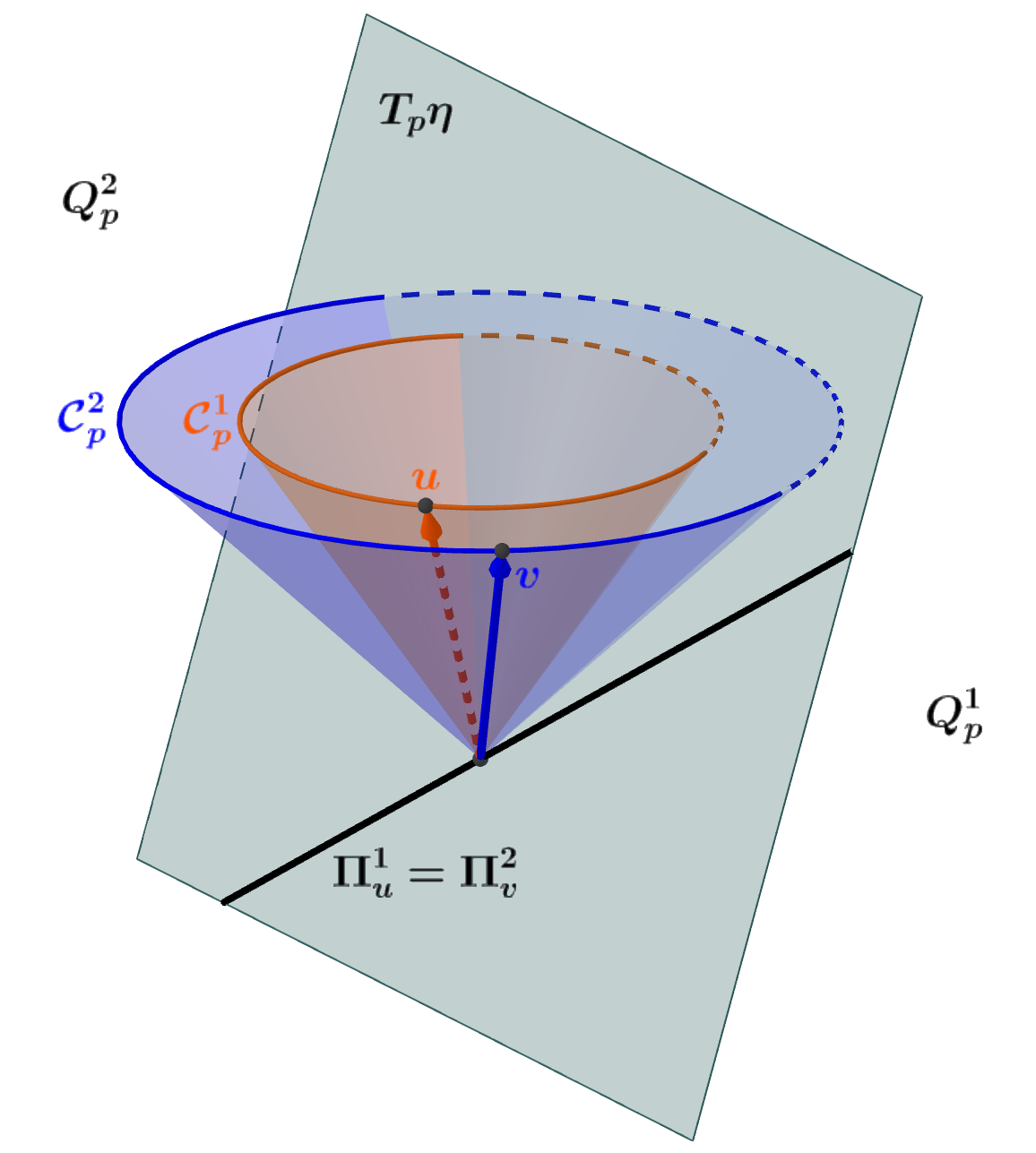}
\caption{Case (A)(i).}
\label{fig:case_ai}
\end{subfigure}
\begin{subfigure}{0.4\textwidth}
\includegraphics[width=\textwidth]{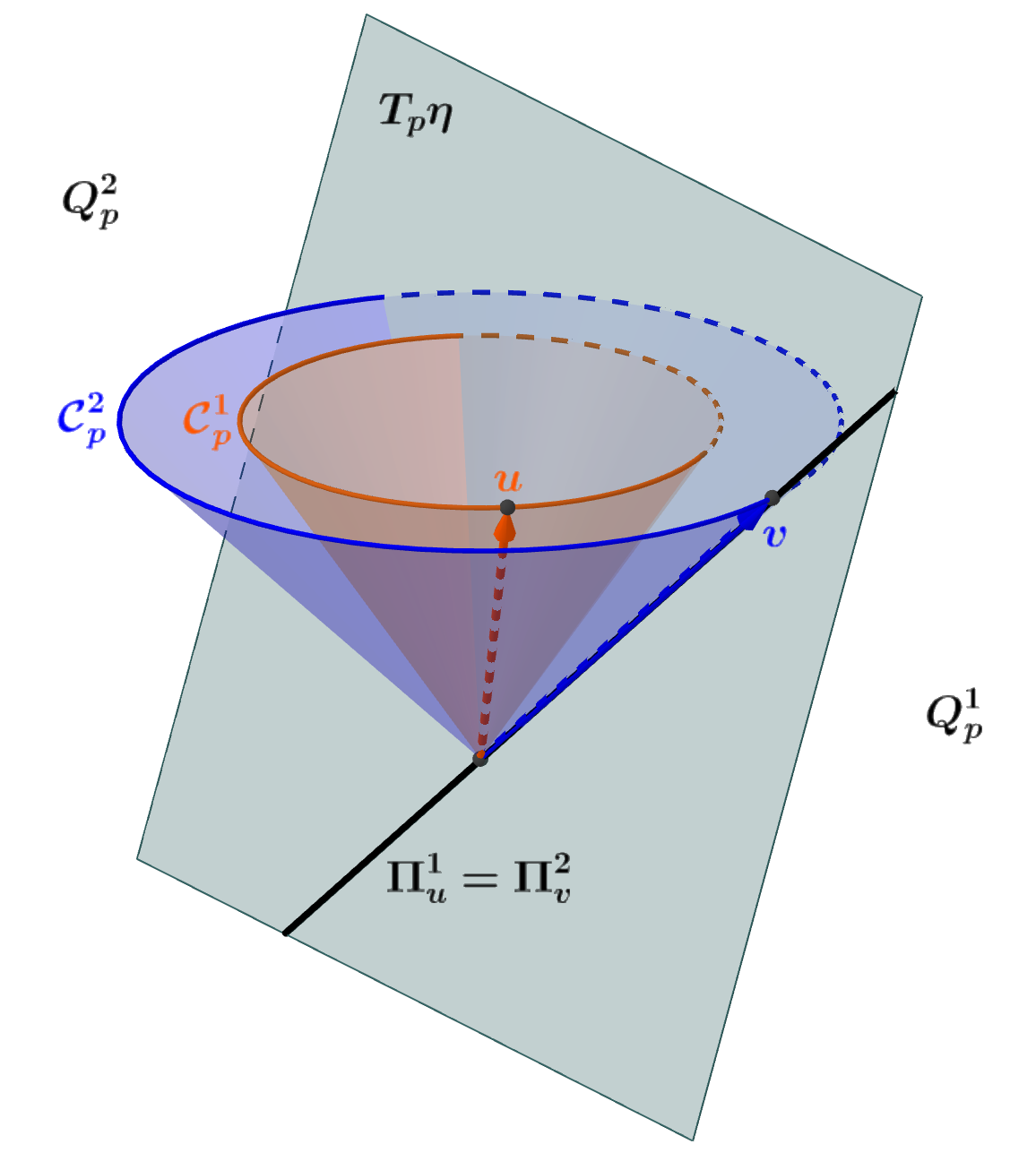}
\caption{Case (A)(ii).}
\label{fig:case_aii}
\end{subfigure}
\caption{Case (A) of Theorem~\ref{thm:existence_refraction}, when $ T_p\eta $ is $ \C^2 $-timelike (and also $ \C^1 $-timelike in this figure). On the left, $ \Pi_u^1 $ is $ \C^2 $-spacelike and $ v \in Q^2_p $ is the unique refracted direction. On the right, $ \Pi_u^1 $ is $ \C^2 $-lightlike and $ v \in T_p\eta $ is the unique refracted direction.}
\label{fig:case_a}
\end{figure}

\begin{figure}
\centering
\begin{subfigure}{0.4\textwidth}
\includegraphics[width=\textwidth]{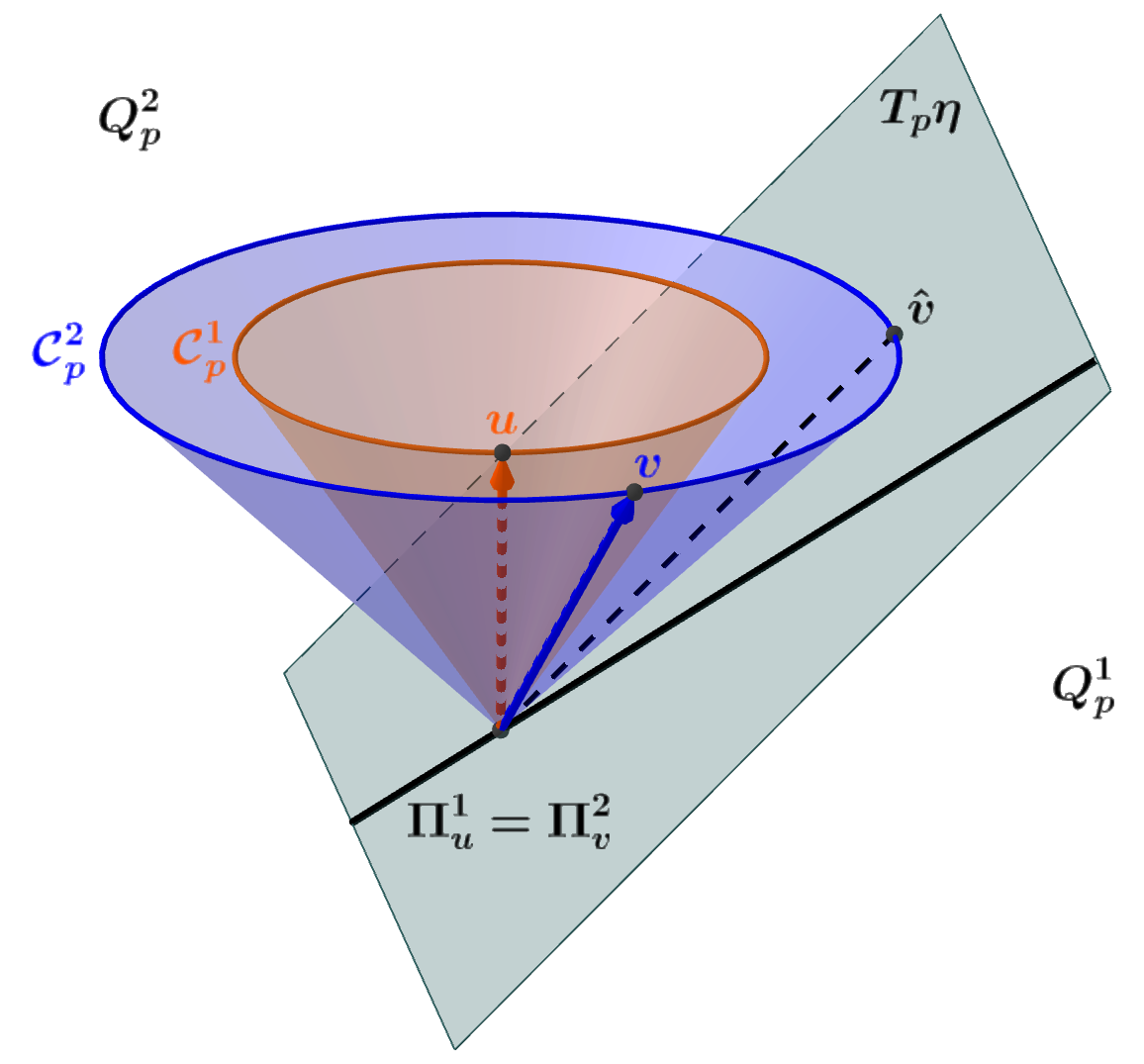}
\caption{Case (B)(i).}
\label{fig:case_bi}
\end{subfigure}
\begin{subfigure}{0.4\textwidth}
\includegraphics[width=\textwidth]{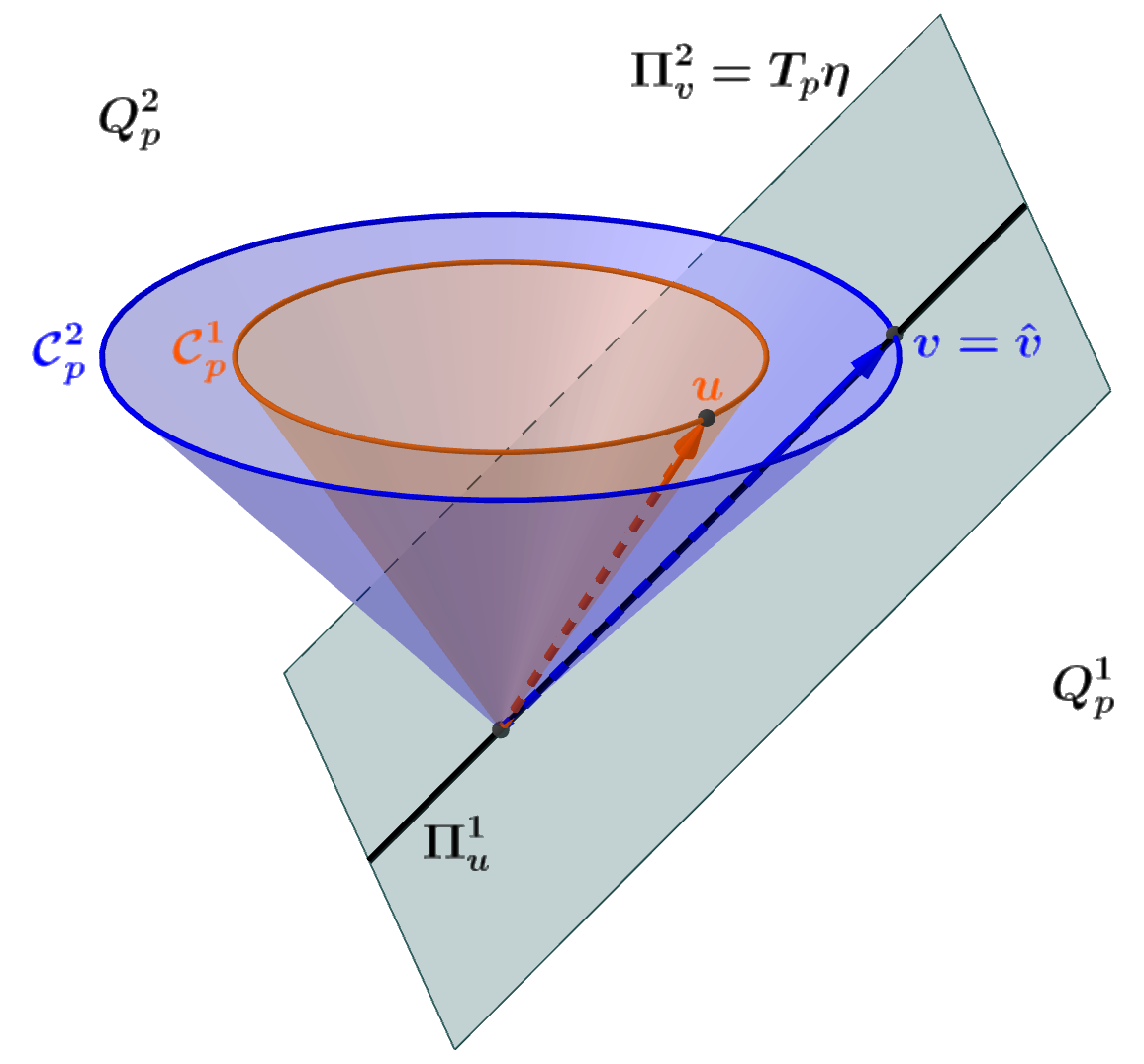}
\caption{Case (B)(ii).}
\label{fig:case_bii}
\end{subfigure}
\caption{Case (B) of Theorem~\ref{thm:existence_refraction}, when $ T_p\eta $ is $ \C^2 $-lightlike and tangent to $ \C^2_p $ along the direction $ \hat{v} $. On the left, $ \Pi_u^1 $ is $ \C^2 $-spacelike and $ v \in Q^2_p $, $ \hat{v} \in T_p\eta $ are two different refracted directions. On the right, $ \Pi_u^1 $ is $ \C^2 $-lightlike and $ v = \hat{v} \in T_p\eta $ is the unique refracted direction.}
\end{figure}

\begin{figure}
\centering
\includegraphics[width=0.6\textwidth]{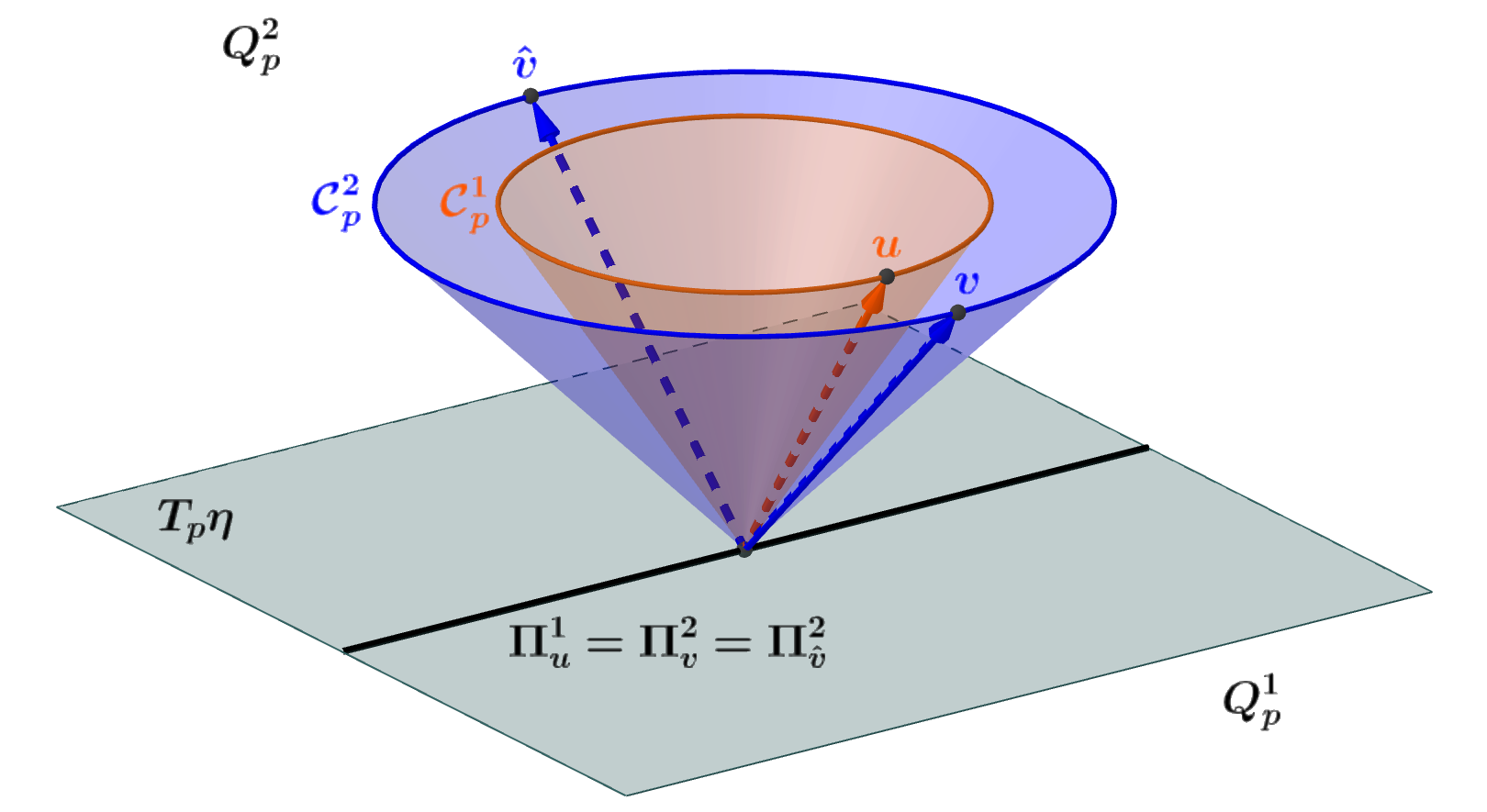}
\caption{Case (C) of Theorem~\ref{thm:existence_refraction}, when $ T_p\eta $ is $ \C^2 $-spacelike (and also $ \C^1 $-spacelike in this figure). There are always two different refracted directions $ v, \hat{v} \in Q^2_p $ in this situation.}
\label{fig:case_c}
\end{figure}

\begin{figure}
\centering
\begin{subfigure}{0.4\textwidth}
\includegraphics[width=\textwidth]{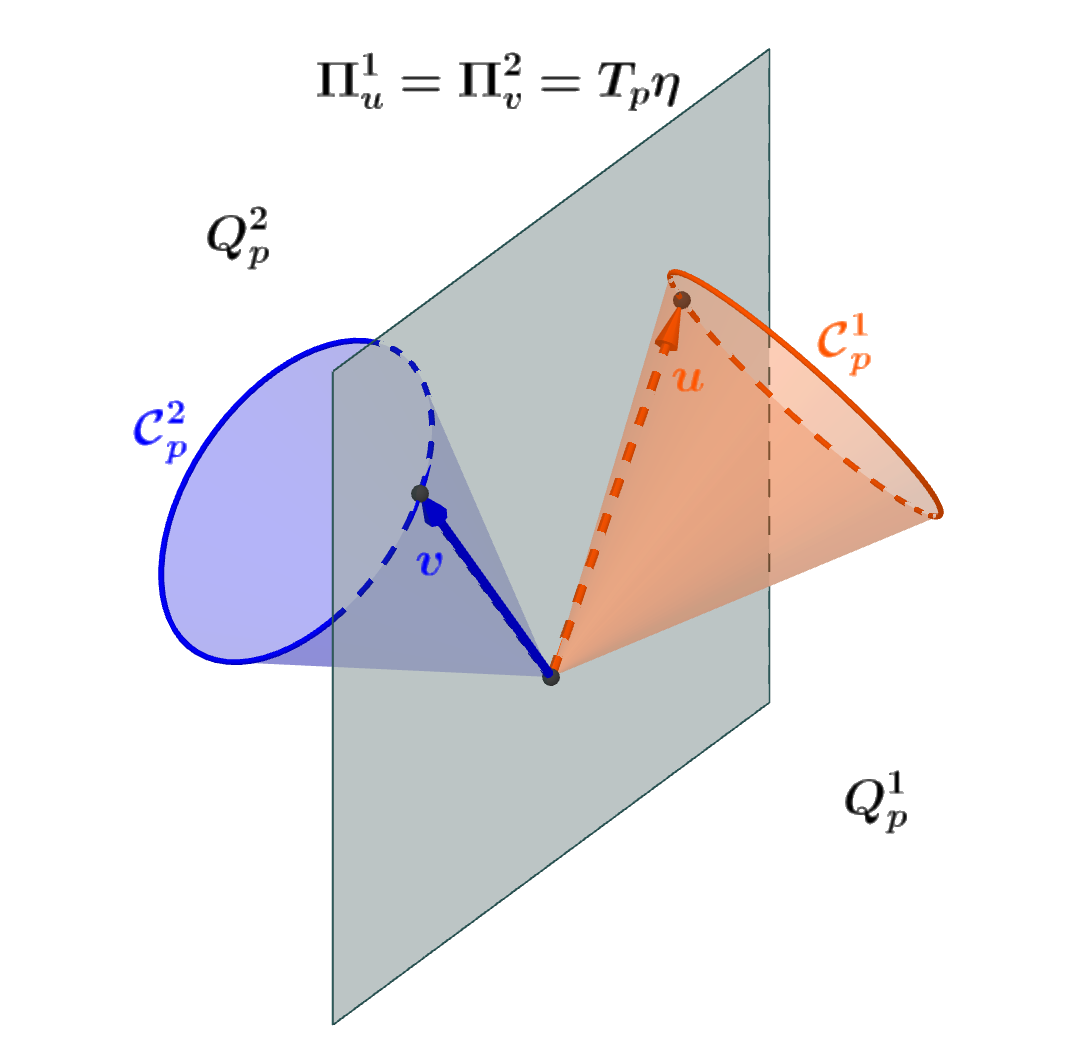}
\caption{$ u^{\perp_{\C^1}} = T_p\eta $.}
\label{fig:ex_refraction_a}
\end{subfigure}
\begin{subfigure}{0.4\textwidth}
\includegraphics[width=\textwidth]{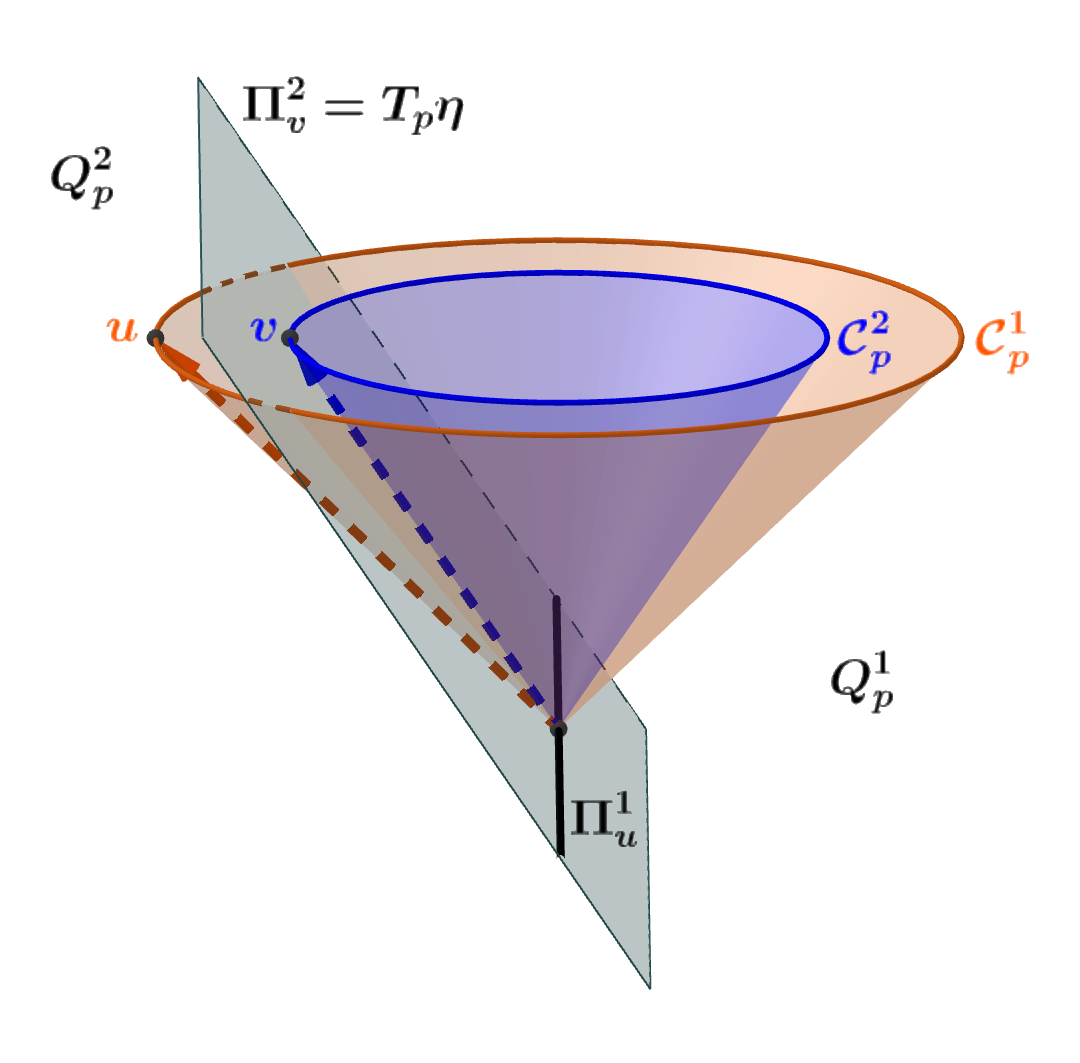}
\caption{$ \C^2_p \cap Q^2_p = \emptyset $.}
\label{fig:ex_refraction_b}
\end{subfigure}
\caption{Final cases from Theorem~\ref{thm:existence_refraction}. On the left, $ \C^2_p \cap Q^2_p \not= \emptyset $ but $ u^{\perp_{\C^1}} = T_p\eta $. On the right, $ u $ is $ \C^1 $-transverse to $ \eta $ but $ \C^2_p \cap Q^2_p = \emptyset $. In both cases, $ \C^2_p $ is tangent to $ T_p\eta $ along the unique refracted direction $ v \in T_p\eta $.}
\end{figure}

To complete the analysis, we can also distinguish between the different cases of non-existence:
\begin{itemize}
\item {\em No critical points}: although there are $\C^2$-lightlike directions at $ p $ pointing to $\overline{Q}_2$, none satisfies Snell's law. Hence, no cone geodesic crossing the interface at $p$ can be a critical point of the arrival time functional. This corresponds to case (A)(iii) of Theorem~\ref{thm:existence_refraction} and to the case where $u^{\perp_{\C^1}}=T_p\eta$, $\C_p^2 \cap Q_p^2 \not= \emptyset$ and $ T_p\eta $ is not $ \C^2 $-lightlike. 
\item {\em No crossing}: no $\C^2$-lightlike direction at $ p $ points to $\overline{Q}_2$. Consequently, cone geodesics arriving at $p$ cannot cross the interface into $\overline{Q}_2$. This corresponds to the case where $\C^2_p \subset Q_p^1$.
\item {\em Not possible}: the causal characters of $T_p\eta$ and $\Pi_u^1$ are incompatible.
\end{itemize}

Table~\ref{tab:refraction} summarizes all possibilities concerning the existence, non-existence, uniqueness and non-uniqueness of refracted directions when $ u $ is $ \C^1 $-transverse to $ T_p\eta $.

\begin{table}
\begin{center}
\begin{tabular}{cccc}
\multicolumn{4}{c}{\textbf{Refracted directions}} \\
\multicolumn{4}{c}{(for $ \C^1 $-transverse incident directions)} \\ \midrule\midrule
$ \boldsymbol{T_p\eta} \ \boldsymbol{\backslash} \ \boldsymbol{\Pi_u^1} $ & $ \boldsymbol{\C^2} $\textbf{-spacelike} & $ \boldsymbol{\C^2} $\textbf{-lightlike} & $ \boldsymbol{\C^2} $\textbf{-timelike} \\ \midrule\midrule
\multirow{2}{*}{$ \boldsymbol{\C^2} $\textbf{-timelike}} & $ \exists! $ ($ \not\in T_p\eta $) & $ \exists! $ ($ \in T_p\eta $) & \multirow{2}{*}{No critical points} \\
 & Figure~\ref{fig:case_ai} & Figure~\ref{fig:case_aii} & \\ \midrule
\multirow{4}{*}{$ \boldsymbol{\C^2} $\textbf{-lightlike}} & $ \exists $ $ 2 $ ($ \not\in T_p\eta $ \& $ \in T_p\eta $) if $ (*) $ & & \multirow{4}{*}{Not possible} \\
 & Figure~\ref{fig:case_bi} & $ \exists! $ ($ \in T_p\eta $) & \\
 & $ \exists! $ ($ \in T_p\eta $) otherwise & Figure~\ref{fig:case_bii} & \\
 & Figure~\ref{fig:ex_refraction_b} & & \\ \midrule
\multirow{3}{*}{$ \boldsymbol{\C^2} $\textbf{-spacelike}} & $ \exists $ $ 2 $ ($ \not\in T_p\eta $) if $ (*) $ & \multirow{3}{*}{Not possible} & \multirow{3}{*}{Not possible} \\
 & Figure~\ref{fig:case_c} & & \\
 & No crossing otherwise & & \\ \midrule\midrule
\multicolumn{4}{c}{$ (*) $: $ \C^2_p \cap Q^2_p \not= \emptyset $}
\end{tabular}
\captionsetup{singlelinecheck=off}
\caption[Table 1]{All possible refracted directions (from Theorem~\ref{thm:existence_refraction}) associated with an incident direction $ u $ at $ p \in \eta $, assuming that $ u $ is $ \C^1 $-transverse to $ \eta $. The cases of non-existence are classified as follows: {\em no critical points} (cone geodesics can cross into $ \overline{Q}_2 $, but none is critical); {\em no crossing} ($ \C_p^2 \subset Q_p^1 $, so cone geodesics cannot cross into $ \overline{Q}_2 $); and {\em not possible} (incompatible causal characters of $ T_p\eta $ and $ \Pi_u^1 $).}
\label{tab:refraction}
\end{center}
\end{table}

\subsection{Reflected directions}
\label{sec:reflected_dir}
In the same way as we reduced the law of reflection in \S~\ref{sec:reflection} as a particular case of Snell's law of refraction (by using the construction of the double manifold in Remark~\ref{rem:double}), we can also identify the reflected directions in the original setting $ \overline{Q}_1 \cup \overline{Q}_2 $ as refracted ones in $ \overline{Q}_1 \cup_{\textup{Id}} \overline{Q}_1 $, obtaining the following analogous result to Theorem~\ref{thm:existence_refraction}.

\begin{cor}[Existence and uniqueness of reflected directions]
\label{thm:existence_reflection}
Let $ u \in \C^1_p $ be an incident direction at $ p \in \eta $. When $ u $ is $ \C^1 $-transverse to $ \eta $:
\begin{itemize}
\item[(A$^*$)] If $ T_p\eta $ is $ \C^1 $-timelike, then $ \Pi^1_u $ cannot be $ \C^1 $-timelike and:
\begin{itemize}
\item[(i)] If $ \Pi_u^1 $ is $ \C^1 $-spacelike, there exists a unique reflected direction, necessarily transverse to $ \eta $ (see Figure~\ref{fig:case_asi}).
\item[(ii)] If $ \Pi_u^1 $ is $ \C^1 $-lightlike, there exists a unique reflected direction, necessarily tangent to $ \eta $, which coincides with $ u $ (see Figure~\ref{fig:case_asii}).
\end{itemize}
\item[(B$^*$)] If $ T_p\eta $ is $ \C^1 $-lightlike, there exists a unique reflected direction, necessarily tangent to $ \eta $ (see Figure~\ref{fig:ex_reflection}).
\item[(C$^*$)] If $ T_p\eta $ is $ \C^1 $-spacelike, there are no reflected directions.
\end{itemize}
Otherwise, when $ u^{\perp_{\C^1}} = T_p\eta $ (see Figure~\ref{fig:ex_reflection2}), there exists a unique reflected direction, necessarily tangent to $ \eta $, which coincides with $ u $.
\end{cor}
\begin{proof}
Consider the double manifold $ \overline{Q}_1 \cup_{\textup{Id}} \overline{Q}_1 $ (as in Remark~\ref{rem:double}) and observe that, when $ u $ is $ \C^1 $-transverse to $ \eta $:
\begin{itemize}
\item If $ T_p\eta $ is $ \C^1 $-timelike, then $ \C^1_p \cap Q^1_p \not= \emptyset $ and we can apply the first part of Theorem~\ref{thm:existence_refraction} in $ \overline{Q}_1 \cup_{\textup{Id}} \overline{Q}_1 $:
\begin{itemize}
\item[$ \circ $] If $ \Pi_u^1 $ is $ \C^1 $-spacelike, case (A)(i) directly converts into case (A$^*$)(i).
\item[$ \circ $] If $ \Pi_u^1 $ is $ \C^1 $-lightlike (so $ u \in T_p\eta $), case (A)(ii) tells us that there exists a unique reflected direction $ w \in \C_p^1 $, necessarily tangent to $ \eta $. Since the incident direction $ u \in \C^1_p \cap T_p\eta $ is a valid reflected direction in this situation, we conclude that $ w = u $.
\item[$ \circ $] $ \Pi_u^1 $ cannot be $ \C^1 $-timelike, since $ \Pi_u^1 \subset u^{\perp_{L_1}} $ but $ u^{\perp_{L_1}} $ only contains $ \C^1 $-spacelike or $ \C^1 $-lightlike directions.
\end{itemize}
\item If $ T_p\eta $ is $ \C^1 $-lightlike, the existence of an incident direction $ u \in \C^1_p $ with $ u^{\perp_{\C^1}} \not= T_p\eta $ implies that $ \C^1_p \cap Q^1_p = \emptyset $. Then, the last statement of Theorem~\ref{thm:existence_refraction} in $ \overline{Q}_1 \cup_{\textup{Id}} \overline{Q}_1 $ ensures the existence of a unique reflected direction, necessarily tangent to $ \eta $.
\item If $ T_p\eta $ is $ \C^1 $-spacelike, the existence of an incident direction $ u \in \C^1_p $ means that $ \C^1_p \subset Q_p^2 $, so there cannot be reflected directions because $ \C^1_p \cap (Q_p^1 \cup T_p\eta) = \emptyset $.
\end{itemize}
When $ u^{\perp_{\C^1}} = T_p\eta $, the last statement of Theorem~\ref{thm:existence_refraction} in $ \overline{Q}_1 \cup_{\textup{Id}} \overline{Q}_1 $ directly implies that $ u \in \C^1_p \cap T_p\eta $ is the unique reflected direction.
\end{proof}

\begin{figure}
\centering
\begin{subfigure}{0.4\textwidth}
\includegraphics[width=\textwidth]{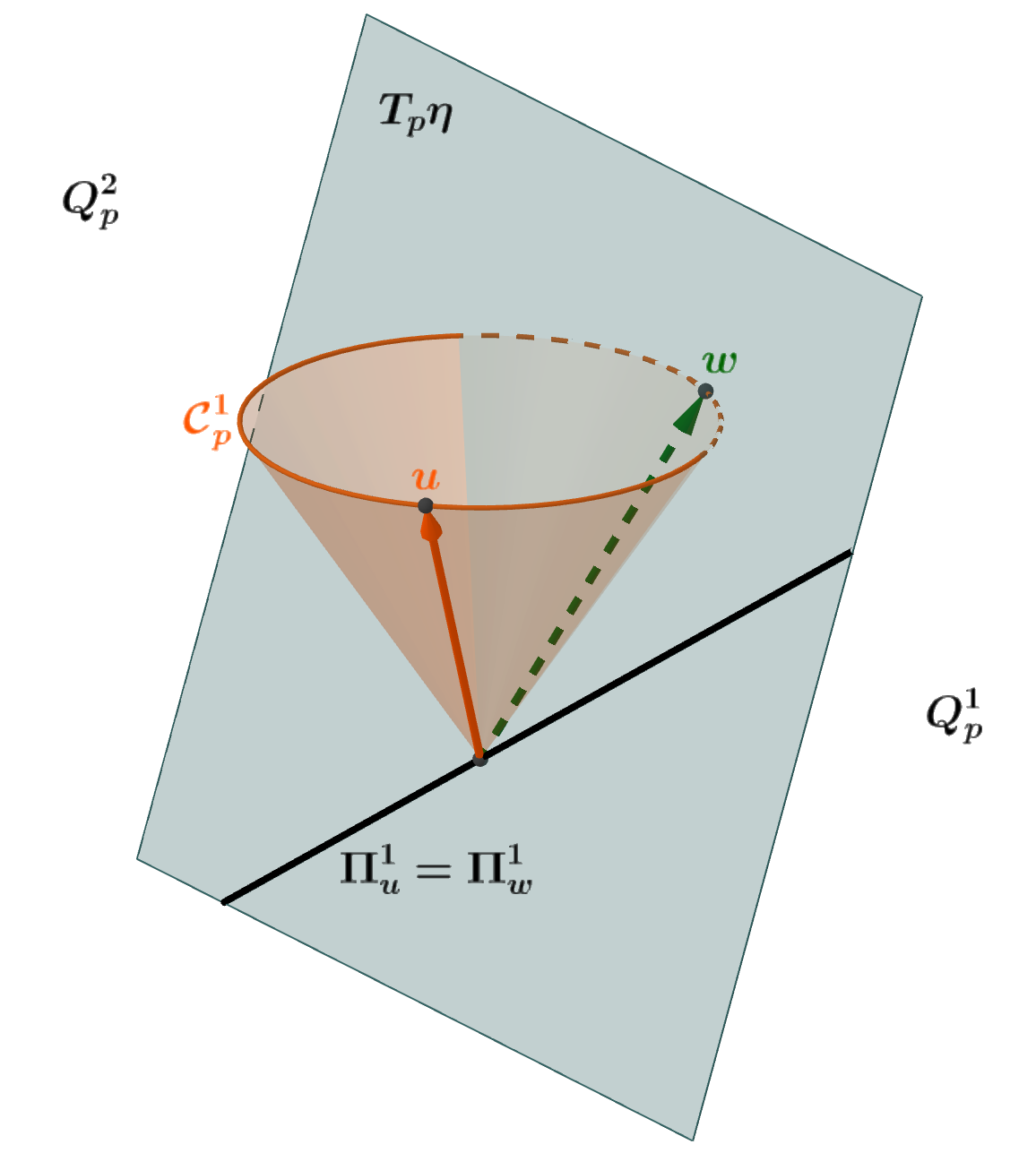}
\caption{Case (A$^*$)(i).}
\label{fig:case_asi}
\end{subfigure}
\begin{subfigure}{0.4\textwidth}
\includegraphics[width=\textwidth]{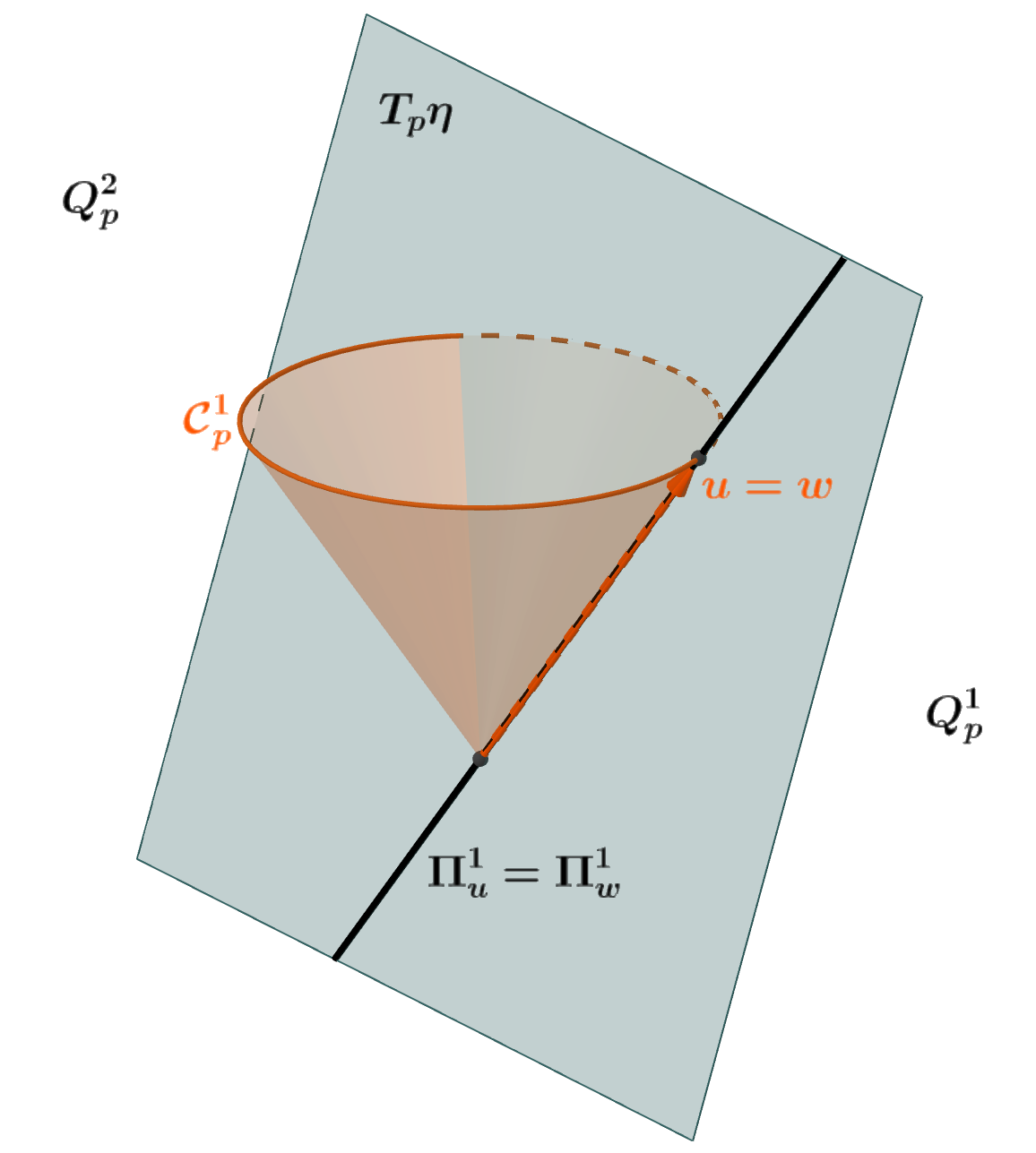}
\caption{Case (A$^*$)(ii).}
\label{fig:case_asii}
\end{subfigure}
\caption{Case (A$^*$) of Corollary~\ref{thm:existence_reflection}, when $ T_p\eta $ is $ \C^1 $-timelike. On the left, $ \Pi_u^1 $ is $ \C^1 $-spacelike and $ w \in Q^1_p $ is the unique reflected direction. On the right, $ \Pi_u^1 $ is $ \C^1 $-lightlike and $ w = u \in T_p\eta $ is the unique reflected direction.}
\end{figure}

\begin{figure}
\centering
\begin{subfigure}{0.4\textwidth}
\includegraphics[width=\textwidth]{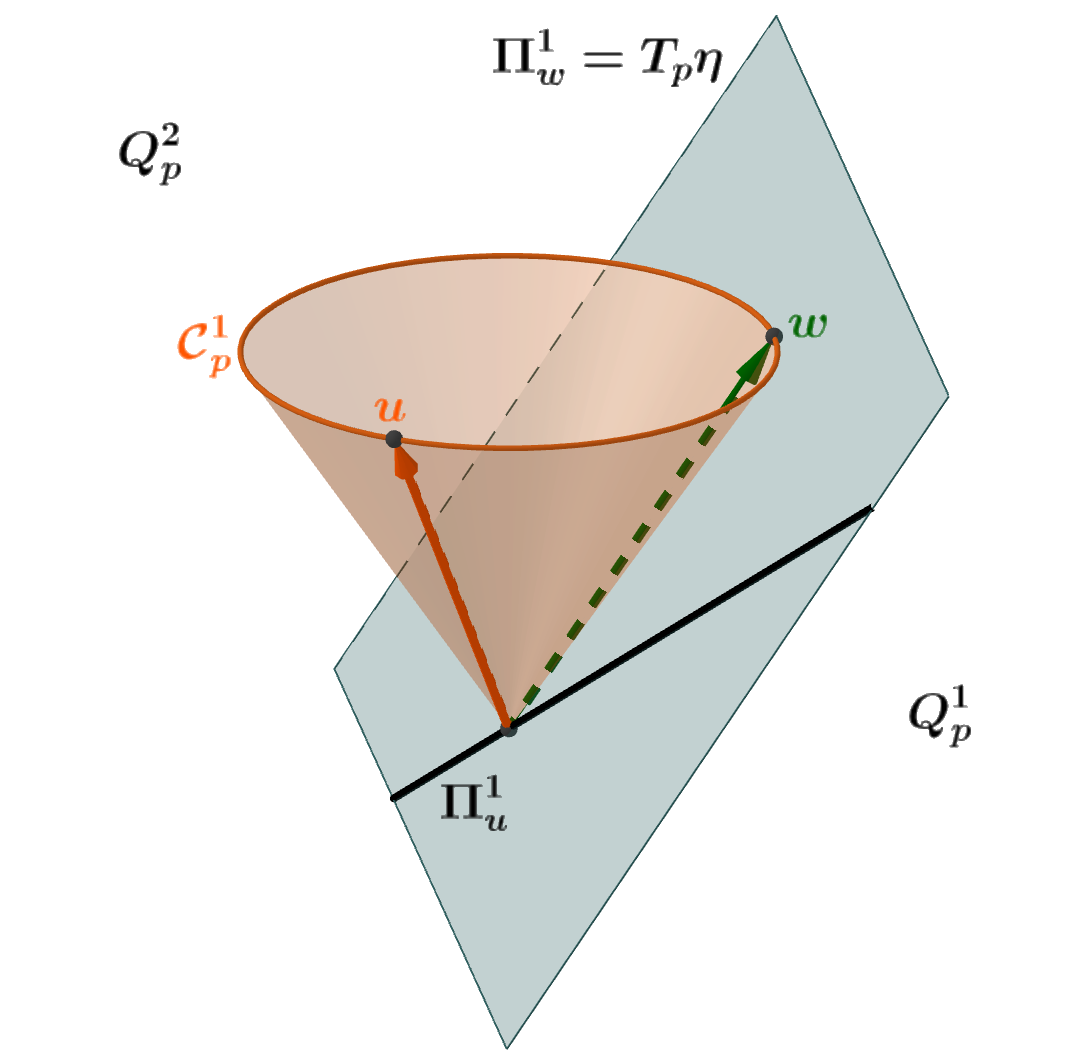}
\caption{Case (B$^*$).}
\label{fig:ex_reflection}
\end{subfigure}
\begin{subfigure}{0.4\textwidth}
\includegraphics[width=\textwidth]{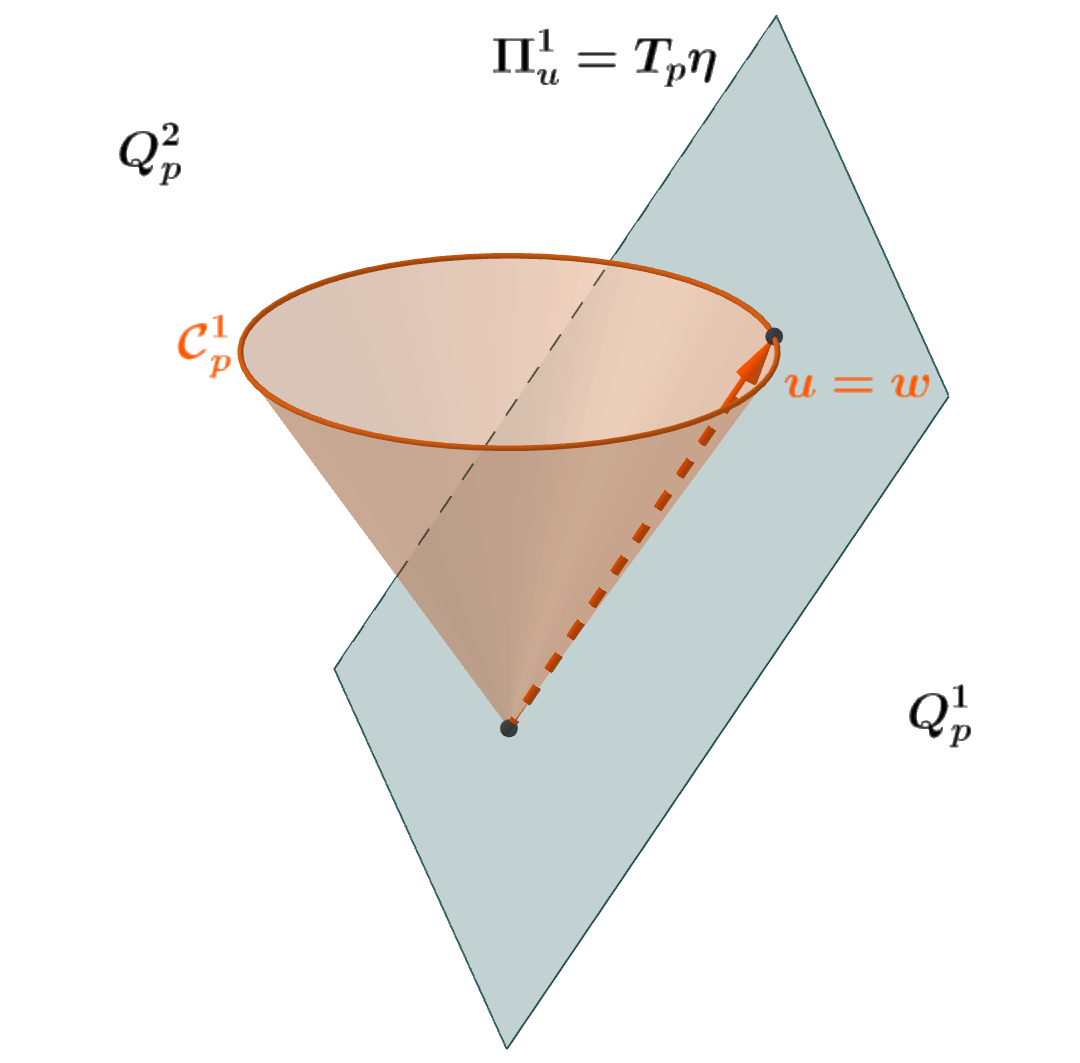}
\caption{$ u^{\perp_{\C^1}} = T_p\eta $.}
\label{fig:ex_reflection2}
\end{subfigure}
\caption{On the left, case (B$^*$) of Corollary~\ref{thm:existence_reflection}, when $ T_p\eta $ is $ \C^1 $-lightlike and $ u $ is $ \C^1 $-transverse to $ \eta $. On the right, the final case of Corollary~\ref{thm:existence_reflection}, when $ u^{\perp_{\C^1}} = T_p\eta $. In both cases, $ \C^1_p $ is tangent to $ T_p\eta $ along the unique reflected direction $ w \in T_p\eta $; however, $ w \not= u $ on the left, whereas $ w = u $ on the right.}
\end{figure}

For reflection, non-existence can be classified into the following cases:
\begin{itemize}
\item {\em No returning}: no $\C^1$-lightlike direction at $ p $ points to $\overline{Q}_1$. Consequently, cone geodesics arriving at $p$ cannot return to $\overline{Q}_1$. This corresponds to the case where $ \C_p^1 \subset Q_p^2 $, i.e. case (C$^*$) of Corollary \ref{thm:existence_reflection}.
\item {\em Not possible}: the causal characters of $T_p\eta$ and $\Pi_u^1$ are incompatible.
\end{itemize}

Table~\ref{tab:reflection} summarizes all possibilities concerning the existence, non-existence, uniqueness and non-uniqueness of reflected directions. Note that $ \C^2 $ plays no role here.

\begin{table}
\begin{center}
\begin{tabular}{cccc}
\multicolumn{4}{c}{\textbf{Reflected directions}} \\
\midrule\midrule
$ \boldsymbol{T_p\eta} \ \boldsymbol{\backslash} \ \boldsymbol{\Pi_u^1} $ & $ \boldsymbol{\C^1} $\textbf{-spacelike} & $ \boldsymbol{\C^1} $\textbf{-lightlike} & $ \boldsymbol{\C^1} $\textbf{-timelike} \\ \midrule\midrule
\multirow{2}{*}{$ \boldsymbol{\C^1} $\textbf{-timelike}} & $ \exists! $ ($ \not\in T_p\eta $) & $ \exists! $ ($ = u \in T_p\eta $) & \multirow{2}{*}{Not possible} \\
 & Figure~\ref{fig:case_asi} & Figure~\ref{fig:case_asii} & \\ \midrule
\multirow{4}{*}{$ \boldsymbol{\C^1} $\textbf{-lightlike}} & & $ \exists! $ ($ = u \in T_p\eta $) if $ (**) $ & \multirow{4}{*}{Not possible} \\
 & $ \exists! $ ($ \in T_p\eta $) & \multirow{2}{*}{Figure~\ref{fig:ex_reflection2}} & \\
 & Figure~\ref{fig:ex_reflection} & & \\
 & & Not possible otherwise & \\
\midrule
$ \boldsymbol{\C^1} $\textbf{-spacelike} & No returning & Not possible & Not possible \\ \midrule\midrule
\multicolumn{4}{c}{$ (**) $: $ u^{\perp_{\C^1}} = T_p\eta $}
\end{tabular}
\captionsetup{singlelinecheck=off}
\caption[Table 2]{All possible reflected directions (from Corollary~\ref{thm:existence_reflection}) associated with an incident direction $ u $ at $ p \in \eta $. The cases of non-existence are classified as follows: {\em no returning} ($ \C_p^1 \subset Q_p^2 $, so cone geodesics cannot return to $ \overline{Q}_1 $); and {\em not possible} (incompatible causal characters of $ T_p\eta $ and $ \Pi_u^1 $). Note that $ T_p\eta $ and $ \Pi_u^1 $ are simultaneously $ \C^1 $-lightlike if and only if $ u^{\perp_{\C^1}} = T_p\eta $ (see Figure~\ref{fig:ex_reflection2}).}
\label{tab:reflection}
\end{center}
\end{table}

\subsection{The total reflection phenomenon}
\label{subsec:total_reflection}
When $ T_p\eta $ is $ \C^1 $- and $ \C^2 $-timelike---which is the most natural situation---case (A)(iii) of Theorem~\ref{thm:existence_refraction} is known as the {\em total reflection phenomenon}: there is no refraction, but the existence of reflection is guaranteed by case (A$^*$) of Corollary~\ref{thm:existence_reflection}, so the incident ray is totally reflected without being able to cross into the second medium.\footnote{There are other exotic situations, e.g. cases (A)(iii) and (B$^*$) can happen simultaneously, in which case there is neither refraction nor reflection.} Case (A)(ii), when the refracted direction is tangent to $ \eta $, marks the limit between the existence and non-existence of the refracted trajectory (see Figure~\ref{fig:case_aii}).

In the classical isotropic case, the total reflection occurs when the speed of the wave in the second medium is greater than in the first one (see Example~\ref{ex:classical_laws} below and also \cite[Lemma~5.1]{MP23} for a similar result in the classical anisotropic case). In our generalized setting, this translates into having incident directions along $ \eta $ which are $ \C^2 $-timelike.

\begin{prop}
Let $ p \in \eta $ such that $ T_p\eta $ is both $ \C^1 $- and $ \C^2 $-timelike. For every incident direction $ u \in \C_p^1 \cap T_p\eta $ which is $ \C^2 $-timelike, there exists an open conic neighborhood $ \Omega \subset \C^1_p $ around $ u $ where there is total reflection---i.e., there is a unique reflection and no refraction for any incident direction in $ \Omega \setminus Q^1_p $.
\end{prop}
\begin{proof}
If $ u \in \C^1_p \cap T_p\eta $ is $ \C^2 $-timelike, then $ \Pi_u^1 $ (which contains $ u $) is also $ \C^2 $-timelike. Therefore, there is an open conic neighborhood $ \Omega \subset \C^1_p $ around $ u $ such that $ \Pi_{\hat{u}}^1 $ is still $ \C^2 $-timelike for any $ \hat{u} \in \Omega $. So, for the incident directions in $ \Omega \setminus Q^1_p $, we are in case (A)(iii) of Theorem~\ref{thm:existence_refraction} and thus, there is no refraction. Additionally, since $ T_p\eta $ is $ \C^1 $-timelike, there is always a unique reflection by case (A$^*$) of Corollary~\ref{thm:existence_reflection}.
\end{proof}

In particular, if $ \C_p^1 $ is entirely contained in the cone domain of $ \C_p^2 $ (e.g., as in Figure~\ref{fig:case_a}), there is total reflection around any $ u \in \C_p^1 \cap T_p\eta $.

\section{Snell cone geodesics}
\label{sec:cone_geodesics}
As we have seen in Theorem~\ref{thm:snell}(I) and Corollary~\ref{thm:reflection}(I), refracted and reflected trajectories (introduced in Definition~\ref{def:directions}) are, in general, critical points of the arrival time functional. Here we will take a step forward and identify the trajectories that can be, in addition, global extrema. Throughout this section, if $ \gamma_u $ is an incident trajectory and $ \gamma_r $ its corresponding reflected or refracted one, we will use the notation $ \gamma_u \cup \gamma_r $ to denote the complete trajectory resulting from the concatenation of $ \gamma_u $ and $ \gamma_r $.

\subsection{Topologically transverse local causality}
\label{subsec:gen_causality}
First, we need to generalize some causality notions already defined in Definition~\ref{def:causality} to the extended setting $ (Q,\C^1 \cup \C^2) $. Since we have so far restricted our study to topologically transverse curves---so that, in particular, incident, refracted, and reflected trajectories are topologically transverse by definition---the ``futures'' and ``pasts'' considered here will be only those generated by such curves. This entails the following restrictions:
\begin{itemize}
\item Causal curves are not allowed to travel along the interface. Although this excludes some interesting global trajectories that ``take advantage'' of the interface to minimize the arrival time (see, e.g., \cite[\S~5]{MP23}), removing this restriction would lead to highly pathological spacetimes in which even the so-called push-up property fails: there may exist points in the chronological future that are not accessible by timelike curves (see Figure~\ref{fig:push_up}). This phenomenon typically occurs when the regularity of the metric is lowered \cite{samann2016} (so that additional hypotheses are required in Lorentzian length spaces to recover this property; see \cite[Theorems~3.20, 4.17 and 4.18]{kunzinger2018}).

\item Causal curves are not allowed to cross the interface more than once. This is sufficient for our purposes here, since we are only concerned with the local behavior of causal curves in an arbitrarily small neighborhood $ \Omega $ of the intersection point with the interface.
\end{itemize}

\begin{figure}
\centering
\includegraphics[width=0.5\textwidth]{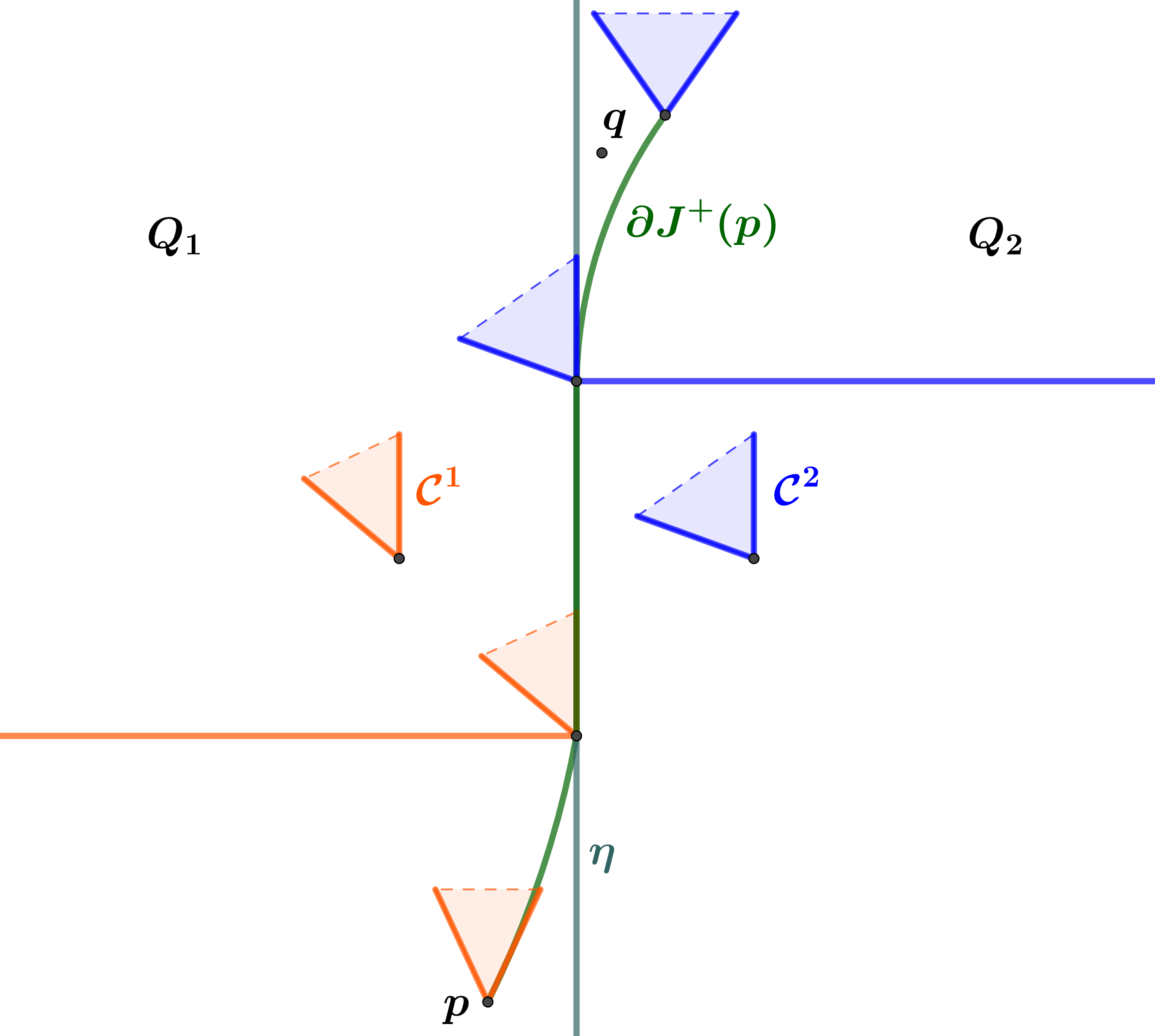}
\caption{Illustration of the failure of the push-up property when non-topologically transverse curves are taken into account to compute the causal futures. The cone structure $ \C^1 $ remains constant above the left half-line, while $ \C^2 $ is constant below the right half-line. The curve $ \partial J^+(p) $ forms the boundary of the causal future of $ p \in Q_1 $, so $ q \in Q_2 $ lies in the interior of $ J^+(p) $ but cannot be reached by any timelike curve.}
\label{fig:push_up}
\end{figure}

\begin{defi}
\label{def:gen_causality}
Let $ \gamma: [a,b] \rightarrow Q $ be a piecewise smooth curve topologically transverse to $\eta $, with $ \gamma(\tau) \in \eta $ for some $ \tau \in (a,b) $, and $ \gamma|_{[a,\tau)} \subset Q_{\mu} $, $ \gamma|_{(\tau,b]} \subset Q_{\nu} $ for some $ \mu, \nu = 1,2 $. We say that $ \gamma $ is {\em timelike}, {\em lightlike}, {\em causal} or {\em spacelike} if both $ \gamma|_{[a,\tau]} $ and $ \gamma|_{[\tau,b]} $ are so in $ (\overline{Q}_{\mu},\C^{\mu}) $ and $ (\overline{Q}_{\nu},\C^{\nu}) $, respectively.

Using these extended causal curves---in addition to the usual $ \C^{\mu} $-causal curves entirely contained in $ \overline{Q}_{\mu} $ (intersecting $ \eta $ at most once)---the {\em (topologically transverse) chronological} and {\em causal futures} and {\em pasts} $ I^{\pm}(p), J^{\pm}(p) $, introduced as in Definition~\ref{def:causality}, are naturally extended to $ (Q,\C^1 \cup \C^2) $. We also define the {\em horismos} as $ E^{\pm}(p) \coloneqq J^{\pm}(p) \setminus I^{\pm}(p) $ and, given an open subset $ \Omega \subset Q $, the subscript $ \Omega $ indicates that we only consider causal curves entirely contained in $ \Omega $ (as in Definition~\ref{def:causality}).
\end{defi}

Theorem~\ref{thm:existence_refraction} and Corollary~\ref{thm:existence_reflection} list all possible refracted and reflected directions. When we look at the trajectories they generate, however, there is a key distinction that extends the notion of cone geodesic.

\begin{defi}
A continuous curve $ \gamma: I \rightarrow Q $ is a {\em Snell cone geodesic} if it is locally horismotic, as in Definition~\ref{def:cone_geodesic}, using the (topologically transverse) chronological and causal futures and pasts extended to $ (Q,\C^1 \cup \C^2) $, as in Definition~\ref{def:gen_causality}.
\end{defi}

\begin{rem}
Note that if $ \gamma: [a,b] \rightarrow Q $ is a refracted (resp. reflected) trajectory, then $ \gamma|_{[a,\tau]} $ is a cone geodesic of $ \C^1 $, $ \gamma|_{[\tau,b]} $ is a cone geodesic of $ \C^2 $ (resp. $ \C^1 $) and both are, in particular, Snell cone geodesics. Therefore, the complete trajectory $ \gamma $ is a Snell cone geodesic if and only if it is locally horismotic at $ \gamma(\tau) \in \eta $.
\end{rem}

The motivation for this notion lies in the fact that cone geodesics of $ \C^{\mu} $ are always (locally) time-minimizing, in the following sense. If $ \gamma $ is a cone geodesic of $ \C^{\mu} $ entirely contained in $ Q_{\mu} $, then for any $ p_0=\gamma(t_0) $ it satisfies, on the one hand, that $ \gamma(t_1) \in E^+_{Q_{\mu}}(p_0) $ for sufficiently small $ t_1-t_0 > 0 $; and on the other hand, if $\gamma(t_2) \in E^+_{Q_\mu}(p_0)$ at some $ t_2>t_1 $, then any (future-directed) $ \C^{\mu} $-timelike curve $ \alpha $ through $ \gamma(t_1) $ enters $ I^+_{Q_\mu}(p_0) $. Therefore, $ \gamma $ is ``first-arriving'' at $ \alpha $, i.e. it is the global minimum of the arrival time functional among $ \C^{\mu} $-causal curves from $ p_0 $ to $ \alpha $.

Snell cone geodesics extend this notion to the setting $ (Q,\C^1 \cup \C^2) $, preserving this minimizing property even when the interface $ \eta $ comes into play. Hence, these curves determine the horismos close to the interface and constitute the unique global extrema (among topologically transverse causal curves). In contrast, when a refracted or reflected trajectory is not a Snell cone geodesic, it cannot be a global time-minimizer (since it already fails to be so locally), although it is still a critical point of the arrival time functional.

Hence, our goal here is to determine whether a refracted or reflected trajectory yields a Snell cone geodesic. In particular, we will restrict our attention to curves transverse to $ \eta $ (in the sense of Definition~\ref{def:top_transverse}). Under this restriction, the classical expectation would be that reflected trajectories are never Snell cone geodesics, while refracted trajectories always are. However, particularly in the case of refraction, this question turns out to be rather subtle: as we will see in the next subsection, an unusual arrangement of the cones (from a classical perspective; see Figure~\ref{fig:reversely}) may disrupt the local horismoticity.

\subsection{Orientation induced by the lightcones}
\label{subsec:orientation}
Interestingly, when the trajectory is transverse to $ \eta $, the property of being a Snell cone geodesic is related to a specific orientation that can be naturally induced on the orthogonal hyperplanes to the incident, refracted and reflected directions. The following definition and results establish a criterion to determine, through a simple construction based on this orientation, whether or not a Snell cone geodesic is obtained.

\begin{notation}
We will only consider orientations in the tangent space at $ p \in \eta $ (therefore, global  orientability will not pose any issue). Specifically, we will focus on the possible concordance of orientations on hyperplanes $ H $ that divide $ T_pQ $. In this context, an {\em orientation} of $ H $ is simply the choice of one of the two sides of $ H $ (i.e., one of the two regions of $ T_pQ \setminus H $) or, equivalently, the choice of an oriented direction transverse to $ H $.
\end{notation}

\begin{defi}
\label{def:orientation}
Let $ r \in \C_p^{\mu} $ be an incident, reflected or refracted direction transverse to $ \eta $ ($ \mu = 1 $ when incident or reflected, $ \mu = 2 $ when refracted). Among the two possible orientations of the hyperplane $ r^{\perp_{\C^{\mu}}} \not= T_p\eta $ (which is tangent to $ \C_p^{\mu} $), consider the one induced by the side where $ \C_p^{\mu} $ is entirely contained. We define $ H_r^{\mu} \subset T_pQ $ as the oriented half-hyperplane $ r^{\perp_{\C^{\mu}}} \cap (Q^{\mu}_p \cup T_p\eta) $ (recall Notation~\ref{not:subspaces}), with the orientation inherited from $ r^{\perp_{\C^{\mu}}} $.

If $ u \in \C_p^1 $ is an incident direction and $ r \in \C_p^{\mu} $ is an associated reflected or refracted direction (both transverse to $ \eta $), then $ H_u^1 \cup H_r^{\mu} \subset T_pQ $ is a broken hyperplane, with the two half-hyperplanes glued along $ \Pi_u^1 = \Pi_r^{\mu} $. This broken hyperplane divides $ T_pQ $ into two subspaces. We say that $ H_u^1 \cup H_r^{\mu} $ is {\em straight oriented} if the orientations of the two half-hyperplanes point to the same subspace (see Figure~\ref{fig:straight}). Otherwise, we say that $ H_u^1 \cup H_r^{\mu} $ is {\em reversely oriented} (see Figure~\ref{fig:reversely}).
\end{defi}

\begin{figure}
\centering
\begin{subfigure}{0.4\textwidth}
\includegraphics[width=\textwidth]{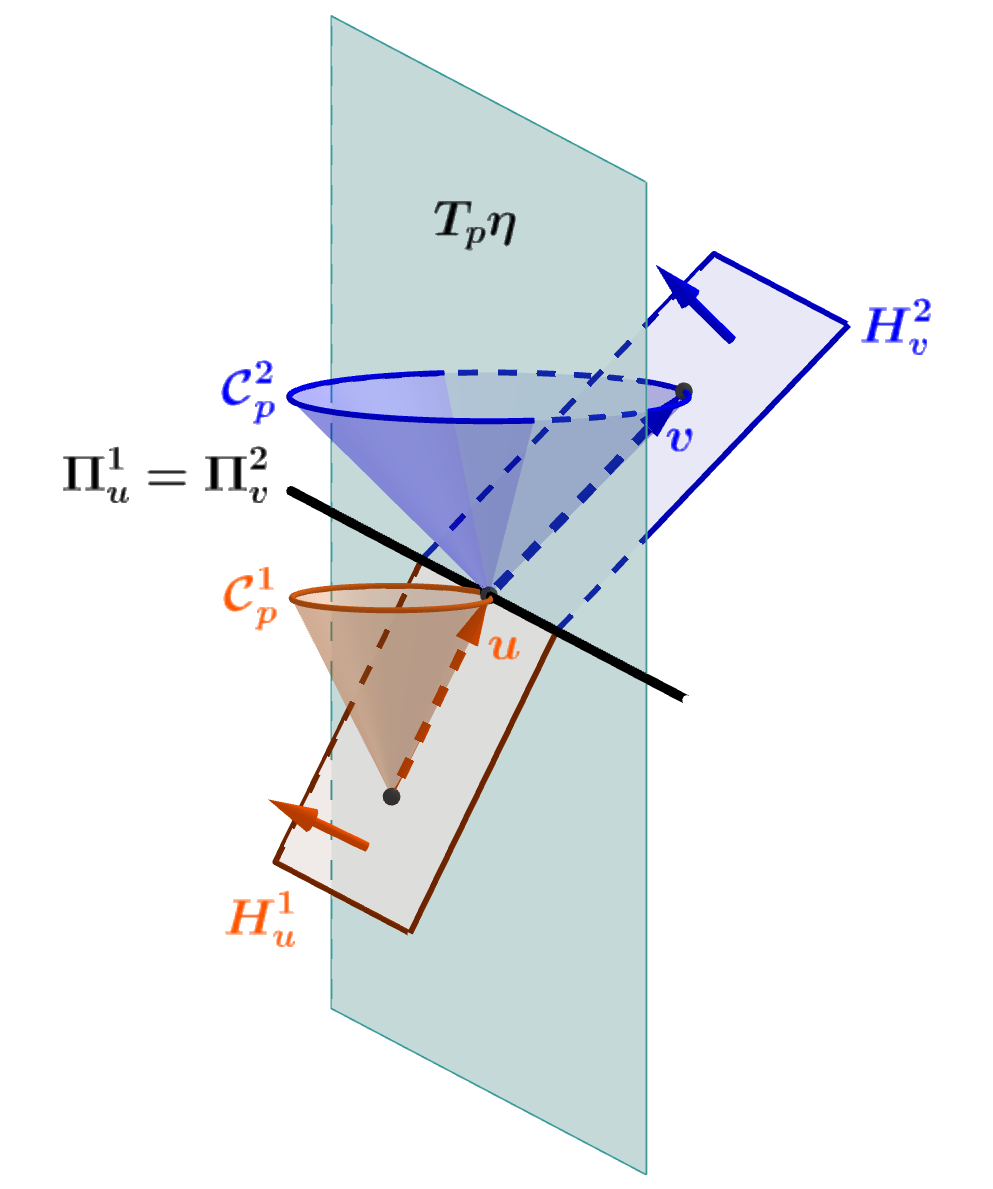}
\caption{$ H_u^1 \cup H_v^2 $ straight oriented.}
\label{fig:straight}
\end{subfigure}
\begin{subfigure}{0.4\textwidth}
\includegraphics[width=\textwidth]{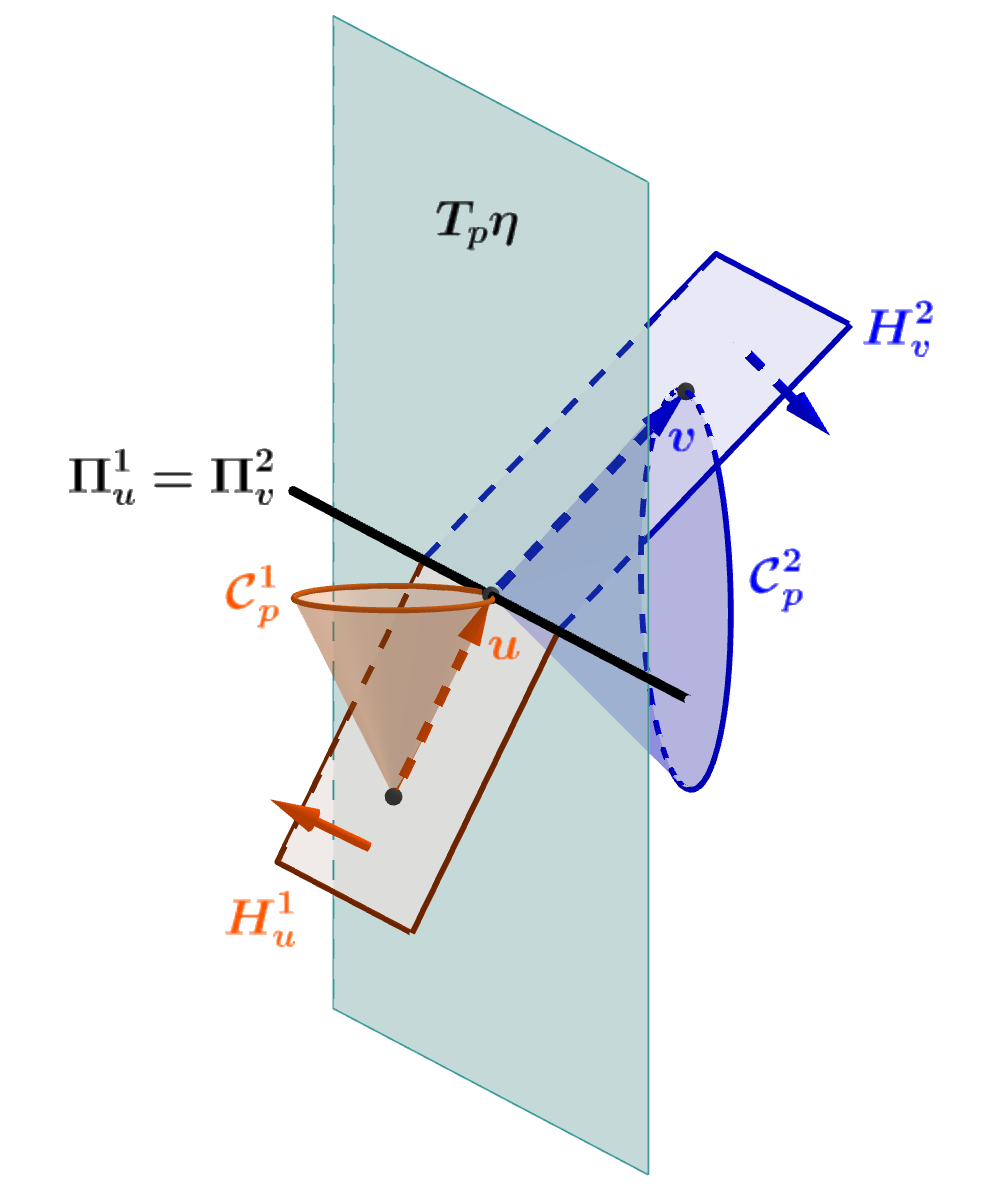}
\caption{$ H_u^1 \cup H_v^2 $ reversely oriented.}
\label{fig:reversely}
\end{subfigure}
\caption{Two different orientations of the broken hyperplane $ H_u^1 \cup H_v^2 $ for the same incident and refracted directions $ u, v $, depending on the position of the cone $\C_p^2$.}
\label{fig:broken_hyperplane}
\end{figure}

\begin{lemma}
\label{lem:strong_convexity}
Let $ \gamma_u: [a,\tau] \rightarrow \overline{Q}_1 $ be an incident trajectory and $ \gamma_r: [\tau,b] \rightarrow \overline{Q}_{\mu} $ an associated reflected or refracted one ($ \mu=1 $ when reflected, $ \mu=2 $ when refracted) at $ p \coloneqq \gamma_u(\tau) = \gamma_r(\tau) \in \eta $, such that $ \gamma_u \cup \gamma_r $ is transverse to $ \eta $. There exists a small coordinate chart $ (\Omega,\varphi) $ around $ p $ adapted to $ \eta $ (namely, $\eta \equiv \{x^n=0\}$) such that $E^+_\Omega(\gamma_u(t))=\partial J^+_{\Omega}(\gamma_u(t))$, $ E^-_{\Omega}(\gamma_r(t))=\partial J^-_{\Omega}(\gamma_r(t)) $ and the following property holds: given $k>0$, for any $ t $ close enough to $ \tau $, $ E^+_\Omega(\gamma_u(t)) \cap \eta $ and $ E^-_\Omega(\gamma_r(t)) \cap \eta $ have a definite second fundamental form at $ p $ (as hypersurfaces of $ \eta $, with respect to the induced Euclidean metric associated with $ (\Omega,\varphi) $) with norm bigger than $ k $.
\end{lemma}
\begin{proof}
We will consider first an incident trajectory $ \gamma_u$, which for simplicity will be denoted $ \gamma $ along the proof. The main idea of the proof is to check first that the second fundamental form of the lightcone of $T_{\gamma(t)}\overline{Q}_1$ intersected with a certain hyperplane is definite and then to check that this property remains true for the image of the lightcone by the exponential map, which, if we are in a convex neighborhood $ \Omega $, is the {\em horismotic lightcone} of $\gamma(t)$:
\begin{equation}\label{e_E}
E^+_\Omega(\gamma(t)):=  J^+_{\Omega}(\gamma(t))\setminus I^+_{\Omega}(\gamma(t))= \textup{exp}_{\gamma(t)}(\C^1_{\gamma(t)}), \quad \text{locally on } \overline{Q}_1.
\end{equation}
For this aim, the proof is divided into four parts.

\medskip
\noindent
(a) {\em  Preliminaries and setting}. We will use the following properties related to the regularity of the exponential map, which are well known to specialists. As a first observation, recall that a Lorentz-Finsler metric $L$ (as well as a Finsler one where $L=F^2$) is $C^2$ at zero if and only if it comes from a semi-Riemannian metric (see \cite[Proposition~4.1]{Warner65} and also \cite[Footnote~16]{JS20}). As a consequence, the Christoffel symbols are not even continuous at the zero section and the exponential map is smooth if and only if the metric is of Berwald type \cite{AZ88}. In any case, the exponential map is always $C^{1,1}$, as proven explicitly by S. Ivanov according to H. Koehler \cite[Proposition~4.1]{Koe14}, and its second derivatives are bounded, as observed by R. Bryant and explained below around \eqref{exploc} and Footnote~\ref{footRobert}. Moreover, as early as 1933, the existence of convex neighborhoods was proven by J. Whitehead \cite{Whitehead}.

Let $\Omega$ be a convex neighborhood of $\gamma(\tau)$ for the extension beyond $\eta$ of a Lorentz-Finsler metric $L_1$ compatible  with  $\C^1$, and assume that $ \gamma $ is parametrized as a geodesic of $ L_1 $ and $\varepsilon>0$ is such that $\gamma([\tau-\varepsilon,\tau])\subset \Omega$.   This means that $ \gamma|_{[\tau-\varepsilon,\tau]} \subset E^+_\Omega(\gamma(\tau-\varepsilon)) \cap \overline{Q}_1 $ and thus, $ \gamma|_{[\tau-\varepsilon,\tau]} $ is a null generator of $ E^+_\Omega(\gamma(\tau-\varepsilon)) \cap \overline{Q}_1 $ (see \cite[Proposition~4]{JP22}, where null generators are identified as integral curves of a unique vector field). Hence, recalling that $u=\dot\gamma(\tau^-)$ and $p=\gamma(\tau)$, we have that $ u^{\perp_{\C^1}} = T_p(E^+_\Omega(\gamma(\tau-\varepsilon)) \cap \overline{Q}_1) $ (see \cite[Proposition~1(2)]{JP22}),  which implies that 
\begin{equation}\label{Piu}
  \Pi_u^1 = T_p\left(E^+_\Omega(\gamma(\tau-\varepsilon) )\cap \eta\right),
\end{equation}  
where we recall that $ \Pi_u^1 \coloneqq u^{\perp_{\C^1}}\cap T_p\eta$.

\medskip
\noindent
(b) {\em  Convexity estimates at the tangent space}.
Since, by \eqref{e_E}, $E^+_\Omega(\gamma(t)) $ is the image by $\exp_{\gamma(t)}$ of (a piece of) $\C^1_{\gamma(t)}$, we will study first the second fundamental form of $\C^1_{\gamma(t)}$ at $(\tau-t)\dot\gamma(t)$ (recall that $\exp_{\gamma(t)}((\tau-t)\dot\gamma(t))=\gamma(\tau)$) in a subspace of its tangent space mapped by the exponential map to $ \Pi_u^1$, because this is the tangent space to $E^+_\Omega(\gamma(t))\cap \eta$. Let us consider a reference frame $E_0,E_1,\ldots,E_{n}$ along $\gamma$ adapted to $ \Pi_u^1$ in the sense that the image by $d(\exp_{\gamma(t)})_{(\tau-t)\dot\gamma(t)}$ of ${\rm Span }\{E_1,\ldots,E_{n-1}\}$ is  $ \Pi_u^1$ (here we are identifying the tangent space to $T_{\gamma(t)}Q$ with itself in the natural way). We will choose as  $E_0$ any timelike vector field along $\gamma$. Moreover, we will choose $E_1(\tau),\ldots,E_{n-1}(\tau)$ as a basis of the subspace $\Pi^1_u$, and then
$E_i(t):=(d(\exp_{\gamma(t)})_{(\tau-t)\dot\gamma(t)})^{-1}(E_i(\tau))$ for $i=1,\ldots,n-1$. Observe that 
\[\Pi^t:={\rm Span}\{E_1(t),\ldots,E_{n-1}(t)\}\]
is a $ \C^1 $-spacelike subspace of $T_{\gamma(t)}Q$, because $d(\exp_{\gamma(t)})$ maps the tangent subspace to $\C^1_{\gamma(t)}$ into the tangent subspace to $E^+_\Omega(\gamma(t))$ and $\Pi^t$ does not contain the only direction in the radical of this tangent subspace, namely, the one generated by $\dot\gamma(t)$ (thanks to the fact that $\gamma$ is transverse to $\eta$). As a consequence, $ \Pi^t $ is tangent to $ \C^1_{\gamma(t)} $. Finally, we will choose $E_n$ in such a way that $\{E_1(t),E_2(t),\ldots,E_n(t)\}$ generates a $ \C^1 $-spacelike hypersurface of $T_{\gamma(t)}Q$. For simplicity, we will also assume that $E_0(t)+E_n(t)=\dot\gamma(t)$, which is always possible by choosing suitably $E_0$ and $E_n$.

In this setting, we have a cone triple for the restriction of $\C^1$ along $\gamma$. Notice that the elements of this triple are determined by $\C^1_{\gamma(t)}$, taking $ E_0 $ as the $ \C^1 $-timelike vector field and the dual one-form with kernel $ {\rm Span}\{E_1(t),\ldots,E_n(t)\}$ (recall Remark~\ref{rem:comments}(4)). Then, we obtain a Finsler metric on the fiber bundle generated by 
$E_1,\ldots,E_n$ along $\gamma$ with the property that if $\pi:T_{\gamma(t)}Q\rightarrow {\rm Span}\{E_1(t),\ldots,E_n(t)\}$ is the natural projection using the basis $E_0(t),\ldots,E_n(t)$, it holds that
\[v\in\C^1_{\gamma(t)} \text{ if and only if } v=F^t(\pi(v))E_0(t)+\pi(v).\]
Now, observe that for every $t\in [\tau-\varepsilon,\tau]$, the intersection of the lightcone $\C^1_{\gamma(t)}$ in the tangent space $T_{\gamma(t)}Q$ with the affine hyperplane 
\[H^t:=E_0(t)+{\rm Span}\{E_1(t),\ldots,E_n(t)\}\] projects into the indicatrix of $F^t$ on ${\rm Span}\{E_1(t),\ldots,E_n(t)\}$.  Moreover, the (affine) second fundamental form $\sigma^t$ of the indicatrix of $F^t$ with respect to the opposite to the position vector,   in this case $-E_n(t)$,  coincides with the fundamental tensor of $F^t$, denoted $g^t$, restricted to the tangent space to the indicatrix, and therefore it is positive definite  (see e.g. \cite[Equation~(2.5)]{JS14}).  Furthermore, as a consequence,  the second fundamental form $\sigma^t_\C$ of $\C^1_{\gamma(t)}$ at $(\tau-t)\dot\gamma(t)$ with respect to $-E_n(t)$
satisfies that 
\begin{equation}\label{sigmatg}
\sigma^t_\C|_{\Pi^t\times\Pi^t}=\frac{1}{\tau-t}g^t|_{\Pi^t\times\Pi^t}.
\end{equation}
Notice that the restriction on the left side corresponds to the second fundamental form of $( \tau-t)H^t \cap \C^1_{\gamma(t)}$, which is related by a homothety to
$H^t \cap \C^1_{\gamma(t)}$, a strongly convex hypersurface of $H^t$ with respect to $-E_n$. 
\medskip
\newline
\emph{Convention:} In the following, we will choose any coordinates in the convex neighborhood adapted to the interface $\eta$, and consider the Euclidean metric associated with these coordinates. We will estimate the values of second fundamental forms using this metric, which consistently implies changing the transversal vector $-E_n$. Anyway, the choice of transversal will be irrelevant for the estimates.

\medskip
\noindent
{\em (c) Estimates of the exponential}. Now we will use the exponential map $\exp_{\gamma(t)}$ to bring all the information about vector lightcones $ \C^1_{\gamma(t)} $ on $T_{\gamma(t)}Q$ to the horismotic lightcones $E^+_{\Omega}(\gamma(t))$ on the convex neighborhood $\Omega$. We are interested in an estimation of the second fundamental form of $E^+_{\Omega}(\gamma(t))$ at $\gamma(\tau)$ in the directions given by $\Pi^1_u$. Next, we will go into details about the exponential map mentioned in the part (a), following some ideas of Robert Bryant\footnote{\label{footRobert}See the link of {\em MathOverflow} (date 19/06/2026): \url{https://mathoverflow.net/questions/479094/how-badly-does-the-geodesic-exponential-map-fail-to-be-c2-on-finsler-manifold}.} to bound the second derivatives of the exponential map.
Let $S$ be the unit sphere of ${\mathds R}^{n+1}$ and identify locally $Q$ with ${\mathds R}^{n+1}$ using the chart $(\Omega,\varphi)$ adapted to $\eta$. Then, it is possible to define a smooth map $E^t:(-\epsilon,\epsilon)\times S\rightarrow \mathds R^{n+1}$ (which also varies smoothly with $t\in [\tau-\varepsilon, \tau]$) in such a way that 
\begin{equation}\label{exploc}
\exp_{\gamma(t)}(r w)= \gamma(t)+rw+r^2E^t(r,w),
\end{equation}
for all $r \in (0,\epsilon)$ and $w\in S$. This is a consequence of the fact that the map ${\rm Exp}^t:(-\epsilon,\epsilon)\times S\rightarrow \mathds R^{n+1}$ defined as ${\rm Exp}^t(r,w)=\gamma_w(r)$, being $\gamma_w$ the geodesic with initial velocity $w\in S\subset {\mathds R}^{n+1}\equiv T_{\gamma(t)}Q$, is smooth as the Christoffel symbols are smooth on the unit spheres of $T_{\gamma(t)} Q$. Indeed, observe that,  given $t$ and $w$, if $f(r)=\exp_{\gamma(t)}(r w)$, then  $f(r)=f(0)+r \int_0^1 \dot f(rs) ds$ by an application of the fundamental theorem of Calculus to  $h(s)=f(rs)$.
As a consequence,
$\frac{f(r)-f(0)}{r}= q(r)$	, with $q$ smooth. Now apply the same idea to $q(r)$ deducing that
$q(r)=q(0)+r\int_0^1 \dot q(rs) ds$. Using that $q(0)=\dot f(0)$ we get
$\frac{f(r)-f(0)}{r}=\dot f(0)+r p(r)$ for a smooth $p$, and then
$f(r)=f(0)+\dot f(0) r+r^2 p(r)$, leading to the expression \eqref{exploc}. As $E^t(r,w)=p(r)$ has been obtained from partial derivatives and integrals of $\exp_{\gamma(t)}(r w)$, it is smooth with respect to $r,w$ and $t$. Furthermore, 
we can choose coordinates $(\theta^1,\ldots,\theta^n)$ in the sphere with $w=(w^0(\theta^1,\ldots,\theta^n),\ldots,w^{n}(\theta^1,\ldots,\theta^n))$. As a consequence, all the partial derivatives with respect to $r$, $\theta^1,\ldots,\theta^n$ and $t$ are smooth and, as  we will see below, the expression \eqref{exploc} and the compactness in the directions $w$ (as well as  the restriction to values of  $t$ close to $\tau$) yield a uniform bound on its second derivatives with respect to the Cartesian coordinates. Of course, the map $E^t(0,w)$ does not take the same value for all $w$, and as a consequence, $\exp_{\gamma(t)}$ is not necessarily $C^2$ at the origin. 

\medskip
\noindent
{\em (d) Convexity estimates for the horismos $E^+_\Omega(\gamma(t))$}. Next, following the Convention above (end of part $(b)$), we will compute the second fundamental form $\hat\sigma^t_\C$ of $\C^1_{\gamma(t)}$ with the above Cartesian coordinates first and, then, the one $\sigma^t_{E^+}$ of $E^+_{\Omega}(\gamma(t))$ and finally $\sigma^t_{{\mathcal P}^t}$ of $E^+_{\Omega}(\gamma(t))\cap \eta$. Consider the normal vector field $ N^t$ at $(\tau-t)\dot\gamma$ pointing to the interior of $\C^1_{\gamma(t)}$. We have that $N^t=f(t) (-E_n(t))+P_t$, with $P_t$ tangent to $\C^1_{\gamma(t)}$ at $(\tau-t)\dot\gamma$  and $f$ smooth and positive on a compact interval. Then, as the change in the transverse vector to compute the second fundamental form only adds the factor $f(t)$, using \eqref{sigmatg}, it follows that
 \[\hat\sigma_\C^t|_{\Pi^t\times \Pi^t}=f(t) \sigma_\C^t|_{\Pi^t\times \Pi^t} =f(t) \frac{1}{(\tau-t)} g^t|_{\Pi^t\times \Pi^t}.\] As $g^t$ is positive definite for every $t$ with smooth dependence, this implies that $\hat\sigma_\C^t|_{\Pi^t\times\Pi^t}$ is also positive definite and unbounded when $t$ gets close to $\tau$. Now, observe that by \eqref{exploc}, the second fundamental form of the horismotic lightcone $E^+_\Omega(\gamma(t))$
in \eqref{e_E} is $\hat\sigma_\C^t$ plus a term depending on the second partial derivatives of $r^2E^t(r,w)$.
The Cartesian coordinates vector fields of $\mathds R^{n+1}\equiv T_{\gamma(t)}Q$ can be expressed as
\[\partial_{x^i}= \frac{x^i}{r}\partial_r+r^{-1} X_i = w^i(\theta^1,\ldots,\theta^n)\partial_r+r^{-1} X_i,\]
being $X_i$ tangent to the sphere.\footnote{Just observe that $\frac{\partial r}{\partial x^i}=\frac{x^i}{r}=w^i$ and locally, if $x^{0}\not=0$ (otherwise choose some $x^j\not=0$ and switch the roles of $x^{0}$ and $x^j$) the coordinates of the sphere $S$ satisfy that $\theta^k=f^k(w^1,\ldots,w^{n})$ for certain smooth functions $f^k$. Therefore, $\partial_{x^i}=\frac{\partial r}{\partial x^i} \partial_r+\frac{\partial \theta^k}{\partial x^i}\partial_{\theta^k}=\frac{x^i}{r}\partial_r+ \frac{1}{r}\left(\delta^l_{\ i}-\frac{x^lx^i}{r^2}\right)\frac{\partial f^k}{\partial w^l}\partial_{\theta^k}$.} This means that when we compute the second derivatives of $r^2 E^t(r,w)$ in Cartesian coordinates, both $r^{-1}$ cancel out with $r^2$ while $w^i$ and its partial derivatives are bounded as they are defined on the sphere. This implies that these derivatives are bounded when computed close to  $r=0$. As a consequence, given $k>0$, by taking $t$ close enough to $\tau$, we can assume that the second fundamental form $\sigma^t_{E^+}$ of the horismotic lightcone $E^+_{\Omega}(\gamma(t)) $ restricted to $\Pi^1_{u}$ is positive definite and has norm bigger than $k$. Now observe that the submanifold ${\mathcal P}^t \coloneqq E^+_{\Omega}(\gamma(t))\cap \eta $ has also $\Pi^1_{u}$ as a tangent space and by the Gauss formula, any curve $\beta:(-\delta,\delta)\rightarrow{\mathcal P}^t $ parametrized by the arc length with $\beta(0)=\gamma(\tau)$ and $\dot\beta(0)=w$ satisfies 
\begin{equation}\label{gaussformula}
\|\ddot\beta(0)\|^2=\|(\ddot\beta(0))^\top\|^2+\sigma^t_{E^+}(w,w)^2,
\end{equation}
where $(\ddot\beta(0))^\top$ denotes the tangent part to $E^+_{\Omega}(\gamma(t)) $.
Finally, if we denote by $\sigma_{\mathcal P^t}$ the second fundamental form  of ${\mathcal P}^t$ at $\gamma(\tau)$,  when considering as $\beta$ a unit geodesic of ${\mathcal P}^t$ (with respect to the induced Euclidean metric on $\mathcal{P}^t$), it follows that
$\|\ddot\beta(0)\|=\sigma_{\mathcal P^t}(w,w)$ and, using  \eqref{gaussformula}:
\[\sigma_{\mathcal P^t}(w,w)^2=\|(\ddot\beta(0))^\top\|^2+\sigma^t_{E^+}(w,w)^2\geq \sigma^t_{E^+}(w,w)^2\geq k^2,\]
which concludes that the norm of $\sigma_{\mathcal P^t}$ is also bigger than $k$.\footnote{Consistently with the Convention above, here we have used the Euclidean metric to write Gauss formula and curvature estimates in a standard Euclidean way. However, the technicality of changing the transverse vector field from $-E_n$ to the Euclidean $N^t$ could be avoided just considering the affine Gauss equation for $-E_n$ (as, for example, in \cite[Proposition 1.1]{NS94}) and using the auxiliary Euclidean metric just to determine unit vectors and verify the divergence of the second fundamental form.}

The proof for the associated refracted or reflected trajectory is analogous just considering the reverse exponential map.
\end{proof}

\begin{thm}
\label{thm:orientation}
Let $ \gamma_u: [a,\tau] \rightarrow \overline{Q}_1 $ be an incident trajectory and $ \gamma_r: [\tau,b] \rightarrow \overline{Q}_{\mu} $ an associated reflected or refracted one ($ \mu=1 $ when reflected, $ \mu=2 $ when refracted), such that $ \gamma_u \cup \gamma_r $ is transverse to $ \eta $. Then, $ \gamma_u \cup \gamma_r $ is a Snell cone geodesic if and only if the broken hyperplane $ H_u^1 \cup H_r^{\mu} $ is straight oriented.
\end{thm}
\begin{proof}
Let $ \Omega \subset Q $ be an open neighborhood of $ p \coloneqq \gamma_u(\tau) $ (where we restrict the computation of the causal futures and pasts) satisfying Lemma~\ref{lem:strong_convexity}, which in particular is contained in a convex neighborhood with coordinates adapted to $\eta$,  and take $ \varepsilon > 0 $ small enough so that $ q_{\varepsilon} \coloneqq \gamma_u(\tau-\varepsilon) \in \Omega $ and $ q_{\varepsilon} \rightarrow_{\Omega} p $, i.e. $ p \in E^+_{\Omega}(q_{\varepsilon}) $. Moreover, by \eqref{Piu},  $\Pi_u^1 = T_p(E^+_{\Omega}(q_{\varepsilon}) \cap \eta) $. Analogously, since $ \gamma_r $ is a cone geodesic of $ \C^{\mu} $, we can find $ q_{\delta} \coloneqq \gamma_r(\tau+\delta) \in \Omega $ with $ \delta > 0 $ small enough so that $ r^{\perp_{\C^{\mu}}} = T_p(E^-_{\Omega}(q_{\delta}) \cap \overline{Q}_{\mu}) $ and hence, $ \Pi_r^{\mu} = T_p(E^-_{\Omega}(q_{\delta}) \cap \eta) $. Observe that $ \Pi_u^1 = \Pi_r^{\mu} $ by hypothesis, so we deduce that $ E^+_{\Omega}(q_{\varepsilon}) \cap \eta $ and $ E^-_{\Omega}(q_{\delta}) \cap \eta $ must be tangent to each other at $ p $. By reducing $ \varepsilon $ and $\Omega$ if necessary, we can ensure that $ E^+_{\Omega}(q_{\varepsilon}) \cap \eta $ (and thus, $ J^+_{\Omega}(q_{\varepsilon}) \cap \eta $) is entirely contained in one of the two sides of $ \Omega \cap \eta $ divided by the affine hyperplane $p+\Pi_u^1$. Indeed, Lemma~\ref{lem:strong_convexity} ensures that $ E^+_{\Omega}(q_{\varepsilon}) \cap \eta $ has definite second fundamental form at $p$ and, by continuity, also in some neighborhood of $p$. By reducing  $\Omega$, we can assume that the intersection of $\Omega$ with $ E^+_{\Omega}(q_{\varepsilon}) \cap \eta $ is connected and with definite second fundamental form. Analogously, $ E^-_{\Omega}(q_{\delta}) \cap \eta $ lies also in one of the two mentioned sides and has definite second fundamental form. Specifically, by construction $ J^+_{\Omega}(q_{\varepsilon}) \cap \eta $ and $ J^-_{\Omega}(q_{\delta}) \cap \eta $ lie in opposite sides if and only if $ H_u^1 \cup H_r^{\mu} $ is straight oriented or, equivalently, they lie in the same side if and only if $ H_u^1 \cup H_r^{\mu} $ is reversely oriented. Therefore:
\begin{itemize}
\item When $ H_u^1 \cup H_r^{\mu} $ is straight oriented, then $ J^+_{\Omega}(q_{\varepsilon}) \cap \eta \cap J^-_{\Omega}(q_{\delta}) = \{p\} $, which implies that $ I^+_{\Omega}(q_{\varepsilon}) \cap \eta \cap I^-_{\Omega}(q_{\delta}) = \emptyset $. Hence, there is no timelike curve in $ \Omega $ going from $ q_{\varepsilon} $ to $ q_{\delta} $, so $ q_{\delta} \in \partial J^+_{\Omega}(q_{\varepsilon}) $ and we conclude that $ \gamma_u \cup \gamma_r $ is locally horismotic at $ p $, i.e. it is a Snell cone geodesic.
\item When $ H_u^1 \cup H_r^{\mu} $ is reversely oriented, then there are points in $ I^+_{\Omega}(q_{\varepsilon}) \cap \eta \cap I^-_{\Omega}(q_{\delta}) $ arbitrarily close to $ p $, i.e. there exist timelike curves from $ q_{\varepsilon} $ to $ q_{\delta} $ arbitrarily close to $ \gamma_u \cup \gamma_r|_{[\tau-\varepsilon,\tau+\delta]} $, which violates the local horismoticity at $ p $. Therefore, $ \gamma_u \cup \gamma_r $ is not a Snell cone geodesic.
\end{itemize}
\end{proof}

In conclusion, as already anticipated, a refracted trajectory is not always a Snell cone geodesic (as in Figure~\ref{fig:reversely}). On the other hand, the classical expectation is true for reflection: a non-smooth reflected trajectory cannot be a Snell cone geodesic, as we show next.

\begin{cor}
\label{cor:cone_reflected}
Let $ \gamma_u: [a,\tau] \rightarrow \overline{Q}_1 $ be an incident trajectory and $ \gamma_w: [\tau,b] \rightarrow \overline{Q}_1 $ an associated reflected one, such that $ \gamma_u \cup \gamma_w $ is transverse to $ \eta $ (case (A$^*$)(i) of Corollary~\ref{thm:existence_reflection}). Then, $ \gamma_u \cup \gamma_w $ is not a Snell cone geodesic.
\end{cor}
\begin{proof}
Let $ p \coloneqq \gamma_u(\tau) = \gamma_w(\tau) $. By hypothesis, $ T_p\eta $ is $ \C^1 $-timelike and $ u, w \in \C_p^1 $, with $ u\not= w $. Therefore, any $ \C^1 $-timelike vector at $ p $ defines the orientation of both half-hyperplanes $ H^1_u, H^1_w $. Since $ H^1_u $ and $ H^1_w $ are contained in the same side of $ T_p\eta $, this means that $ H^1_u \cup H^1_w $ must be reversely oriented, so $ \gamma_u \cup \gamma_w $ is not a Snell cone geodesic by Theorem~\ref{thm:orientation}.
\end{proof}

\begin{rem}
Note that, in both Theorem~\ref{thm:orientation} and Corollary~\ref{cor:cone_reflected}, we require the curve to be transverse to $ \eta $. This hypothesis cannot be weakened to cone-transversality---to match the hypothesis level used in Theorem~\ref{thm:snell}(II) and Corollary~\ref{thm:reflection}(II)---because, when the curve is tangent to $ \eta $, on the one hand we cannot extend the construction of the broken hyperplane given in Definition~\ref{def:orientation}, and on the other hand, the techniques used in the proof of Lemma~\ref{lem:strong_convexity} are no longer applicable (in particular, the transversality of the curve is used explicitly in part (b) of the proof).
\end{rem}

\subsection{The double refraction phenomenon}
\label{subsec:double_refraction}
In contrast with the total reflection phenomenon described in \S~\ref{subsec:total_reflection} (where there is no refraction), case (C) of Theorem~\ref{thm:existence_refraction} is the only situation in which two distinct refracted directions transverse to $ \eta $ arise (see Figure~\ref{fig:case_c}). We refer to this as the {\em double refraction phenomenon}.\footnote{Case (B)(i) of Theorem~\ref{thm:existence_refraction} also exhibits two different refracted directions (see Figure~\ref{fig:case_bi}), although one of them is tangent to $ \eta $ and satisfies Snell's law as a strict inclusion, which, as explained in Remark~\ref{rem:directions}(5), is not expected to occur in physical applications.}

The fact that there are two possible refractions might seem unrealistic at first glance, as one may wonder what happens to a particle (e.g., a photon) that undergoes this situation. However, the following result shows that both refracted directions are not equivalent: only one of them generates a Snell cone geodesic. This resolves the physical ambiguity, since free particles always tend to follow locally horismotic paths (unless forced to choose other trajectories, e.g. when reflected).

\begin{prop}
\label{prop:double_refraction}
Let $ \gamma_u: [a,\tau] \rightarrow \overline{Q}_1 $ be an incident trajectory and $ \gamma_v, \gamma_{\hat{v}}: [\tau,b] \rightarrow \overline{Q}_2 $ two associated refracted trajectories, such that $ \gamma_u \cup \gamma_v $ and $ \gamma_u \cup \gamma_{\hat{v}} $ are transverse to $ \eta $ (case (C) of Theorem~\ref{thm:existence_refraction}). Then, either $ \gamma_u \cup \gamma_v $ is a Snell cone geodesic and $ \gamma_u \cup \gamma_{\hat{v}} $ is not, or vice versa.
\end{prop}
\begin{proof}
Consider the three half-hyperplanes $ H_u^1 $, $ H_v^2 $ and $ H_{\hat{v}}^2 $, glued along $ \Pi_u^1 = \Pi_v^2 = \Pi_{\hat{v}}^2 $. Since $ \C_p^2 $ is located between $ H_v^2 $ and $ H_{\hat{v}}^2 $ (it is tangent to both half-hyperplanes), the orientation induced on $ H_v^2 $ is opposite to the one induced on $ H_{\hat{v}}^2 $, i.e. if $ H_u^1 \cup H_v^2 $ is straight oriented, then $ H_u^1 \cup H_{\hat{v}}^2 $ must be reversely oriented, and vice versa. Applying Theorem~\ref{thm:orientation}, we conclude that either $ \gamma_u \cup \gamma_v $ is a Snell cone geodesic and $ \gamma_u \cup \gamma_{\hat{v}} $ is not, or vice versa.
\end{proof}

Observe that the double refraction occurs when there is a discontinuity of the medium in time (at least for $ Q_2 $) and thus, it goes beyond the classical intuition and interpretation of Snell's law and the refraction phenomenon. Nevertheless, it may be of independent interest for various applications, such as the discretization of spacetimes (see \S~\ref{sec:discretization} below), where such discontinuities naturally arise when discretizing in time and only refracted trajectories that are Snell cone geodesics must be taken into account to determine the horismos.

\section{Physical interpretations and applications}
\label{sec:applications}
In this last section we analyze some interesting physical situations and applications.

\subsection{A natural physical case}
\label{sec:natural}
Although all the results we have presented so far depend solely on the cone structures, it is usually preferable to choose specific Lorentz-Finsler metrics when modeling particular physical phenomena, as this facilitates explicit computations. With this in mind, the following choices arise naturally when describing the propagation of anisotropic waves \cite{JPS21,P24}---such as sound waves \cite{gibbons2010,gibbons2011}, wildfire fronts \cite{JPS23,markvorsen2016}, and especially seismic waves \cite{antonelli2003,yajima2009}, where refraction between different layers of the Earth plays a key role (see \cite{BS02,SW99})---or the travel of moving objects in a current, as in the so-called Zermelo's navigation problem (see e.g. \cite{bao2004,CJS14,JS20,MPR25,P24b}):
\begin{itemize}
\item Underlying manifold: $ Q = \mathds{R} \times N $, with $ N $ representing the {\em space} and the natural projection $ t: Q \rightarrow \mathds{R} $ representing the {\em time} (in a non-relativistic context).
\item Interface: $ \eta = \mathds{R} \times \tilde{\eta} $, so that $ \tilde{\eta} $ is a hypersurface of $ N $ and $ Q_1 = \mathds{R} \times N_1 $, $ Q_2 = \mathds{R} \times N_2 $, where $ N_1, N_2 $ are open subsets of $ N $ whose closures $ \overline{N}_1, \overline{N}_2 $ are smooth manifolds with the same boundary $ \tilde{\eta} $.
\item Metrics: let $ F_1^t, F_2^t $ be two time-dependent Finsler metrics defined on $ N_1, N_2 $, respectively, whose indicatrices (unit vectors) represent the velocity vectors of a wave or moving object---as measured by the observers' reference frame $ \partial_t $---varying at each point, time and direction. We assume that $ F_1^t, F_2^t $ can be smoothly extended to $ \tilde{\eta} $. Consider then the associated Lorentz-Finsler metrics $ L_1 = dt^2-(F_1^t)^2, L_2 = dt^2-(F_2^t)^2 $ defined on $ Q_1, Q_2 $, respectively, with lightcones $ \C^{\mu} = L_{\mu}^{-1}(0) \cap dt^{-1}((0,\infty)) $ (recall Remark~\ref{rem:comments}(5)). Observe that $ \eta $ and $ \partial_t $ are $ L_1 $- and $ L_2 $-timelike, and $ t $ is a temporal function---so both Finsler spacetimes are stably causal.
\end{itemize}

In this setting, we are interested in the arrival time of lightlike curves $ \gamma $ from an $ L_1 $-spacelike submanifold $ P \subset \{a\} \times N_1 $ to the vertical line $ \alpha: t \mapsto (t,x) $ (an observer in $ \partial_t $), where $ x \in N_2 $ is a fixed point in the space and $ \gamma $ meets $ \eta $ once. Assuming $ \gamma(t) = (t,\sigma(t)) $ is $ t $-parametrized, with $ \gamma(a) \in P $, $ \gamma(\tau) \in \eta $ and $ \gamma(b) \in \textup{Im}(\alpha) $, note that its arrival time is $ \T[\gamma] = \alpha^{-1}(\gamma(b)) = b $ (this is the time measured by $ \alpha $). Moreover, $ \gamma $ satisfies
\begin{equation*}
\begin{split}
L_1(\dot{\gamma}(t)) = 0 & \Leftrightarrow F_1^t(\dot{\sigma}(t)) = 1, \quad \forall t \in [a,\tau], \\
L_2(\dot{\gamma}(t)) = 0 & \Leftrightarrow F_2^t(\dot{\sigma}(t)) = 1, \quad \forall t \in [\tau,b],
\end{split}
\end{equation*}
so $ t $-parametrized lightlike curves $ \gamma $ in the spacetime project onto spatial trajectories $ \sigma $ of the wave or object ($ \dot{\sigma}(t) $ coincides with the spatial velocity, as it is $ F_{\mu}^t $-unitary). This way, we effectively model the travel of such a wave or moving object between two different media, propagating with a time-dependent anisotropic velocity.\footnote{This can be regarded as a generalization of the time-dependent Zermelo's navigation problem (see \cite{P24}) when there is a discontinuity in the medium of propagation.}

In this particular setting, Theorems~\ref{thm:snell} and \ref{thm:orientation} take the following form, thus generalizing the anisotropic Snell's law \cite[Theorem~3.4]{MP23} to the time-dependent (rheonomic) case.

\begin{cor}[Refraction for the anisotropic rheonomic extension of the classical setting]
\label{cor:snell_physical}
In the setting presented above, let $ \gamma \in \mathcal{N}_{P,\alpha} $ be $ t $-parametrized. Then, $ \gamma(t) = (t,\sigma(t)) $ is a critical point of $ \T $ (recall Definition~\ref{def:arrival_time}) if and only if $ \dot{\sigma}(a) \perp_{F_1^a} P $, $ \gamma|_{[a,\tau]} $ is a ($ t $-parametrized) pregeodesic of $ L_1 $, $ \gamma|_{[\tau,b]} $ is a ($ t $-parametrized) pregeodesic of $ L_2 $ and
\begin{equation}
\label{eq:snell_natural}
g^{F_1^\tau}_{\dot{\sigma}(\tau^-)}(\dot{\sigma}(\tau^-),\tilde{z}) = g^{F_2^\tau}_{\dot{\sigma}(\tau^+)}(\dot{\sigma}(\tau^+),\tilde{z}), \quad \forall \tilde{z} \in T_{\sigma(\tau)}\tilde{\eta}.
\end{equation}
Moreover, if this holds and $\gamma$ is transverse to $\eta$, then $\gamma$ is a Snell cone geodesic.

In particular, when $ F_1^t = F_1, F_2^t = F_2 $ are time-independent, then $ \gamma(t) = (t,\sigma(t)) $ is a critical point of $ \T $ if and only if $ \dot{\sigma}(a) \perp_{F_1} P $, $ \sigma|_{[a,\tau]} $ is a unit-speed geodesic of $ F_1 $, $ \sigma|_{[\tau,b]} $ is a unit-speed geodesic of $ F_2 $ and \eqref{eq:snell_natural} holds.
\end{cor}
\begin{proof}
Throughout this proof, indices $ i,j,k,r $ run from $ 0 $ to $ n $, while $ \lambda,\nu,\xi $ run from $ 1 $ to $ n $. Let $ \varphi = (t=x^0,\ldots,x^n) $ be a coordinate system on $ Q $, so $ \tilde{\varphi} = (x^1,\ldots,x^n) $ is a coordinate system on $ N $. Note first that the timelike character of $ \eta $ and $ \alpha $ ensures that the restrictions in Remark~\ref{rem:restrictions} are always satisfied. In particular, the converse of Theorem~\ref{thm:snell}(I) holds and Snell's law \eqref{eq:snell_L} is equivalent to its version with equality \eqref{eq:snell_equality}. Also, consider the equivalent formulation of Theorem~\ref{thm:snell}(I) in terms of $ L_1, L_2 $, given at the beginning of its proof.

For the time-dependent case, observe that the relationship between the fundamental tensors of $ L_{\mu} = dt^2-(F_{\mu}^t)^2 $ and $ F_{\mu}^t $ is
\begin{equation*}
g^{L_{\mu}}_{v}(u,w) = u^0 w^0 - g^{F_{\mu}^t}_{\tilde{v}}(\tilde{u},\tilde{w}), \quad \mu=1,2,
\end{equation*}
for any $ v = (v^0,\tilde{v}) \in \C^{\mu}_{(t,x)} $ and any vectors $ u = (u^0,\tilde{u}), w = (w^0,\tilde{w}) \in T_{(t,x)}Q_{\mu} \equiv \mathds{R} \times T_xN_{\mu} $, with $ (t,x) \in Q_{\mu} $. In matrix form:
\begin{equation*}
(g^{L_{\mu}}_{v})_{ij} = \left(
\begin{array}{cccc}
1 & 0 & \cdots & 0 \\
0 & \ & \ & \ \\
\vdots & \ & -(g^{F_{\mu}^t}_{\tilde{v}})_{\lambda\nu} & \ \\
0 & \ & \ & \
\end{array}\right),
\quad
(g^{L_{\mu}}_{v})^{ij} = \left(
\begin{array}{cccc}
1 & 0 & \cdots & 0 \\
0 & \ & \ & \ \\
\vdots & \ & -(g^{F_{\mu}^t}_{\tilde{v}})^{\lambda\nu} & \ \\
0 & \ & \ & \
\end{array}\right).
\end{equation*}
Since $ P \subset \{a\} \times N_1 $, every $ z \in T_{\gamma(a)}P $ is of the form $ z = (0,\tilde{z}) $, with $ \tilde{z} \in T_{\sigma(a)}P $, so
\begin{equation*}
\begin{split}
\dot{\gamma}(a) \perp_{L_1} P & \Leftrightarrow g^{L_1}_{\dot{\gamma}}(\dot{\gamma},z) = 0, \quad \forall z \in T_{\gamma(a)}P \\
& \Leftrightarrow g^{F_1^a}_{\dot{\sigma}}(\dot{\sigma},\tilde{z}) = 0, \quad \forall \tilde{z} \in T_{\sigma(a)}P \Leftrightarrow \dot{\sigma}(a) \perp_{F_1^a} P.
\end{split}
\end{equation*}
It only remains to show that \eqref{eq:snell_equality} reduces to \eqref{eq:snell_natural}. Since $ \eta = \mathds{R} \times \tilde{\eta} $, $ z = (z^0,\tilde{z}) \in T_{\gamma(\tau)}\eta $ if and only if $ \tilde{z} \in T_{\sigma(\tau)}\tilde{\eta} $. Then, \eqref{eq:snell_equality} becomes
\begin{equation*}
g^{F_1^\tau}_{\dot{\sigma}(\tau^-)}(\dot{\sigma}(\tau^-),\tilde{z}) = z^0 \Leftrightarrow g^{F_2^\tau}_{\dot{\sigma}(\tau^+)}(\dot{\sigma}(\tau^+),\tilde{z}) = z^0, \quad \forall \tilde{z} \in T_{\sigma(\tau)}\tilde{\eta},
\end{equation*}
and notice that $ z^0 \in \mathds{R} $ is arbitrary, so this condition is equivalent to \eqref{eq:snell_natural}.

If $ \gamma $ is a critical point of $ \mathcal{T} $ transverse to $ \eta $, then the orientation on the half-hyperplanes $ H^1_{\dot{\gamma}^-}, H^2_{\dot{\gamma}^+} $ is given by $ \partial_t|_p $, as this vector is both $ \C^1 $- and $ \C^2 $-timelike. Since $ H^1_{\dot{\gamma}^-} $ and $ H^2_{\dot{\gamma}^+} $ are contained in opposite sides of $ T_p\eta $, this means that $ H^1_{\dot{\gamma}^-} \cup H^2_{\dot{\gamma}^+} $ is straight oriented, so $ \gamma $ is a Snell cone geodesic by Theorem~\ref{thm:orientation}.

Finally, the claim for the time-independent case is straightforward from the fact that $ t $-parametrized lightlike pregeodesics of $ L_{\mu} $ project onto unit-speed geodesics of $ F_{\mu} $ (see e.g. \cite[Proposition~B.1]{CS18} and \cite[Proposition~4.1]{CJS11}). Indeed, the ordinary differential equation system for a $ t $-parametrized curve $ \gamma(t) = (t,\sigma(t)) $ to be a pregeodesic of $ L_{\mu} $ is (see \cite[\S~4.2]{JPS21})
\begin{equation}
\label{eq:L_pregeodesics}
\ddot{\sigma}^{\xi} = -\gamma^{\xi}_{\ ij}(\dot{\gamma}) \dot{\gamma}^i \dot{\gamma}^j + \gamma^0_{\ ij}(\dot{\gamma}) \dot{\gamma}^i \dot{\gamma}^j \dot{\sigma}^{\xi}, \quad \xi = 1,\ldots,n,
\end{equation}
where $ \gamma^k_{\ ij} $ are the formal Christoffel symbols of $ L_{\mu} $, given by \eqref{eq:christoffel}. In the time-independent case, $ \gamma_{\ ij}^0(v) = \gamma_{\ i0}^k(v) = 0 $ and $ \gamma_{\ \lambda\nu}^{\xi}(v) = \tilde{\gamma}_{\ \lambda\nu}^{\xi}(\tilde{v}) $ for any $ v = (v^0,\tilde{v}) \in \C^{\mu}_{(t,x)} $, being $ \tilde{\gamma}_{\ \lambda\nu}^{\xi} $ the formal Christoffel symbols of $ F_{\mu} $ (defined analogously). Therefore, \eqref{eq:L_pregeodesics} reduces to
\begin{equation*}
\ddot{\sigma}^{\xi} = -\tilde{\gamma}_{\ \lambda\nu}^{\xi}(\dot{\sigma}) \dot{\sigma}^{\lambda} \dot{\sigma}^{\nu}, \quad \xi = 1,\ldots,n,
\end{equation*}
which are the geodesic equations of $ F_{\mu} $. Moreover, since $ \gamma $ is lightlike,
\begin{equation*}
L_{\mu}(\dot{\gamma}(t)) = 0 \Leftrightarrow F_{\mu}(\dot{\sigma}(t)) = 1,
\end{equation*}
so $ \gamma $ is a $ t $-parametrized lightlike pregeodesic of $ L_{\mu} $ if and only if $ \sigma $ is a unit-speed geodesic of $ F_{\mu} $.
\end{proof}

Analogously, the particularization of Corollaries~\ref{thm:reflection} and \ref{cor:cone_reflected} to this setting generalizes the anisotropic law of reflection \cite[Theorem~4.1]{MP23} to the time-dependent case.

\begin{cor}[Reflection for the anisotropic rheonomic extension of the classical setting]
\label{cor:reflection_physical}
In the setting presented above, let $ \gamma \in \mathcal{N}^*_{P,\alpha} $ be $ t $-parametrized. Then, $ \gamma(t) = (t,\sigma(t)) $ is a critical point of $ \T^* $ (recall Definition~\ref{def:arrival_time_reflection}) if and only if $ \dot{\sigma}(a) \perp_{F_1^a} P $, $ \gamma|_{[a,\tau]} $ and $ \gamma|_{[\tau,b]} $ are ($ t $-parametrized) pregeodesics of $ L_1 $, and
\begin{equation}
\label{eq:reflection_natural}
g^{F_1^\tau}_{\dot{\sigma}(\tau^-)}(\dot{\sigma}(\tau^-),\tilde{z}) = g^{F_1^\tau}_{\dot{\sigma}(\tau^+)}(\dot{\sigma}(\tau^+),\tilde{z}), \quad \forall \tilde{z} \in T_{\sigma(\tau)}\tilde{\eta}.
\end{equation}
Moreover, if this holds and $\gamma$ is transverse to $\eta$, then $\gamma$ is not a Snell cone geodesic.

In particular, when $ F_1^t = F_1, F_2^t = F_2 $ are time-independent, then $ \gamma(t) = (t,\sigma(t)) $ is a critical point of $ \T^* $ if and only if $ \dot{\sigma}(a) \perp_{F_1} P $, $ \sigma|_{[a,\tau]} $ and $ \sigma|_{[\tau,b]} $ are unit-speed geodesics of $ F_1 $, and \eqref{eq:reflection_natural} holds.
\end{cor}

\begin{rem}[Coordinates]
In coordinates, taking into account that $ \sigma $ is $ F_{\mu} $-unitary and $ g^{F_{\mu}}_{\tilde{v}}(\tilde{v},\tilde{z}) = F_{\mu}(\tilde{v}) \frac{\partial F_{\mu}}{\partial y^{\xi}}(\tilde{v})\tilde{z}^{\xi} $, with $ \xi = 1,\ldots,n $ (from \eqref{eq:fund_tensor_2}), Snell's law \eqref{eq:snell_natural} is equivalent to (see also \cite[Equation~(9)]{MP23})
\begin{equation*}
\left( \frac{\partial F_1^{\tau}}{\partial y^{\xi}}(\dot{\sigma}(\tau^-)) - \frac{\partial F_2^{\tau}}{\partial y^{\xi}}(\dot{\sigma}(\tau^+)) \right) \tilde{z}^{\xi} = 0, \quad \forall \tilde{z} \equiv (\tilde{z}^1,\ldots,\tilde{z}^n) \in T_{\sigma(\tau)}\tilde{\eta},
\end{equation*}
and the law of reflection \eqref{eq:reflection_natural} is equivalent to (see also \cite[Equation~(11)]{MP23})
\begin{equation*}
\left( \frac{\partial F_1^{\tau}}{\partial y^{\xi}}(\dot{\sigma}(\tau^-)) - \frac{\partial F_1^{\tau}}{\partial y^{\xi}}(\dot{\sigma}(\tau^+)) \right) \tilde{z}^{\xi} = 0, \quad \forall \tilde{z} \equiv (\tilde{z}^1,\ldots,\tilde{z}^n) \in T_{\sigma(\tau)}\tilde{\eta}.
\end{equation*}
\end{rem}

The following remark offers yet another interpretation of these laws, this time in terms of Finslerian angles. This perspective is more akin to the classical interpretation and, in fact, the subsequent example explicitly shows how these results reduce to the classical Snell's law and the law of reflection in the isotropic case (see also \cite[Examples~3.7 and 4.3]{MP23}).

\begin{rem}[Finslerian angles]
Fixing $ (t,x) \in Q $ and $ \tilde{v} \in T_xN $, we define the {\em $ F_{\mu}^t $-angle} between two vectors $ \tilde{u}, \tilde{w} \in T_xN $ in the direction $ \tilde{v} $ as the unique value $ \theta_{\tilde{v}}^{F_{\mu}^t}(\tilde{u},\tilde{w}) \in [0,\pi] $ satisfying
\begin{equation}
\label{eq:finsler_angle}
\cos{\left(\theta_{\tilde{v}}^{F_{\mu}^t}(\tilde{u},\tilde{w})\right)} = \frac{g^{F_{\mu}^t}_{\tilde{v}}(\tilde{u},\tilde{w})}{F_{\mu}^t(\tilde{u}) F_{\mu}^t(\tilde{w})}.
\end{equation}
In terms of these Finslerian angles, taking into account that $ \sigma $ is $ F_{\mu} $-unitary, Snell's law \eqref{eq:snell_natural} is equivalent to
\begin{equation}
\label{eq:snell_angles}
F_1^{\tau}(\tilde{z}) \cos{\left(\theta_{\dot{\sigma}(\tau^-)}^{F_1^\tau}(\dot{\sigma}(\tau^-),\tilde{z})\right)} = F_2^{\tau}(\tilde{z}) \cos{\left(\theta_{\dot{\sigma}(\tau^+)}^{F_2^\tau}(\dot{\sigma}(\tau^+),\tilde{z})\right)}, \quad \forall \tilde{z} \in T_{\sigma(\tau)}\tilde{\eta},
\end{equation}
and the law of reflection \eqref{eq:reflection_natural} is equivalent to
\begin{equation}
\label{eq:reflection_angles}
\cos{\left(\theta_{\dot{\sigma}(\tau^-)}^{F_1^\tau}(\dot{\sigma}(\tau^-),\tilde{z})\right)} = \cos{\left(\theta_{\dot{\sigma}(\tau^+)}^{F_1^\tau}(\dot{\sigma}(\tau^+),\tilde{z})\right)}, \quad \forall \tilde{z} \in T_{\sigma(\tau)}\tilde{\eta}.
\end{equation}

In particular, from \eqref{eq:finsler_angle} it is straightforward to check that Finslerian angles are conformally invariant. Therefore, if $ F_1^t, F_2^t $ are conformal at $ \gamma(\tau) = (\tau,\sigma(\tau)) \in \eta $, i.e. there exists $ \lambda > 0 $ such that $ F_2^\tau(\tilde{v}) = \lambda F_1^\tau(\tilde{v}) $ for all $ \tilde{v} \in T_{\sigma(\tau)}N $, then Snell's law \eqref{eq:snell_angles} reduces to
\begin{equation}
\label{eq:snell_conf}
\cos{\left(\theta_{\dot{\sigma}(\tau^-)}^{F}(\dot{\sigma}(\tau^-),\tilde{z})\right)} = \lambda \cos{\left(\theta_{\dot{\sigma}(\tau^+)}^{F}(\dot{\sigma}(\tau^+),\tilde{z})\right)}, \quad \forall \tilde{z} \in T_{\sigma(\tau)}\tilde{\eta},
\end{equation}
and the law of reflection \eqref{eq:reflection_angles} to
\begin{equation}
\label{eq:reflection_conf}
\cos{\left(\theta_{\dot{\sigma}(\tau^-)}^{F}(\dot{\sigma}(\tau^-),\tilde{z})\right)} = \cos{\left(\theta_{\dot{\sigma}(\tau^+)}^{F}(\dot{\sigma}(\tau^+),\tilde{z})\right)}, \quad \forall \tilde{z} \in T_{\sigma(\tau)}\tilde{\eta},
\end{equation}
where $ F $ is any Minkowski norm on $ T_{\sigma(\tau)}N $ conformal to $ F_1^\tau, F_2^\tau $.
\end{rem}

\begin{exe}[Classical laws]
\label{ex:classical_laws}
In $ Q = \mathds{R}^{n+1} $, let $ \varphi = (t=x^0,x^1,\ldots,x^n) $ be the natural Euclidean coordinates and assume that $ Q_1, Q_2 $ represent two different media where the propagation of light (or any other wave) is isotropic. Namely, the speed of light in $ Q_{\mu} $ is given by $ c/n_{\mu} $, where $ c $ is the speed of light in vacuum and $ n_{\mu}(t,x) $ is the refractive index of $ Q_{\mu} $, which may vary in space and time---but is independent of the direction. At each $ (t,x) \in Q_{\mu} $, the (spatial) velocity vectors of the wave are then given by the unit vectors of the time-dependent Finsler (in fact, Riemannian) metric
\begin{equation*}
F^t_{\mu} = \frac{n_{\mu}(t,x)}{c} ||\cdot||, \quad \mu = 1,2,
\end{equation*}
where $ ||\cdot|| $ denotes the natural Euclidean norm on $ \mathds{R}^n $. Using the notation introduced in \S~\ref{sec:existence}, let $ \gamma_u = (t,\sigma_u) $, $ \gamma_v = (t,\sigma_v) $ and $ \gamma_w = (t,\sigma_w) $ be an incident, refracted and reflected ($ t $-parametrized) trajectory, respectively, with its corresponding incident, refracted and reflected direction $ u = (1,\tilde{u}) $, $ v = (1,\tilde{v}) $ and $ w = (1,\tilde{w}) $ at some point $ p = (\tau,x) \in \eta = \mathds{R} \times \tilde{\eta} $. Since $ \sigma_r $ must be $ F^t_{\mu} $-unitary ($ \mu = 1 $ when $ r = u,w $; $ \mu = 2 $ when $ r = v $), we have the additional conditions
\begin{equation}
\label{eq:norms}
n_1(p)||\tilde{u}|| = n_2(p)||\tilde{v}||, \qquad ||\tilde{u}|| = ||\tilde{w}||.
\end{equation}
Note that $ F_2^t = \frac{n_2}{n_1} F_1^t $ and both Finsler metrics are conformal to the Euclidean norm, so Snell's law \eqref{eq:snell_conf} and the law of reflection \eqref{eq:reflection_conf} become, respectively,
\begin{equation}
\label{eq:laws_classic}
\begin{split}
& n_1(p) \cos{\left(\theta(\tilde{u},\tilde{z})\right)} = n_2(p) \cos{\left(\theta(\tilde{v},\tilde{z})\right)},\quad \forall \tilde{z} \in T_x\tilde{\eta},\\
& \cos{\left(\theta(\tilde{u},\tilde{z})\right)} = \cos{\left(\theta(\tilde{w},\tilde{z})\right)},\quad \forall \tilde{z} \in T_x\tilde{\eta},
\end{split}
\end{equation}
where $ \theta(\cdot,\cdot) $ refers to the usual Euclidean angle between two vectors. 
Now, let $ \{\tilde{z}_1,\ldots,\tilde{z}_n\} $ be a Euclidean orthonormal basis of $ T_xN $ such that $ \{\tilde{z}_i\}_{i=2}^n $ is a basis of $ T_x\tilde{\eta} $. If we write $ \tilde{r} \equiv (\tilde{r}^1,\ldots,\tilde{r}^n) $ in this basis, for $ \tilde{r} = \tilde{u},\tilde{v},\tilde{w} $, \eqref{eq:norms} and \eqref{eq:laws_classic} yield the following relations:
\begin{equation*}
\begin{split}
\tilde{v}^i & = \langle \tilde{v},\tilde{z}_i \rangle = ||\tilde{v}|| \cos{(\theta(\tilde{v},\tilde{z}_i))} = \frac{n_1(p)^2}{n_2(p)^2} ||\tilde{u}|| \cos{(\theta(\tilde{u},\tilde{z}_i))} = \frac{n_1(p)^2}{n_2(p)^2} \tilde{u}^i, \quad i=2,\ldots,n,\\
\tilde{w}^i & = \langle \tilde{w},\tilde{z}_i \rangle = ||\tilde{w}|| \cos{(\theta(\tilde{w},\tilde{z}_i))} = ||\tilde{u}|| \cos{(\theta(\tilde{u},\tilde{z}_i))} = \tilde{u}^i, \quad i=2,\ldots,n,
\end{split}
\end{equation*}
where $ \langle \cdot,\cdot \rangle $ denotes the natural Euclidean metric. Therefore:
\begin{equation*}
\begin{split}
\tilde{v} & = \tilde{v}^1 \tilde{z}_1 + \frac{n_1(p)^2}{n_2(p)^2}\sum_{i=2}^n \tilde{u}^i \tilde{z}_i = \tilde{v}^1 \tilde{z}_1 + \frac{n_1(p)^2}{n_2(p)^2} (\tilde{u} - \tilde{u}^1 \tilde{z}_1) = \left( \tilde{v}^1- \frac{n_1(p)^2}{n_2(p)^2} \tilde{u}^1 \right) \tilde{z}_1 + \frac{n_1(p)^2}{n_2(p)^2} \tilde{u},\\
\tilde{w} & = \tilde{w}^1 \tilde{z}_1 + \sum_{i=2}^n \tilde{u}^i \tilde{z}_i = \tilde{w}^1 \tilde{z}_1 + \tilde{u} - \tilde{u}^1 \tilde{z}_1 = (\tilde{w}^1-\tilde{u}^1)\tilde{z}_1 + \tilde{u},
\end{split}
\end{equation*}
which means that $ \tilde{v}, \tilde{w} $ can be written as a linear combination of $ \tilde{z}_1 $ and $ \tilde{u} $, i.e. the (spatial) refracted and reflected directions $ \tilde{v}, \tilde{w} $ belong to the 2-dimensional linear subspace generated by the incident direction $ \tilde{u} $ and the Euclidean normal $ \tilde{z}_1 $ to the interface $ \tilde{\eta} $ at $ p $.\footnote{This is a feature of the isotropic case, but it is not true in general due to the fact that the measurement of the Finslerian angles depends on the direction.}

Hence, without loss of generality, we can assume that $ n = 2 $ and $ \tilde{z}_i = \partial_{x^i} $, for $ i = 1,2 $ (see Figure~\ref{fig:classical}). In this case, instead of considering angles in $ [0,\pi] $ between two (oriented) vectors, it is more convenient---following the classical formulation---to work with angles in $ [0,\frac{\pi}{2}] $ between (non-oriented) lines: we define the {\em angle of incidence}, {\em refraction} and {\em reflection} $ \theta_r \in [0,\frac{\pi}{2}] $ (associated with $ r = u,v,w $, respectively) as the Euclidean angle between $ \{\lambda \tilde{r}: \lambda \in \mathds{R}\} $ and $ \{\lambda \tilde{z}_1: \lambda \in \mathds{R}\} $ (the normal line to the interface). Note that $ \{\tilde{z}_2\} $ is now a basis of $ T_x\tilde{\eta} $ and
\begin{equation*}
\cos{(\theta(\tilde{r},\tilde{z}_2))} = \pm \sin{\theta_r}, \quad r = u,v,w.
\end{equation*}
So, in the end, Snell's law and the law of reflection \eqref{eq:laws_classic} take the classical form:
\begin{equation}
\label{eq:classical}
n_1(p)\sin{\theta_u} = n_2(p)\sin{\theta_v}, \qquad \theta_u = \theta_w,
\end{equation}
where the previous $ \pm $ cancels out because $ \sin{\theta_r} > 0 $, and $ \sin{\theta_u} = \sin{\theta_w} $ implies $ \theta_u = \theta_w $ because $ \theta_r \in [0,\frac{\pi}{2}] $. It is worth noting that, since the refraction and reflection occur at a specific point $ p \in \eta $ (fixing both the position and instant of time), and these laws operate solely at the level of the tangent space $ T_pQ $, the classical laws are recovered even when the media are inhomogeneous and time-dependent.\footnote{For this to hold, it is essential that $ \eta = \mathds{R} \times \tilde{\eta} $, i.e. the interface must remain the same over time. Otherwise, its velocity must be taken into account, which modifies the classical laws (essentially, \eqref{eq:snell_natural} and \eqref{eq:reflection_natural} no longer hold and one must instead apply the general laws \eqref{eq:snell} and \eqref{eq:reflection}).}
\end{exe}

\begin{rem}[Classical total reflection]
\label{rem:clas_total_reflection}
The simplicity of \eqref{eq:classical} facilitates, for instance, the interpretation of the total reflection phenomenon described in \S~\ref{subsec:total_reflection}. According to \eqref{eq:classical}, the angle of refraction approaches its limit $ \theta_v \rightarrow \frac{\pi}{2} $ as $ \theta_u \rightarrow \widehat{\theta} $, where $ \widehat{\theta} \coloneqq \arcsin{(n_2(p)/n_1(p))} $ is known as the {\em critical angle}. This implies that such a limit can only be reached---in the isotropic case---when $ n_1(p) \geq n_2(p) $, i.e. when the speed of light in the second medium at $ p $ is greater than in the first one or, in terms of cone structures, when $ \C^1_p $ is entirely contained in $ \C^2_p $ (otherwise, $ \widehat{\theta} $ does not exist). When this condition is met, then for angles of incidence satisfying $ \theta_u > \widehat{\theta} $, the incident trajectory undergoes total reflection (see Figure~\ref{fig:classical}). At $ \theta_u = \widehat{\theta} $, $ \gamma_v $ travels along $ \eta $, yet it can still be regarded as a refracted trajectory since, in this particular case, it is in fact a global minimum of the arrival time functional (see \cite[Proposition~5.6]{MP23}).
\end{rem}

\begin{figure}
\centering
\begin{subfigure}{0.47\textwidth}
\includegraphics[width=\textwidth]{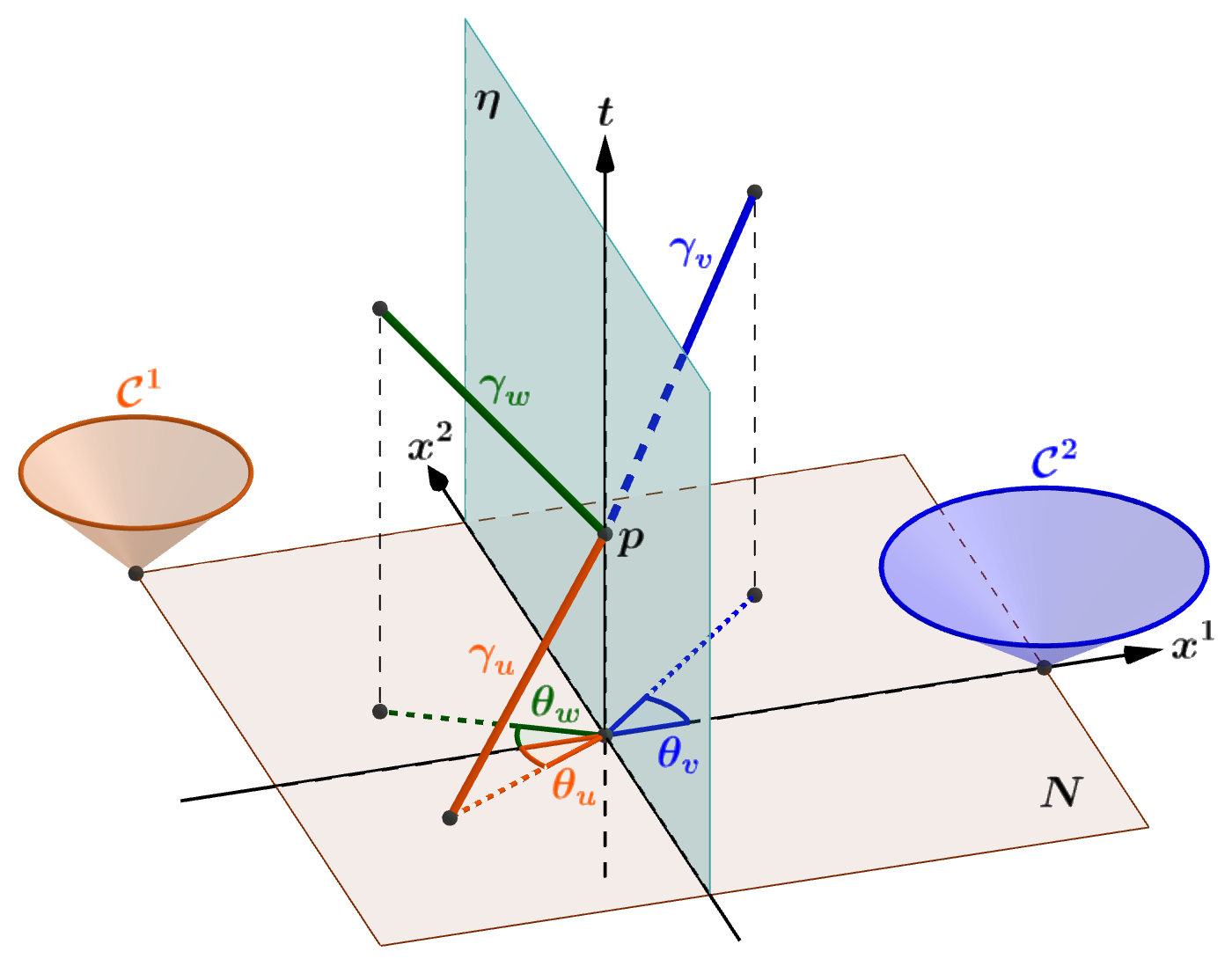}
\caption{$ \theta_u < \widehat{\theta} $.}
\end{subfigure}
\begin{subfigure}{0.37\textwidth}
\includegraphics[width=\textwidth]{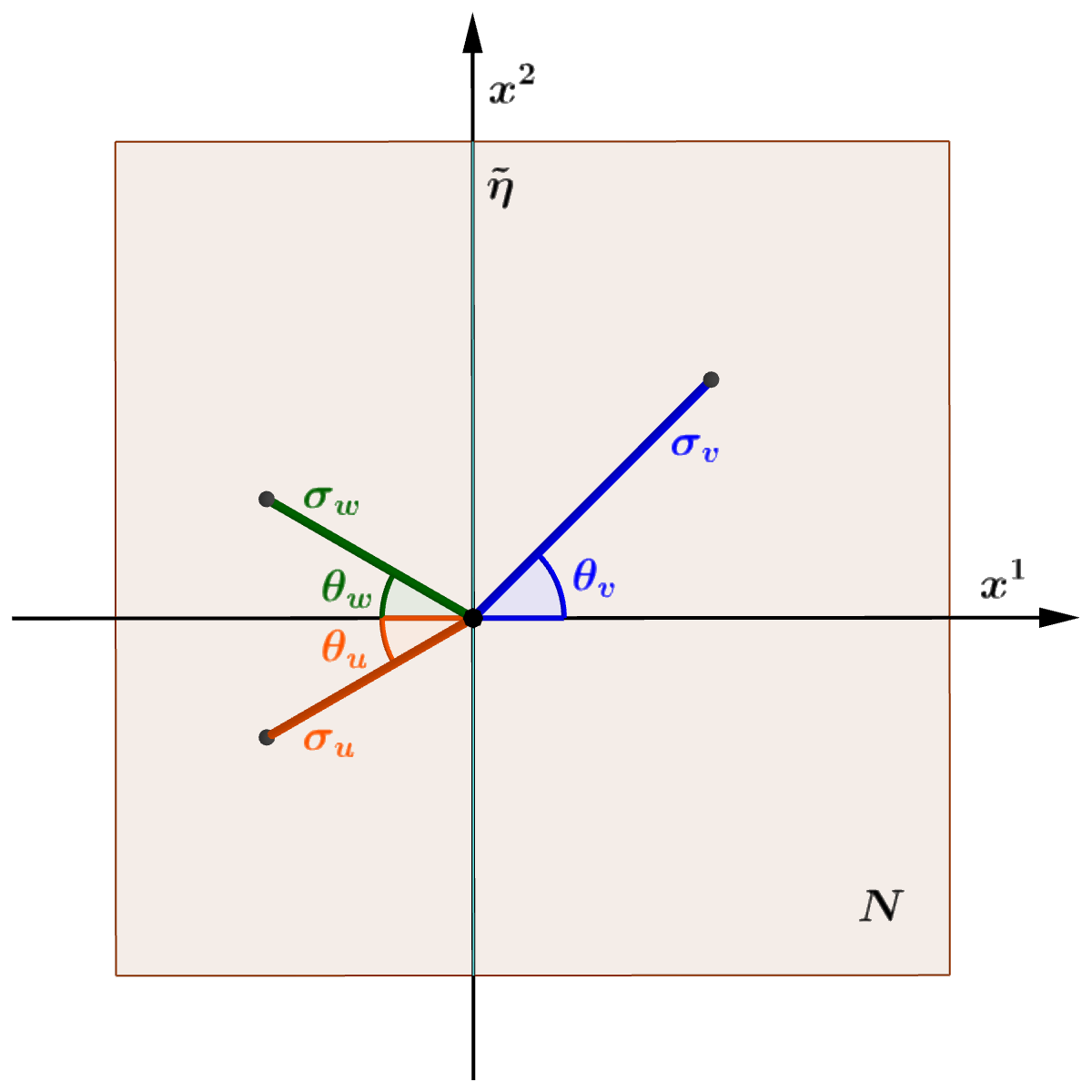}
\caption{$ \theta_u < \widehat{\theta} $. Projection on $ N $.}
\end{subfigure}
\begin{subfigure}{0.47\textwidth}
\includegraphics[width=\textwidth]{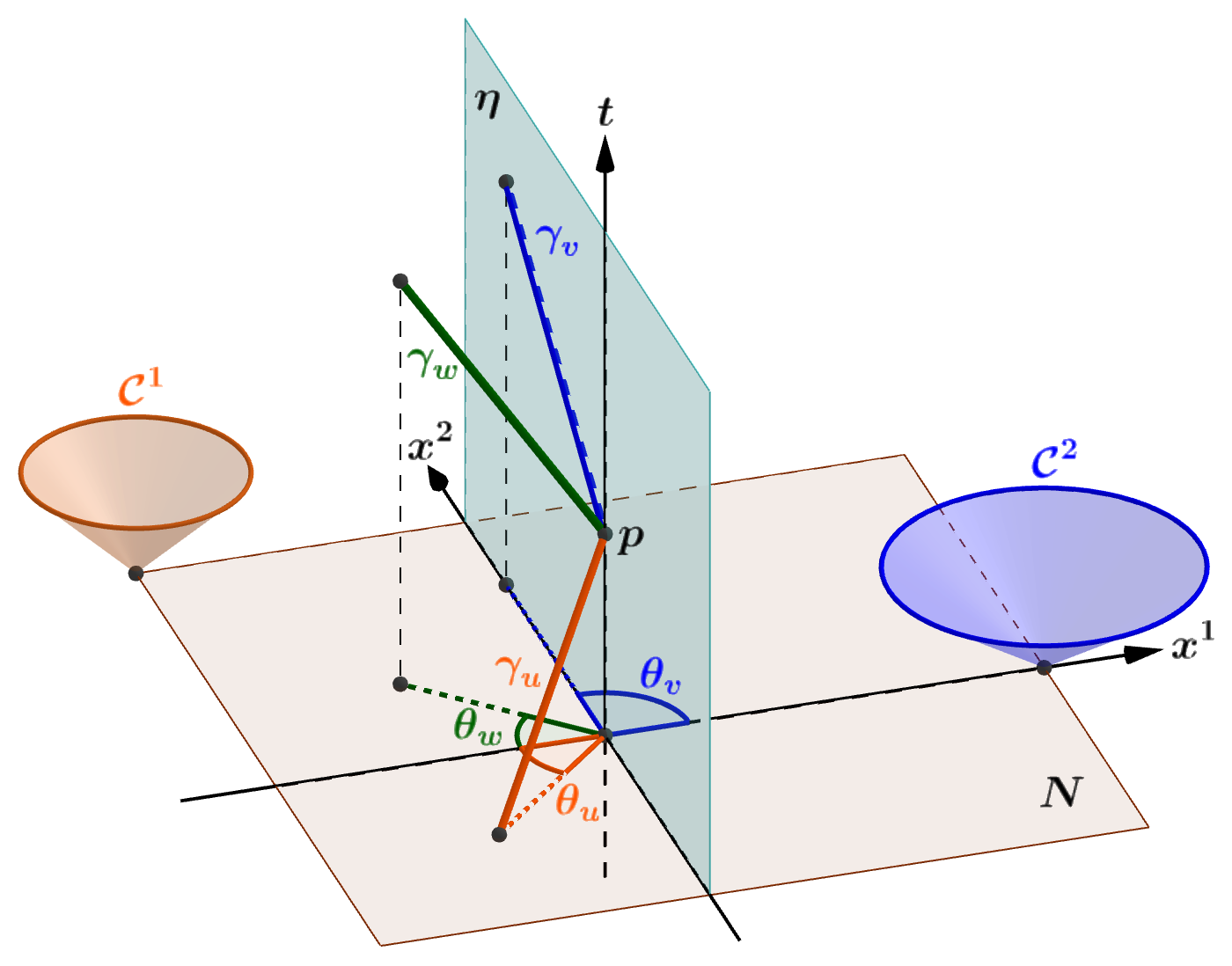}
\caption{$ \theta_u = \widehat{\theta} $.}
\end{subfigure}
\begin{subfigure}{0.37\textwidth}
\includegraphics[width=\textwidth]{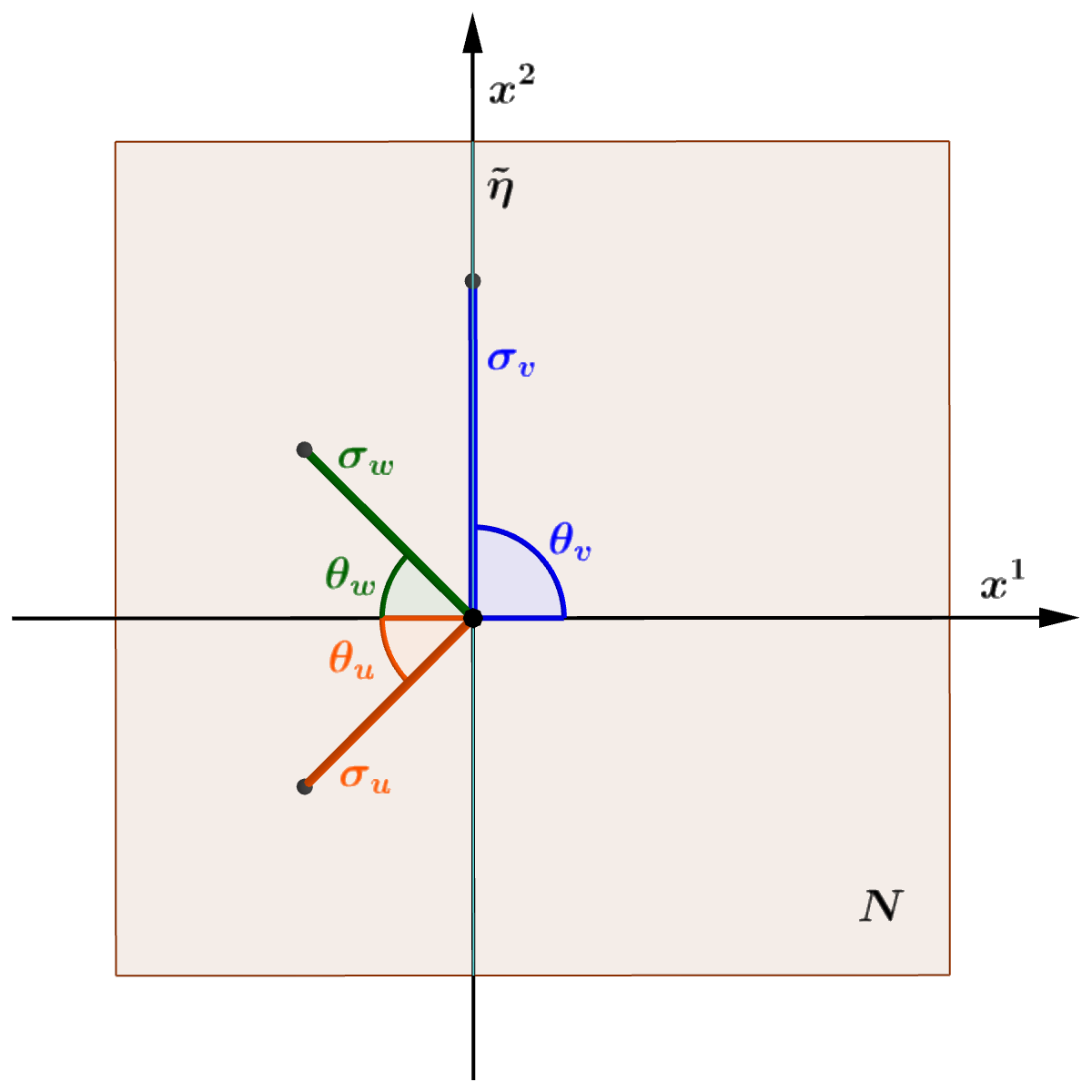}
\caption{$ \theta_u = \widehat{\theta} $. Projection on $ N $.}
\end{subfigure}
\begin{subfigure}{0.47\textwidth}
\includegraphics[width=\textwidth]{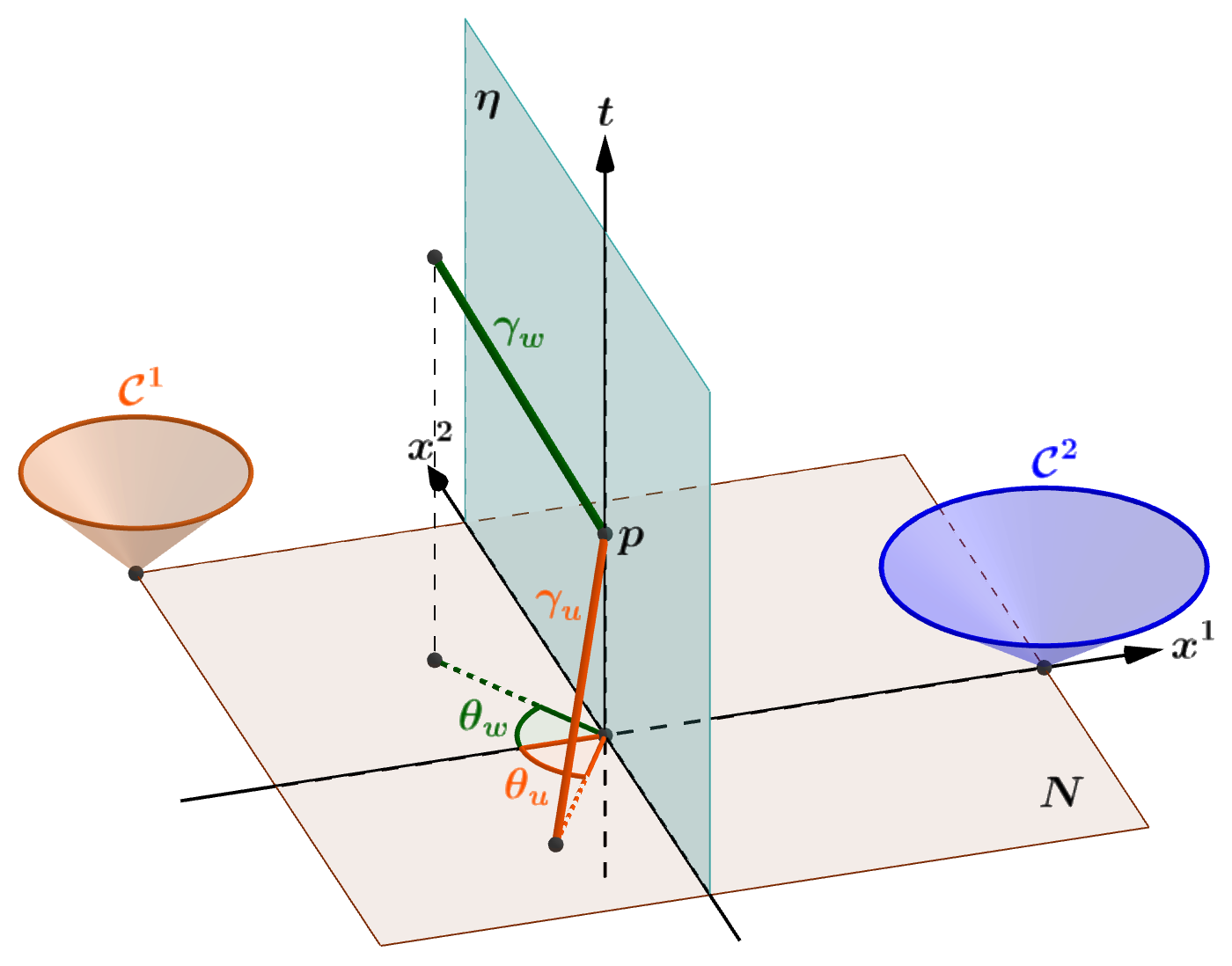}
\caption{$ \theta_u > \widehat{\theta} $.}
\end{subfigure}
\begin{subfigure}{0.37\textwidth}
\includegraphics[width=\textwidth]{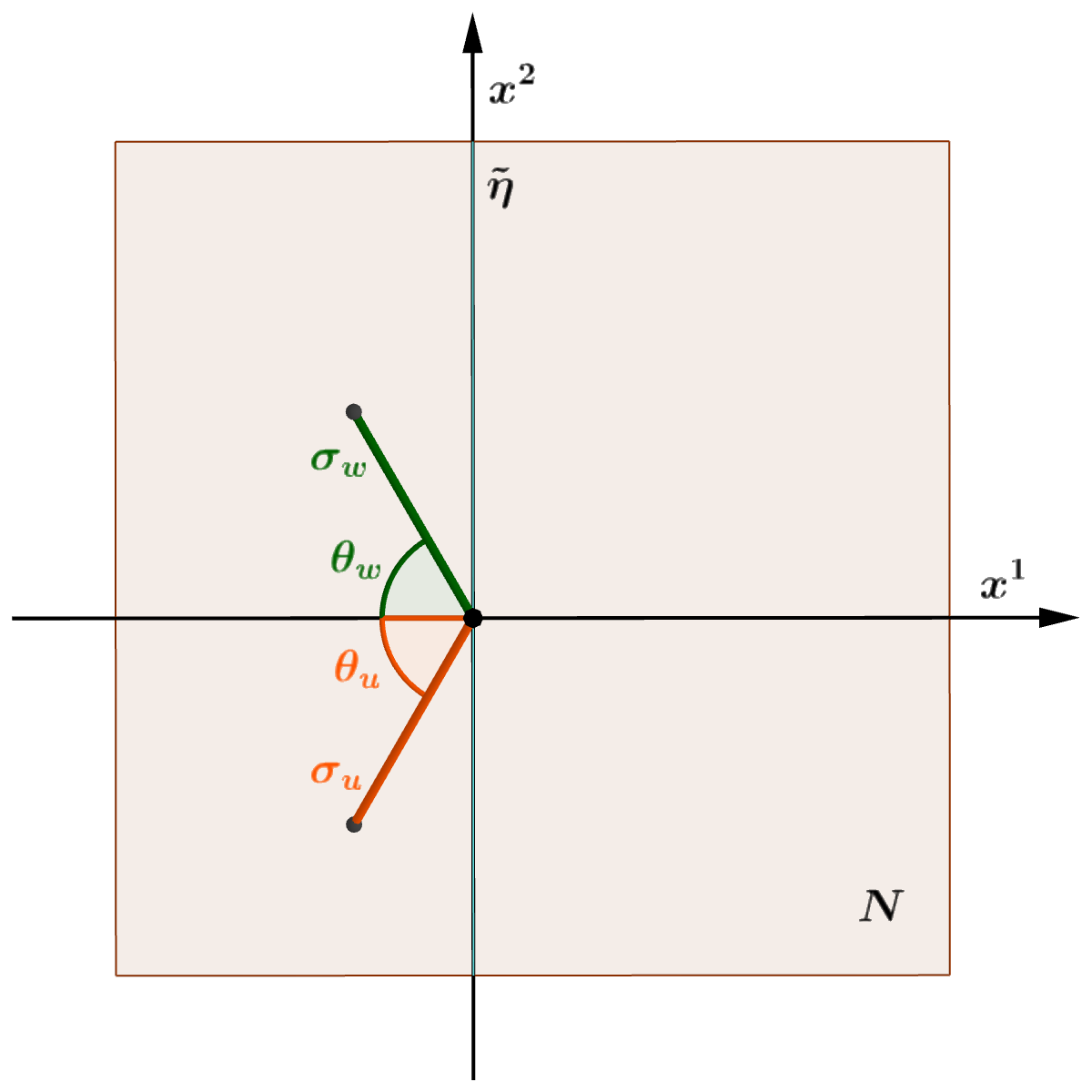}
\caption{$ \theta_u > \widehat{\theta} $. Projection on $ N $.}
\end{subfigure}
\caption{Illustration of the total reflection phenomenon in the isotropic, homogeneous and time-independent case (cone geodesics are straight lines) as $ \theta_u $ approaches and eventually exceeds the critical angle $ \widehat{\theta} $ (see Remark~\ref{rem:clas_total_reflection}). Observe that the refracted trajectory is faster than the incident and reflected ones, as $ n_1 > n_2 $ ($ \C^2 $ is wider than $ \C^1 $). In this case, the critical angle must appear because, as $ \theta_u $ increases, the line $ \Pi_u^1 \coloneqq u^{\perp_{\C^1}} \cap T_p\eta $ becomes tangent to $ \C^2_p $ (when $ v \in \Pi^1_u $ and $ \theta_u = \widehat{\theta} $) before it becomes tangent to $ \C^1_p $ (when $ u \in T_p\eta $). Between these two events, there is total reflection.}
\label{fig:classical}
\end{figure}

\subsection{Discretized spacetimes}
\label{sec:discretization}
As a final note on the applications, we briefly pose the following connection with the study of discretized spacetimes, to be explored in detail in a future work. Suppose we take a smooth cone structure $ \C $ on a smooth manifold $ Q $ and discretize it (in both space and time),\footnote{Formally, the discretization of a standard spacetime (with a specific Lorentzian metric) can be split into the discretization of the conformal structure and the conformal factor. The latter, however, is not relevant here.} i.e. we consider a grid covering $ Q $ where $ \C $ is assumed to be constant within each cell, and the cell boundaries are treated as interfaces (see also the physical setting in the review \cite{sanchez2026}). Cone geodesics in this discretized spacetime can be computed as straight lightlike curves within the cells, applying Snell's law of refraction \eqref{eq:snell} each time they cross an interface. It is expected that these discretized cone geodesics converge to the original smooth ones as the grid becomes finer, thereby recovering the causality of the original spacetime. For the computation of objects such as event horizons, only Snell cone geodesics should be considered, whose determination---particularly in the case of the double-refraction phenomenon---is described in \S~\ref{sec:cone_geodesics}. In future work, our procedure will be compared with more standard approaches in numerical relativity, such as ray tracing \cite{LehnerRT} (which is, in principle, less efficient), as well as with methods addressing more open problems, such as neutrino propagation \cite{LehnerN}.

As we have seen throughout this work, however, determining refracted trajectories via Snell's law is not straightforward. In particular, most of the challenges and difficulties we have encountered arise when the interface is $ \C^2 $-lightlike (or $ \C^1 $-lightlike in the case of reflection): in such cases, Snell's law may take the form of a strict inclusion rather than an equality (recall Remark~\ref{rem:directions}(5)). Nevertheless, this is not a practical concern here, since these cases are non-generic and the grid can always be slightly perturbed to ensure that they never occur.\footnote{The same solution can also be applied if the cone geodesic encounters the (non-smooth) corner of a cell.} With this adjustment, the grid induces only discontinuous jumps of the following type:
\begin{itemize}
\item Discontinuity along space ($ T_p\eta $ is $ \C^2 $-timelike): there is no ambiguity in the determination of the refracted trajectory (see Table~\ref{tab:refraction}).
\item Discontinuity along time ($ T_p\eta $ is $ \C^2 $-spacelike): the double refraction phenomenon appears (see \S~\ref{subsec:double_refraction}), but there is no ambiguity in the determination of the locally horismotic trajectory (Snell cone geodesic).
\end{itemize}
For reflected trajectories, the picture is even simpler (see Table~\ref{tab:reflection}), since a $ \C^1 $-spacelike interface does not allow for reflected directions.

\appendix

\section{The case where the receiver is a submanifold}
\label{app:receiver}
To our knowledge, there is only one (previously considered) generalization of the classical Fermat's principle that is not directly encompassed by the general framework presented in this work. This corresponds to the case in which the receiver $ B \subset Q_2 $ is a $ \C^2 $-timelike submanifold (rather than a curve $ \alpha $) endowed with a temporal function $ t_B $, as studied in \cite{PP98} within the Lorentzian setting. In this context, $ B $ represents a family of synchronized observers---e.g., moving along the directions of the integral curves of the gradient of $t_B$, which is conformally invariant in the Lorentzian case---distributed over a submanifold. In fact, in comparison with \cite{PP98}, two extensions can be introduced. First, Fermat's principle can be obtained without explicitly using the gradient of $t_B$ (which depends on the metric in Finsler spacetimes), thereby allowing the result to be fully extended to cone structures. Second, in line with the results presented throughout this work (where the receiver curve $\alpha$ can have any causal character), our framework admits a broader setting in which $B\subset Q_2$ is a submanifold of arbitrary causal character and $t_B$ is a submersion on $B$. In this general situation, the gradient of $ t_B $ need not even exist for any choice of Lorentz-Finsler metric $L$, since $L$ may be defined only on the causal cone.

To incorporate this case into our cone structure framework---i.e., keeping $ Q $, $ Q_{\mu} $, $ \C^{\mu} $, $ L_{\mu} $ ($ \mu = 1, 2 $) and $ \eta $ as defined at the beginning of Section~\ref{sec:setting}---we must first adapt the set of trial curves and the arrival time functional.

\begin{defi}
\label{def:app_arrival_time}
Let $ P $ be a submanifold of $ Q_1 $, $ B $ a submanifold of $ Q_2 $ and $ t_B: B \rightarrow \mathds{R} $ a submersion on $ B $,\footnote{This means that $ (dt_{B})_p \not= 0 $ for all $ p \in B $, i.e. $ dt_B $ is surjective everywhere.} with $ 0 \leq \textup{dim}(P) \leq n $ and $ 1 \leq \textup{dim}(B) \leq n $. Fixing an interval $ [a,b] $ and some $ \tau \in (a,b) $, let $ \mathcal{N}_{P,B} $ be the set of all (regular) piecewise smooth curves $ \gamma: [a,b] \rightarrow Q $ such that:
\begin{itemize}
\item $ \gamma(a) \in P $, $ \gamma(\tau) \in \eta $ and $ \gamma(b) \in B $.
\item $ \gamma $ is topologically transverse to $ \eta $ and lightlike.
\end{itemize}
We define the {\em arrival time functional} $ \T $ as
\begin{equation*}
\begin{array}{crcl}
\T \colon & \mathcal{N}_{P,B} & \longrightarrow & \mathds{R} \\
& \gamma & \longmapsto & \T[\gamma] \coloneqq t_B(\gamma(b)).
\end{array}
\end{equation*}
\end{defi}

Then, we can define an admissible variation $ \Lambda $ of $ \gamma \in \mathcal{N}_{P,B} $ and the associated variational vector field $ Z \in \mathfrak{X}(\gamma) $ as in Definition~\ref{def:variation}, simply replacing $ \mathcal{N}_{P,\alpha} $ with $ \mathcal{N}_{P,B} $ (i.e., we require that $ \Lambda(\omega,\cdot) \coloneqq \gamma_{\omega} \in \mathcal{N}_{P,B} $). With these new elements in place, Proposition~\ref{prop:variational_field}, Lemma~\ref{lem:Z(t_0)} and Theorem~\ref{thm:snell} can be reformulated accordingly as follows.

\begin{prop}
\label{prop:app}
Let $ \gamma \in \mathcal{N}_{P,B} $ and $ Z \in \mathfrak{X}(\gamma) $:
\begin{itemize}
\item[(I)] If $ Z $ is an admissible variational vector field, then $ Z(a) \in T_{\gamma(a)}P $, $ Z(\tau) \in T_{\gamma(\tau)}\eta $, \eqref{eq:lemma} holds and $ Z(b) \in T_{\gamma(b)}B $.
\item[(II)] The converse of (I) holds when $ \gamma $ is transverse to $ \eta $ and $ T_{\gamma(b)}B \not\subset \dot{\gamma}(b)^{\perp_{\C^2}} $.
\end{itemize}
\end{prop}
\begin{proof}
If $ Z \in \mathfrak{X}(\gamma) $ is the variational vector field associated with an admissible variation $ \Lambda(\omega,t) = \gamma_{\omega}(t) $, then $ \gamma_{\omega}(b) \in B $ for all $ \omega \in (-\varepsilon,\varepsilon) $, which directly implies that $ Z(b) \in T_{\gamma(b)}B $. The remaining properties follow exactly as in the proof of Proposition~\ref{prop:variational_field}(I).

To prove (II), we can also reproduce the proof of Proposition~\ref{prop:variational_field}(II), with the only modification that, now, we set the coordinate system $ \varphi_l $ so that $ B \equiv \{x^1_l=\ldots=x^{s}_l=0\} $ around $ \gamma(b) $, where $ s \coloneqq \textup{codim}(B) $. Hence, $ \partial_{x^0_l} $ is tangent to $ B $ around $ \gamma(b) $ and we can choose it so that $ \dot{\gamma} \not\perp_{\C^\mu} \partial_{x^0_l} $, thanks to hypothesis $ T_{\gamma(b)}B \not\subset \dot{\gamma}(b)^{\perp_{\C^2}} $. This way, $ \gamma^1_l(b)=\ldots=\gamma^s_l(b)=0 $ and $ Z^1_l(b)=\ldots=Z^s_l(b)=0 $ by hypothesis, so the lightlike curves of the final constructed variation $ \Lambda(\omega,t) = \gamma_{\omega}(t) $ satisfy $ \gamma_{\omega}^1(b)=\ldots=\gamma_{\omega}^s(b)=0 $, i.e. $ \gamma_{\omega}(b) \in B $. We can then conclude, following the original proof, that $ \Lambda $ is an admissible variation of $ \gamma $ whose associated variational vector field is precisely $ Z $.
\end{proof}

\begin{lemma}
\label{lem:app}
Let $ \gamma \in \mathcal{N}_{P,B} $. Then, $ \gamma $ is a critical point of $ \T $ if and only if $ Z(b) \in \textup{Ker}((dt_B)_{\gamma(b)}) $ for every admissible variational vector field $ Z \in \mathfrak{X}(\gamma) $.
\end{lemma}
\begin{proof}
Let $ Z \in \mathfrak{X}(\gamma) $ be any admissible variational vector field. Then
\begin{equation*}
\left. \frac{d}{d\omega} \T[\gamma_{\omega}] \right\rvert_{\omega=0} = \left. \frac{d}{d\omega} t_B(\gamma_{\omega}(b)) \right\rvert_{\omega=0} = (dt_B)_{\gamma(b)}(Z(b)),
\end{equation*}
so we conclude that
\begin{equation*}
\left. \frac{d}{d\omega} \T[\gamma_{\omega}] \right\rvert_{\omega=0} = 0 \Leftrightarrow Z(b) \in \textup{Ker}((dt_B)_{\gamma(b)}).
\end{equation*}
\end{proof}

\begin{thm}[Fermat's principle with a submanifold as receiver]
\label{thm:app}
Let $ \gamma \in \mathcal{N}_{P,B} $ such that $ T_{\gamma(b)}B \not\subset \dot{\gamma}(b)^{\perp_{\C^2}} $:
\begin{itemize}
\item[(I)] If $ \dot{\gamma}(a) \perp_{\C^1} P $,  $ \gamma|_{[a,\tau]} $ is a cone geodesic of $ \C^1 $, $ \gamma|_{[\tau,b]} $ is a cone geodesic of $ \C^2 $, Snell's law \eqref{eq:snell} holds and $ \textup{Ker}((dt_B)_{\gamma(b)}) \subset \dot{\gamma}(b)^{\perp_{\C^2}} $, then $ \gamma $ is a critical point of $ \T $.
\item[(II)] When $ \gamma $ is cone-transverse to $ \eta $, the converse of \textup{(I)} holds with equality in \eqref{eq:snell}.
\end{itemize}
\end{thm}
\begin{proof}
If $ \gamma $ satisfies the hypotheses in (I), then by following the same steps as in the proof of Theorem~\ref{thm:snell}(I), we deduce that $ Z(b) \in \dot{\gamma}(b)^{\perp_{\C^2}} $ for any admissible variational vector field $ Z $. Also, by Proposition~\ref{prop:app}(I) we know that $ Z(b) \in T_{\gamma(b)}B $. Now, observe that $ \dot{\gamma}(b)^{\perp_{\C^2}} $ is transverse to $ T_{\gamma(b)}B $ by hypothesis, so
\begin{equation*}
\begin{split}
\textup{dim}(\dot{\gamma}(b)^{\perp_{\C^2}} \cap T_{\gamma(b)}B) & = \textup{dim}(\dot{\gamma}(b)^{\perp_{\C^2}}) + \textup{dim}(B) - \textup{dim}(\dot{\gamma}(b)^{\perp_{\C^2}} + T_{\gamma(b)}B) = \\
& = n+\textup{dim}(B)-(n+1) = \textup{dim}(B)-1.
\end{split}
\end{equation*}
On the other hand, $ \textup{Ker}((dt_B)_{\gamma(b)}) \subset \dot{\gamma}(b)^{\perp_{\C^2}} $ by hypothesis, $ \textup{Ker}((dt_B)_{\gamma(b)}) \subset T_{\gamma(b)}B $ by definition (since $ t_B $ is only defined on $ B $) and $ \textup{dim}(\textup{Ker}((dt_B)_{\gamma(b)})) = \textup{dim}(B)-1 $ because $ t_B $ is a submersion, so we conclude that $ \textup{Ker}((dt_B)_{\gamma(b)}) = \dot{\gamma}(b)^{\perp_{\C^2}} \cap T_{\gamma(b)}B $. Therefore, $ Z(b) \in \textup{Ker}((dt_B)_{\gamma(b)}) $ for every admissible variational vector field $ Z $, which means, by Lemma~\ref{lem:app}, that $ \gamma $ is a critical point of $ \T $.

To prove (II), if $ \gamma $ is cone-transverse to $ \eta $ and a critical point of $ \T $, fix $ U \in \mathfrak{X}(\gamma) $ such that $ U(\tau) \in T_{\gamma(\tau)}\eta $, $ U(b) \in T_{\gamma(b)}B $ with $ U(b) \notin \textup{Ker}((dt_B)_{\gamma(b)}) $ and $ \dot{\gamma} \not\perp_{L_{\mu}} U $ everywhere, and define
\begin{equation*}
\mathcal{V} \coloneqq \{ W \in \mathfrak{X}(\gamma): W(a) \in T_{\gamma(a)}P, W(\tau) \in T_{\gamma(\tau)}\eta \text{ and } W(b) \in \textup{Ker}((dt_B)_{\gamma(b)}) \}.
\end{equation*}
Given any $ W \in \mathcal{V} $, we construct $ Z_W \in \mathfrak{X}(\gamma) $ exactly as in \eqref{eq:Z_W}, which satisfies all the hypotheses in Proposition~\ref{prop:app}(II), and $ \tilde{\mathcal{V}} $ as in \eqref{tildeV}. Applying Lemma~\ref{lem:app}, we know that $ Z_W(b) \in \textup{Ker}((dt_B)_{\gamma(b)}) $ for any $ W \in \tilde{\mathcal{V}} $, which means---due to the fact that $ U(b) \notin \textup{Ker}((dt_B)_{\gamma(b)}) $---that \eqref{eq:integrals} must hold. Obviously, choosing $ W $ so that $ W(b) = 0 $, we recover from the proof of Theorem~\ref{thm:snell}(II) all the required properties of $ \gamma $, except for the last one. Now, applying these properties but choosing an arbitrary $ W \in \tilde{\mathcal{V}} $, from \eqref{eq:integrals} (integrating by parts) we obtain the additional condition
\begin{equation*}
g_{\dot{\gamma}(b)}^{L_2}(\dot{\gamma}(b),W(b)) = 0, \quad \forall W \in \tilde{\mathcal{V}},
\end{equation*}
and since $ \textup{Ker}((dt_B)_{\gamma(b)}) = \{W(b): W \in \tilde{\mathcal{V}} \} $, we conclude that $ \textup{Ker}((dt_B)_{\gamma(b)}) \subset \dot{\gamma}(b)^{\perp_{\C^2}} $.
\end{proof}

\begin{rem}
Some final observations regarding the comparison between Theorems~\ref{thm:snell} and \ref{thm:app} are worth noting:
\begin{itemize}
\item The new conditions introduced in Theorem~\ref{thm:app} are $ T_{\gamma(b)}B \not\subset \dot{\gamma}(b)^{\perp_{\C^2}} $ (i.e., $ \dot{\gamma}(b) \not\perp_{\C^2} T_{\gamma(b)}B $) and $ \textup{Ker}((dt_B)_{\gamma(b)}) \subset \dot{\gamma}(b)^{\perp_{\C^2}} $ (i.e., $ \dot{\gamma}(b) \perp_{\C^2} \textup{Ker}((dt_B)_{\gamma(b)}) $). Note that $ \textup{Ker}((dt_B)_{\gamma(b)}) = T_{\gamma(b)}(t_B^{-1}(\T[\gamma])) $, where the inverse image $ t_B^{-1}(c) $ is (either empty or) a hypersurface of $ B $, for any $ c \in \mathds{R} $.
\item In particular, when $ \textup{dim}(B) = 1 $ and the curve $ B=\textup{Im}(\alpha) $ is parametrized by $ t_B $ (so that the arrival time functionals $ \T $ in Definitions~\ref{def:arrival_time} and \ref{def:app_arrival_time} coincide), then $ T_{\gamma(b)}B \not\subset \dot{\gamma}(b)^{\perp_{\C^2}} $ is equivalent to $ \dot{\alpha}(\T[\gamma]) \not\in \dot{\gamma}(b)^{\perp_{\C^2}} $ (i.e., $ \dot{\gamma}(b) \not\perp_{\C^2} \dot{\alpha}(\T[\gamma]) $), $ \textup{Ker}((dt_B)_{\gamma(b)}) \subset \dot{\gamma}(b)^{\perp_{\C^2}} $ holds automatically (because $ \textup{Ker}((dt_B)_{\gamma(b)}) = \{0\} $) and Theorem~\ref{thm:app} reduces to Theorem~\ref{thm:snell}.
\item When $ B $ is $ \C^2 $-timelike and $ t_B $ is a temporal function, then $ T_{\gamma(b)}B \not\subset \dot{\gamma}(b)^{\perp_{\C^2}} $ holds automatically and $ t_B^{-1}(c) $ is (either empty or) a $ \C^2 $-spacelike hypersurface of $ B $, for any $ c \in \mathds{R} $. In this case, Theorem~\ref{thm:app} reduces to \cite[Theorem~1]{PP98} in the (smooth) Lorentzian setting.
\item Although not stated explicitly in Theorem~\ref{thm:app}, the restrictions on $ P $ (given by $ \dot{\gamma}(a) \perp_{\C^1} P $) are the same as in Theorem~\ref{thm:snell}.
\end{itemize}
\end{rem}

\newpage

\section*{Acknowledgments}
The authors warmly acknowledge Professor Luis Lehner (Perimeter Institute for Theoretical Physics) for his explanations on numerical relativity, his comments on our techniques, and for pointing out the references \cite{LehnerN,LehnerRT}.

MAJ and EPR were supported in part by the Spanish Ministerio de Ciencia e Innovación and Agencia Estatal de Investigación MCIN/AEI/10.13039/501100011033/ under project PID2021-124157NB-I00, cofunded by the European Union (ERDF); in part by the Spanish Ministerio de Ciencia, Innovación y Universidades and Agencia Estatal de Investigación MICIU/AEI/10.13039/501100011033/ under project PID2025-167696NB-I00; and in part by Fundación Séneca-Agencia de Ciencia y Tecnología de la Región de Murcia under project ``Ayudas a proyectos para el desarrollo de investigación científica y técnica por grupos competitivos (Comunidad Autónoma de la Región de Murcia)'', included in ``Programa Regional de Fomento de la Investigación Científica y Técnica (Plan de Actuación 2022)'', REF. 21899/PI/22. EPR was also partially supported by the Spanish Ministerio de Ciencia e Innovación and Agencia Estatal de Investigación MCIN/AEI/10.13039/501100011033/ under project PID2020-116126GB-I00. MS was supported in part by the Spanish Ministerio de Ciencia, Innovación y Universidades and Agencia Estatal de Investigación MICIU/AEI/10.13039/501100011033/ under project PID2024-156031NB-I00, cofunded by the European Union (ERDF); and in part by the framework IMAG-Mar\'{i}a de Maeztu grant CEX2020-001105-M.

\end{document}